\documentclass[11pt]{amsart}

\usepackage{graphicx}
\usepackage{enumitem}
\usepackage{lipsum}
\usepackage{verbatim}
\usepackage{url}
\usepackage{amssymb}
\usepackage{amsthm}
\usepackage{amsmath}
\usepackage[all, cmtip]{xy}
\usepackage{tikz}
\usepackage{dsfont}
\usepackage{spectralsequences}
\usepackage{mathtools}
\usetikzlibrary{matrix}
\usetikzlibrary{cd}
\usepackage[center]{caption}
\usepackage{stmaryrd}
\graphicspath{{pics/}}
\usepackage{scalerel}[2016/12/29]
\usepackage{enumitem}
\usepackage{bbm}
\usepackage{multicol}
\usepackage{bm}

\usepackage{thmtools}
\usepackage{lipsum}

\declaretheoremstyle[
spaceabove=6pt, spacebelow=6pt,
headfont=\normalfont\itshape,
headpunct=.,
notefont=\normalfont\bfseries, notebraces={\hspace*{-4pt}}{},
bodyfont=\normalfont\itshape,
postheadspace=0.5em,
numbered=no,
]{mystyle}
\declaretheorem[style=mystyle,name={}]{labelenv}

\calclayout

\newlength\mylen
\newlist{mycases}{enumerate}{1}
\setlist[mycases,1]{label=\textbf{Case~\arabic*.}, 
  labelwidth=\dimexpr-\mylen-\labelsep\relax,leftmargin=0pt,align=right}

\newtheorem{question}{Question}
\newtheorem{theorem}{Theorem}

\newtheorem{lemma}{Lemma}[section]
\newtheorem{prop}[lemma]{Proposition}

\theoremstyle{definition}

\theoremstyle{remark}

\numberwithin{equation}{section}

\newcommand{\Q}{\mathbb Q}
\newcommand{\Z}{\mathbb Z}

\newcommand{\calE}{\mathcal E}
\newcommand{\cI}{\mathcal I}

\newcommand{\calR}{\mathcal R}
\newcommand{\p}[1]{\medskip \noindent \emph{#1}.}
\newcommand{\definition}[1]{\medskip \noindent \emph{#1.}}
\newcommand{\cB}{\mathcal B}
\newcommand{\cC}{\mathcal C}

\newcommand{\calU}{\mathcal U}

\newcommand{\bE}{\mathbb E}

\newcommand{\calF}{\mathcal F}

\newcommand{\calH}{\mathcal H}
\newcommand{\calI}{\mathcal I}

\newcommand{\Tau}{\mathcal T}
\renewcommand{\mod}{\mo}

\newcommand{\bn}{\noindent}

\newcommand{\calB}{\mathcal{B}}

\newcommand{\cut}{{\ssearrow}}

\newcommand{\finalbound}{51}

\newcommand{\calS}{\mathcal{S}}
\newcommand{\calV}{\mathcal{V}}
\newcommand{\midwedge}{\scaleobj{1.5}{\wedge}}
\newcommand{\rrightarrow}{\twoheadrightarrow}
\newcommand{\calD}{\mathcal{D}}
\newcommand{\calW}{\mathcal{W}}
\newcommand{\calK}{\mathcal{K}}

\let\oldunderset\underset
\protected\def\underset{\oldunderset}

\DeclareMathOperator{\Ind}{Ind}

\DeclareMathOperator{\Mod}{Mod}
\DeclareMathOperator{\Stab}{Stab}
\DeclareMathOperator{\Sp}{Sp}

\DeclareMathOperator{\Aut}{Aut}

\DeclareMathOperator{\bp}{bp}

\DeclareMathOperator{\im}{im}

\DeclareMathOperator{\ab}{ab}

\DeclareMathOperator{\Cech}{\check{C}ech}

\DeclareMathOperator{\ho}{Hom}

\DeclareMathOperator{\mo}{mod}

\DeclareMathOperator{\cok}{cok}
\DeclareMathOperator{\Span}{Span}

\DeclareMathOperator{\Diff}{Diff}

\DeclareMathOperator{\BM}{BM}

\DeclareMathOperator{\maxrk}{maxrk}
\DeclareMathOperator{\rk}{rk}
\DeclareMathOperator{\proj}{proj}

\DeclareMathOperator{\degen}{degen}

\DeclareMathOperator{\eval}{eval}
\DeclareMathOperator{\GL}{GL}
\DeclareMathOperator{\rkshrink}{rkshrink}
\DeclareMathOperator{\algshrink}{algshrink}
\DeclareMathOperator{\nummaxrk}{nummaxrk}
\DeclareMathOperator{\maxalg}{maxalg}
\DeclareMathOperator{\nummaxalg}{nummaxalg}

\begin{document}

\title[Second homology of Torelli]{THE SECOND RATIONAL HOMOLOGY OF THE TORELLI GROUP IS FINITELY GENERATED}

\author{Daniel Minahan}
\address{Skiles Classroom Building, Georgia Institute of Technology, Atlanta, GA 30332}
\email{dminahan\finalbound@gatech.edu}

\subjclass[2000]{Primary 54C40, 14E20; Secondary 49E25, 20C20}

\date{\today}

\keywords{Torelli group}

\begin{abstract}
We prove that second rational homology of the Torelli group of an orientable closed surface of genus $g$ is finite dimensional for $g \geq 51$.  This rules out the simplest obstruction to the Torelli group being finitely presented and provides a partial answer to a question of Bestvina. 
\end{abstract}

\vspace*{-1.5cm}

\maketitle

\p{Historical Remark} This is a permanent preprint.  The results of this paper have been subsumed and significantly improved upon by the work of the author and Putman \cite{MinahanPutmanH2}.

\section{Introduction}\label{introsection}

Let $S_g^b$ be a compact, orientable surface of genus $g$ with $b$ boundary components.  The \textit{mapping class group} $\Mod(S_g)$ is $\pi_0(\Diff^+(S_g))$.  The action of $\Mod(S_g)$ on $H_1(S_g;\Z)$ induces a representation $\Mod(S_g) \rightarrow \Sp(2g,\Z)$ called the \textit{symplectic representation}, where $\Sp(2g,\Z)$ is the group of invertible linear transformations of $H_1(S_g;\Z)$ that respect the algebraic intersection form $\langle \cdot, \cdot \rangle$.  The kernel of this representation is called the \textit{Torelli group} and is denoted $\cI_g$.  There is a short exact sequence 
\begin{displaymath}
1 \rightarrow \cI_g \rightarrow \Mod(S_g) \rightarrow \Sp(2g,\Z) \rightarrow 1.
\end{displaymath}
\bn By the work of McCullough--Miller~\cite{McCulloughMiller}, Mess~\cite{Messfree} and Johnson~\cite{JohnsonI}, the Torelli group is finitely generated if and only if $g \neq 2$.  Birman posed the following question~\cite[Problem 29]{Birmanbook}:

\begin{question}\label{Birmanquestion}
  Is $\cI_g$ finitely presented for any sufficiently large $g$?
\end{question}

\bn Similar questions were asked by Mess~\cite[Page 90]{Kirbyproblems} and Morita~\cite[Problem 2.1]{Moritasurveyprospect}.  

\p{Obstructions to finite presentability} If $G$ is a finitely generated group with $H_2(G;\Q)$ infinite dimensional, then $G$ is not finitely presentable.  The main result of the paper is the following theorem, which shows that $H_2(\cI_g;\Q)$ does not obstruct finite presentability of $\cI_g$ for $g \gg 0$.

\begin{theorem}\label{mainthm}
Let $g \geq \finalbound$.  The vector space $H_2(\cI_g;\Q)$ is finite dimensional.
\end{theorem}

\bn Theorem~\ref{mainthm} partially answers the following question of Bestvina~\cite[Page 5]{Farbproblems}.  Margalit also asked a similar question~\cite[Question 5.12]{Margalitproblems}).

\begin{question}\label{homolquestion}
For which choices of $k, g$, and commutative ring $R$ is $H_k(\cI_g;R)$ finitely generated?
\end{question}

\definition{Prior partial answers to Question~\ref{homolquestion}} Mess showed that $\cI_2$ is an infinitely generated free group~\cite{Messfree}, which implies that $H_1(\cI_2)$ is infinitely generated and $H_i(\cI_2) = 0$ for $i \geq 2$.  Johnson~\cite{JohnsonIII} computed $H_1(\cI_g;\Z)$ for $g \geq 3$ and proved that
\begin{displaymath}
H_1(\cI_g;\Q) \cong \midwedge^3 H_1(S_g;\Q)/H_1(S_g;\Q).
\end{displaymath}

\bn In fact, Johnson defined a map $\tau_g:\cI_g \rightarrow \midwedge^3 H_1(S_g;\Z)/H_1(S_g;\Z)$ now called the Johnson homomorphism, and showed that $(\tau_g)_*$ induces an isomorphism in first rational homology~\cite{JohnsonIII}.  Akita showed that $H_*(\cI_g;\Z)$ is infinitely generated as an abelian group for $g \geq 7$~\cite{Akita}.  Hain showed that $H_3(\cI_3;\Z)$ is infinitely generated~\cite{Hainthird}.  Bestvina, Bux, and Margalit showed that the cohomological dimension of $\cI_g$ is $3g-5$~\cite[Theorem A]{BBM}, which implies that $H_k(\cI_g;\Z) = 0$ for $k \geq 3g - 4$.  In this same paper, they proved that $H_{3g-5}(\cI_g;\Z)$ is infinitely generated for $g \geq 2$~\cite[Theorem C]{BBM}.  Gaifullin later showed that $H_k(\cI_g;\Z)$ is infinitely generated for $2g - 3\leq k \leq 3g - 5$~\cite{Gaifullin}, confirming a conjecture of Bestvina--Bux--Margalit~\cite{BBM}. 

\subsection{The outline of the proof of Theorem~\ref{mainthm}}\label{theoremBtheoremAsection}

Theorem~\ref{mainthm} will follow from the following result.

\begin{theorem}\label{fincokthm}
Let $g \geq \finalbound$.  Let $a \subseteq S_g$ be a nonseparating simple closed curve.  The cokernel of the pushforward map $\iota_*:H_2(\Stab_{\cI_g}(a);\Q) \rightarrow H_2(\cI_{g};\Q)$ is finite dimensional.
\end{theorem}

\bn Most of the paper will be devoted to proving Theorem~\ref{fincokthm}.  

\p{The connection between Theorem~\ref{mainthm} and Theorem~\ref{fincokthm}}  What follows is a trick.  We will expand this trick in Proposition \ref{gengrpprop}.  Let $\calB$ be a finite collection of nonseparating simple closed curves such that the set of Dehn twists
\begin{displaymath}
\{T_c: c \in \calB\}
\end{displaymath}
\bn generates $\Mod(S_g)$, e.g., the curves corresponding to the Humphries generators~\cite{FarbMarg}.  Consider the map
\begin{displaymath}
\varphi: H_2(\cI_g;\Q) \rightarrow \bigoplus_{c \in \calB} \cok(H_2(\Stab_{\cI_g}(c);\Q) \rightarrow H_2(\cI_g;\Q)).
\end{displaymath}
\bn  Now, for any $c \in \calB$, the quotient map $\rho_c:H_2(\cI_g;\Q) \rightarrow \cok(H_2(\Stab_{\cI_g}(c);\Q) \rightarrow H_2(\cI_g;\Q))$ satisfies 

\begin{displaymath}
\ker(\rho_c) = \im(H_2(\Stab_{\cI_g}(c);\Q) \rightarrow H_2(\cI_g;\Q)).
\end{displaymath}
\bn  Hence we have
\begin{displaymath}
    \ker(\varphi) = \bigcap_{c \in \calB} \im(H_2(\Stab_{\cI_g}(c);\Q) \rightarrow H_2(\cI_g;\Q)).
\end{displaymath}
\bn Then for any $c \in \calB$, $T_c$ acts trivially $\im(H_2(\Stab_{\cI_g}(c);\Q) \rightarrow H_2(\cI_g;\Q))$.  Therefore every element of $\{T_c:c \in \calB\}$ acts trivially on $\ker(\varphi)$.  Since $\{T_c: c \in \calB\}$ generates $\Mod(S_g)$, the vector space $\ker(\varphi)$ is a trivial $\Sp(2g,\Z)$--module.  We will show in Lemma~\ref{KassPutaltlemma} that $H_0(\Sp(2g,\Z);H_2(\cI_g;\Q))$ is finite dimensional.  This is a result of Kassabov and Putman~\cite{KassabovPutman}, for which we provide an alternative proof.  Theorem~\ref{fincokthm} says that each vector space $\cok(H_2(\Stab_{\cI_g}(c);\Q) \rightarrow H_2(\cI_g;\Q))$ is finite dimensional.  Since $\calB$ is a finite set, the vector space $\im(\varphi)$ is finite dimensional.  Consider the short exact sequence
\begin{displaymath}
0 \rightarrow \ker(\varphi) \rightarrow H_2(\cI_g;\Q) \rightarrow \im(\varphi) \rightarrow 0.
\end{displaymath}
\bn  A portion of the corresponding long exact sequence in $\Sp(2g,\Z)$--homology is given by
\begin{displaymath}
\ldots \rightarrow H_1(\Sp(2g,\Z);\im(\varphi)) \rightarrow H_0(\Sp(2g,\Z); \ker(\varphi)) \rightarrow H_0(\Sp(2g,\Z); H_2(\cI_g;\Q)) \rightarrow \ldots
\end{displaymath}
\bn Since $\Sp(2g,\Z)$ is finitely generated and $\im(\varphi)$ is finite dimensional, $H_1(\Sp(2g,\Z);\im(\varphi))$ is finite dimensional as well.  By Lemma~\ref{KassPutaltlemma}, $H_0(\Sp(2g,\Z); H_2(\cI_g;\Q))$ is finite dimensional.  Hence the vector space  $H_0(\Sp(2g,\Z), \ker(\varphi))$ is finite dimensional.  Then $H_0(\Sp(2g,\Z); \ker(\varphi))$ is isomorphic to $\ker(\varphi)$ because $\ker(\varphi)$ is a trivial $\Sp(2g,\Z)$--module, so $\ker(\varphi)$ is finite dimensional.  Since $\im(\varphi)$ is finite dimensional as well, we have $H_2(\cI_g;\Q)$ finite dimensional.

\subsection{The proof of Theorem~\ref{fincokthm}}

The proof of Theorem~\ref{fincokthm} proceeds by considering the action of $\cI_g$ on a complex called the complex of homologous curves, and then studying the equivariant homology spectral sequence for this group action.

\p{The complex of homologous curves}  In this paper, a \textit{curve} on a surface $S$ is an isotopy class of oriented essential embedded copies of $S^1$.  Two curves $c$ and $d$ are \textit{disjoint} if they have disjoint representatives.  The \textit{geometric intersection number}, denoted $|c \cap d|$, is the minimum number of intersection points over all representatives of $c$ and $d$.  Let $a \subseteq S_g$ be a nonseparating curve and let $[a] \in H_1(S_g;\Z)$ denote the corresponding first homology class.  The \textit{complex of homologous curves} $\cC_{[a]}(S_g)$ is the complex where the vertices are curves $c$ such that $c$ represents $[a]$ in $H_1(S_g;\Z)$.  A $k$--cell in $\cC_{[a]}(S_g)$ is a collection of $k+1$ pairwise disjoint vertices of $\cC_{[a]}(S_g)$.  A theorem of the author says that $\widetilde{H}_k(C_{[a]}(S_g);\Z) = 0$ for $0 \leq k \leq g - 3$~\cite[Theorem A]{Minahanhomolconn}.  

\p{The equivariant homology spectral sequences}  We consider the equivariant homology spectral sequence $\bE_{*,*}^*$ for the action of $\cI_g$ on $\cC_{[a]}(S_g)$~\cite[Section VII]{Brownbook}.  Let $a \subseteq S_g$ be a nonseparating curve and let $X_g = \cC_{[a]}(S_g)/\cI_g$.  We will use the notation $X_g$ throughout the paper to refer to this quotient.  If $\sigma$ is a cell of $X_g$, let $(\cI_g)_{\sigma}$ denote the stabilizer $\Stab_{\cI_g}(\widehat{\sigma})$ for some arbitrary lift of $\sigma$ to $\cC_{[a]}(S_g)$.  It is known that if $x$ and $y$ are nonseparating homologous curves in $S_g$, then there is an $f \in \cI_g$ such that $fx = y$ (see, e.g., \cite[Lemma 6.2]{Putmantorelli}).  Hence the set $X_g^{(0)}$ of vertices is a singleton.  It follows that page 1 of $\bE_{*,*}^*$ is as in Figure~\ref{bcloutlinepg1simpleex}.

\begin{sseqdata}[name = homologous bcl simple, homological Serre grading, y axis gap = 1.5cm, y axis gap = 1.5cm , x axis extend end = 2cm, classes = {draw = none}, xscale = 3.2, yscale = 1]
\class["H_2((\cI_g)_a)"](0,2)
\class["H_1((\cI_g)_a)"](0,1)
\class["\underset{e \in X^{(1)}}{\bigoplus}H_2((\cI_g)_e)"](1,2)
\class["\underset{e \in X^{(1)}}{\bigoplus}H_1((\cI_g)_e)"](1,1)
\class["\underset{\sigma \in X^{(2)}}{\bigoplus}H_1((\cI_g)_\sigma)"](2,1)
\class["\underset{e\in X^{(1)}}{\bigoplus}H_0((\cI_g)_e)"](1,0)
\class["\underset{\sigma \in X^{(2)}}{\bigoplus}H_0((\cI_g)_\sigma)"](2,0)
\class["\underset{\rho \in X^{(3)}}{\bigoplus}H_0((\cI_g)_\rho)"](3,0)
\d[source anchor = 180, target anchor = 0, shorten < = 0.5]1(1,2)
\d[source anchor = 180, target anchor = 0, shorten < = 0.5]1(1,1)
\d[source anchor = 174, target anchor = 6, shorten < = 0.5]1(2,1)
\d[source anchor = 174, target anchor = 6, shorten < = 0.5]1(3,0)
\d[source anchor = 174, target anchor = 6, shorten < = 0.5]1(2,0)
\end{sseqdata}

\begin{figure}[ht]
\printpage[name = homologous bcl simple,page = 1]
\caption{Page 1 of $\bE_{*,*}^*$ for the action of $\cI_g$ on $\cC_{[a]}(S_g)$.  All coefficients are in $\Q$.}\label{bcloutlinepg1simpleex}
\end{figure}

Since $\cC_{[a]}(S_g)$ is 2--acyclic for $g \geq 5$~\cite[Theorem A]{Minahanhomolconn}, the spectral sequence $\bE_{*,*}^*$ converges to $H_2(\cI_g)$~\cite[Section VII]{Brownbook}.  Examining Figure~\ref{bcloutlinepg1simpleex}, we see that the cokernel of the map
\begin{displaymath}
\iota_*:H_2(\Stab_{\cI_g}(a);\Q) \rightarrow H_2(\cI_g;\Q)
\end{displaymath}
\bn is noncanonically identified with the direct sum
\begin{displaymath}
\bE_{1,1}^\infty \oplus \bE_{2,0}^\infty.
\end{displaymath}
\bn  Hence the cokernel of $\iota_*$ is noncanonically identified with a subquotient of
\begin{displaymath}
\bE_{1,1}^2 \oplus \bE_{2,0}^2.
\end{displaymath} 
\bn The proof of Theorem~\ref{fincokthm} will proceed in two steps:

\begin{enumerate}[label=\Roman{enumi}.]
\item the vector space $\bE_{2,0}^2$ is finite dimensional, and
\item the vector space $\bE_{1,1}^2$ is finite dimensional.
\end{enumerate}
\bn By definition, $\bE_{2,0}^2$ is canonically identified with $H_2(X_g;\Q)$.  Step I is recorded as the following proposition.
\begin{prop}\label{homolcurvequotprop}
Let $g \geq \finalbound$ and $a \subseteq S_g$ be a nonseparating simple closed curve.  Let $X_g = \cC_{[a]}(S_g)/\cI_g$.  The vector space $H_2(X_g;\Q)$ is finite dimensional.
\end{prop}
\bn Proposition~\ref{homolcurvequotprop} will be proven in Sections~\ref{homolcurvequotsection}--\ref{finquotsectionpt2}.  Step II of the proof of Theorem~\ref{fincokthm} is recorded as the following result.

\begin{prop}\label{11finprop}
Let $g \geq \finalbound$ and let $a \subseteq S_g$ be a nonseparating simple closed curve.  Let $\bE_{*,*}^*$ denote the equivariant homology spectral sequence in rational coefficients for the action of $\cI_g$ on $\cC_{[a]}(S_g)$.  The vector space $\bE_{1,1}^2$ is finite dimensional.
\end{prop}

\bn Proposition~\ref{11finprop} will be proven in Sections~\ref{11finquotsection}--\ref{fincoksection}.  Before discussing the proofs of Propositions~\ref{homolcurvequotprop} and~\ref{11finprop}, we will discuss the main technical tool of the paper.

\subsection{The main technical tool of the paper} We will now explain the main technical tool that we use to prove Theorem \ref{mainthm}, Proposition \ref{homolcurvequotprop}, and Proposition \ref{11finprop}.  A specific instance of this tool has already been discussed in Section \ref{theoremBtheoremAsection}.

\p{Notation} Throughout this paper, $\Stab$ will be taken to mean the pointwise stabilizer, as opposed to the setwise stabilizer.

\begin{labelenv}[Proposition \ref{gengrpprop}] Let $g \geq 1$, and let $N \subseteq S_g$ be a nonseparating multicurve such that $|N| < g$.  Let $G \subseteq \Sp(2g,\Z)$ be the image of the map $\Stab_{\Mod(S_g)}(N) \rightarrow \Sp(2g,\Z)$.  Let $V$ be a $G$--representation over $\Q$.  Suppose that there is a constant $0 \leq d \leq g- |N|$ such that the following hold:

\begin{enumerate}
\item For any $M \subseteq S_g$ such that:
\begin{itemize}
\item $|M| \geq d$,
\item $M$ is disjoint from $N$, and
\item $M \sqcup N$ is nonseparating,
\end{itemize}
\bn the cokernel of the map $\bigoplus_{c \in M} V^{T_c} \rightarrow V$ is finite dimensional.
\item For any $M \subseteq S_g$ such that:
\begin{itemize}
\item $|M| < d$,
\item $M$ is disjoint from $N$, and
\item $M \sqcup N$ is nonseparating,
\end{itemize}
\bn the coinvariants module $V_{\Stab_G(M)}$ is finite dimensional.
\end{enumerate}

\bn Then $V$ is finite dimensional.
\end{labelenv}

\bn In practice, we will often show that the first hypothesis is satisfied by showing that the natural map $\bigoplus_{c \in M} V^{T_c} \rightarrow V$ is surjective.  We will prove Proposition \ref{gengrpprop} in Section \ref{gengrpsection}.  

\subsection{The outline of the paper} The paper is organized into four chunks.

\begin{itemize}
\item Section \ref{gengrpsection}, which contains the proof of Proposition \ref{gengrpprop}.
\item Sections \ref{homolcurvequotsection} -- \ref{finquotsectionpt2}, which contain the proof of Proposition \ref{homolcurvequotprop}.
\item Sections \ref{11finquotsection} and \ref{fincoksection}, which contain the proof of Proposition \ref{11finprop}.  Additionally, Section \ref{fincoksection} contains the proof of Theorem \ref{fincokthm}.
\item Section \ref{mainpfsection}, which contains the proof of Theorem \ref{mainthm}.
\end{itemize}

\bn We now explain the organization of the second and third chunk in more detail.

\p{Sections \ref{homolcurvequotsection} -- \ref{finquotsectionpt2}, which prove Proposition \ref{homolcurvequotprop}} Our goal is to prove Proposition \ref{homolcurvequotprop} by applying Proposition \ref{gengrpprop} with $G = \Stab_{\Sp(2g,\Z)}(\vec{x})$, $d = 1$, and $V = H_2(X_g;\Q)$ (recall that $X_g = \cC_{\vec{x}}(S_g)/\cI_g$).  In particular, if $a \subseteq S_g$ is a representative of $\vec{x}$, we will prove the following:
\begin{enumerate}
\item For any primitive $c$ disjoint from and not homologous to $a$, the cokernel of $H_2(X_g;\Q)^{T_{[c]}} \rightarrow H_2(X_g;\Q)$ is finite dimensional, and
\item the coinvariants module $H_2(X_g;\Q)_{G}$ is finite dimensional.
\end{enumerate}
\bn The latter property is not too difficult to prove, while the former property is the bulk of the work of Sections \ref{homolcurvequotsection}-- \ref{finquotsectionpt2}.  The proof of the first hypothesis proceeds in the following steps:
\begin{enumerate}
\item Construct a subspace spanned by fundamental classes of tori $T \subseteq X_g$ called Bestvina--Margalit tori.  These tori are constructed in Section \ref{homolcurvequotsection}, and the subspace spanned by their fundamental classes is denoted $\BM_2(X_g;\Q)$.
\item Show that for any primitive $c \subseteq S_g$ disjoint from and inhomologous to $a$, the cokernel $\BM_2(X_g;\Q)^{T_{[c]}} \rightarrow \BM_2(X_g;\Q)$ is spanned by classes in $\BM_2(X_g;\Q)$.  This is the bulk of the work of Section \ref{homolcurvequotsection}, and is stated as Lemma \ref{vstablemma}.  
\item Use Proposition \ref{gengrpprop} with $G = \Stab_{\Sp(2g,\Z)}(\vec{x}, [c])$, $d = 9$, and $V$ the cokernel of the map $H_2(X_g;\Q)^{T_{[c]}} \rightarrow H_2(X_g;\Q)$.  This is carried out in two substeps, which each verify a hypothesis of Proposition \ref{gengrpprop}.
\begin{enumerate}
\item The subspace $\BM_2(X_g;\Q)$ is the image of a subspace $H_2^{\ab,\bp}(\cI_g;\Q) \subseteq H_2(\cI_g;\Q)$ generated by abelian cycles consisting of bounding pair maps.  The main work of Section \ref{abelcyclesection} is Proposition \ref{abcycleprop}, which says that for a nonseparating multicurve $M \subseteq S_g$ containing at least 9 curves, the map
\begin{displaymath}
\bigoplus_{d \subseteq M} \left(H_2^{\ab,\bp}(\cI_g;\Q)\right)^{T_{[d]}} \rightarrow H_2^{\ab,\bp}(\cI_g;\Q)
\end{displaymath}
\bn is surjective.  Since $V = H_2(X_g;\Q)/H_2(X_g;\Q)^{T_c}$ is a quotient of $\BM_2(X_g;\Q)$ by Lemma \ref{vstablemma} and $\BM_2(X_g;\Q)$ is a quotient of $H_2^{\ab,\bp}(\cI_g;\Q)$, the map
\begin{displaymath}
\rho: \bigoplus_{d \subseteq M} V^{T_{[d]}} \rightarrow V
\end{displaymath}
\bn is surjective.  In particular, this implies that $\cok(\rho)$ is finite dimensional, so hypothesis (1) of Proposition \ref{gengrpprop} is satisfied for $G = \Stab_{\Sp(2g,\Z)}(\vec{x}, c)$, $V = H_2(X_g;\Q)/H_2(X_g;\Q)^{T_c}$ and $d = 9$.
\item We now show that for $M$ a multicurve with $|M| \leq 8$ such that $M$ is disjoint from $a$ and $c$ and $a \sqcup c \sqcup M$ is nonseparating, the coinvariants module
\begin{displaymath}
V_{\Stab_{[M]}G}
\end{displaymath}
\bn is finite dimensional, where $[M]$ denotes the set of homology classes represented by elements of $M$.  Since $V$ is a quotient of $\BM_2(X_g;\Q)$ by Lemma \ref{vstablemma}, it suffices to show that $\BM_2(X_g;\Q)_{\Stab_{[M]}G}$ is finite dimensional.  This is the content of Lemma \ref{bmfinquotlemma} and is the main work of Section \ref{bmfinquotsection}.
\end{enumerate}
\bn Statements (a) and (b) are the hypotheses of Proposition \ref{gengrpprop} for $G = \Stab_{\Sp(2g,\Z)}(\vec{x}, [c])$, $d = 9$, and $V$ the cokernel of the map $H_2(X_g;\Q)^{T_{[c]}} \rightarrow H_2(X_g;\Q)$.  Hence $V$ is finite dimensional by Proposition \ref{gengrpprop}.
\end{enumerate}

\bn Hence for any primitive $c$ disjoint from and not homologous to $a$, the cokernel of $H_2(X_g;\Q)^{T_{[c]}} \rightarrow H_2(X_g;\Q)$ is finite dimensional.  If $G = \Stab_{\Sp(2g,\Z)}(\vec{x})$ then it is not too difficult to show that $H_2(X_g;\Q)_{G}$ is finite dimensional.  Hence by applying Proposition \ref{gengrpprop} with $G = \Stab_{\Sp(2g,\Z)}(\vec{x})$, $d = 1$, and $V = H_2(X_g;\Q)$, we conclude that $H_2(X_g;\Q)$ is finite dimensional, which is the statement of Proposition \ref{homolcurvequotprop}.

\p{Sections \ref{11finquotsection} and \ref{fincoksection}, which prove Proposition \ref{11finprop}} The approach is to apply Proposition \ref{gengrpprop} with $G = \Stab_{\Sp(2g,\Z)}(\vec{x})$, $d = 8$, and $V = \bE_{1,1}^2$.  We verify each hypothesis of Proposition \ref{gengrpprop} in turn.
\begin{enumerate}
\item Hypothesis (1) is stated as Lemma \ref{11stabgenlemma}, and is the main content of Section \ref{11stabgensection}.
\item Hypothesis (2) is stated as Lemma \ref{11finquotlemma} and is the main content of Section \ref{fincoksection}.
\end{enumerate}
\bn Given these two results, we prove Proposition \ref{11finprop} by applying Proposition \ref{gengrpprop}.  Additionally, Section \ref{fincoksection} contains the proof of Theorem \ref{fincokthm}.

\p{Acknowledgments}  We would like to thank our advisor Dan Margalit for many helpful math conversations and his continued support and encouragement.  We would like to thank Andy Putman for providing many helpful comments on several different iterations of this paper.  We would like to thank Peter Patzt for helping to clarify the exposition and for pointing out some confusing notational choices.  We would like to thank Benson Farb for reading and providing comments on an earlier version of this paper.  We would like to thank  Igor Belegradek, Wade Bloomquist, Katherine Williams Booth, John Etnyre, Annie Holden, Thang Le, Nick Salter, Roberta Shapiro, Cindy Tan, and Bena Tshishiku for helpful conversations.

\section{A finiteness result about groups acting on vector spaces}\label{gengrpsection}

The goal in this section is to prove Proposition~\ref{gengrpprop}, which is a result that we will apply repeatedly to determine that certain representations of subgroups of the symplectic group are finite dimensional.  The statement is as follows.

\begin{prop}\label{gengrpprop}
Let $g \geq 1$, and let $N \subseteq S_g$ be a nonseparating multicurve such that $|N| < g$.  Let $G \subseteq \Sp(2g,\Z)$ be the image of the map $\Stab_{\Mod(S_g)}(N) \rightarrow \Sp(2g,\Z)$.  Let $V$ be a $G$--representation over $\Q$.  Suppose that there is a constant $0 \leq d \leq g- |N|$ such that the following hold:

\begin{enumerate}
\item For any $M \subseteq S_g$ such that:
\begin{itemize}
\item $|M| \geq d$,
\item $M$ is disjoint from $N$, and
\item $M \sqcup N$ is nonseparating,
\end{itemize}
\bn the cokernel of the map $\bigoplus_{c \in M} V^{T_c} \rightarrow V$ is finite dimensional.
\item For any $M \subseteq S_g$ such that:
\begin{itemize}
\item $|M| < d$,
\item $M$ is disjoint from $N$, and
\item $M \sqcup N$ is nonseparating,
\end{itemize}
\bn the coinvariants module $V_{\Stab_G(M)}$ is finite dimensional.
\end{enumerate}

\bn Then $V$ is finite dimensional.
\end{prop}

\bn Before proving the proposition, we will need a result about generating sets of subgroups of the symplectic group, which is Lemma \ref{spgenlemma}. We will also standardize some terminology that we will use throughout the paper.

\p{Terminology on symplectic lattices} A \textit{quasi--unimodular lattice} is a finitely generated free abelian group $L$ equipped with an alternating form $\langle \cdot, \cdot \rangle: L \times L \rightarrow \Z$.  The lattice $L$ is \textit{unimodular} if the form $\langle \cdot, \cdot \rangle$ is nondegenerate.  The \textit{genus of} $L$, denoted $g(L)$, is $\frac 1 2 \rk(W)$, where $W \subseteq L$ is a maximal free abelian subgroup of $L$ such that the restriction of $\langle \cdot, \cdot \rangle$ to $W$ is nondegenerate.  If $v \in L$ is some element, then $v$ is \textit{primitive} if there is no integer $m \geq 2$ and nonzero $w \in L$ with $mw= v$.  For any $v \in L$, the \textit{transvection along} $v$ is the homomorphism $T_v: L \rightarrow L$ given by
\begin{displaymath}
T_v(w) = w + \langle v,w \rangle v.
\end{displaymath}
\bn  We will say that the  transvection $T_v$ is \textit{primitive} if $v$ is primitive.  We will say that a subgroup $L' \subseteq L$ is \textit{primitive} if $L$ is generated by primitive elements.  The set of degenerate elements in $L$, i.e., elements $v \in L$ such that $\langle v, \cdot \rangle: L \rightarrow \Z$ is the zero map, will be denoted $L^{\degen}$.  If $\calV \subseteq L$ is a set of elements, we will denote by $\calV^{\perp}$ the subgroup $\{w \in L: \langle w, v \rangle = 0 \text{ for all } v \in \calV\}$.  We let 
\begin{displaymath}
    \Sp(L,\Z) = \{g \in \GL(L,\Z): \langle gv, gw \rangle = \langle v, w \rangle \text{ for all } v,w\in L\}.
\end{displaymath}
\bn  If $L' \subseteq L$ is a primitive, unimodular subgroup, let $\proj_{L'}:L \rightarrow L'$ denote the projection map induced by the form $\langle \cdot, \cdot \rangle$.  

\p{Multicurves} We will say that a \textit{multicurve} $M \subseteq S_g$ is an isotopy class of pairwise disjoint oriented curves.  We will use the notation $\pi_0(M)$ to denote the set of isotopy classes of curves in $M$.  The stabilizer group $\Stab_{\Mod(S_g)}(M)$ will denote the subgroup of $\Mod(S_g)$ consisting of elements that stabilize each curve $c \in \pi_0(M)$.  Note that since our multicurves are oriented, the stabilizer $\Stab_{\Mod(S_g)}(M)$ does not change the orientation of the curves of $M$.

\p{Notation} For the remainder of the paper, we will use $\cut$ to denote Farb and Margalit's notion of cutting along curves on surfaces~\cite[Section 3.6]{FarbMarg}.

\p{Mapping class groups of cut open surfaces} Let $M \subseteq S_g$ be a nonseparating multicurve.  The mapping class group $\Mod(S_g \cut M)$ is the mapping class group of the cut open surface.  There is a natural inclusion $S_g \cut M \hookrightarrow M$, and this induces a map $\Mod(S_g \cut M) \rightarrow \Mod(S_g)$ given by extending $\varphi \in \Mod(S_g \cut M)$ by the identity along $S_g \setminus \im(S_g \cut M \rightarrow S_g)$.

We are now ready to state and prove a preparatory lemma, after which we will conclude Section \ref{gengrpsection} by proving Proposition \ref{gengrpprop}.

\begin{lemma}\label{spgenlemma}
Let $g \geq 2$ and $b \geq 0$, and let $S = S_g^b$ be an oriented, compact surface of genus $g$ with $b$ boundary components.  Let $M \subseteq S$ be a nonseparating multicurve such that $g(S \cut M) \geq 1$.  Let $G$ be the image of the composition
\begin{displaymath} 
\Mod(S \cut M) \rightarrow \Mod(S) \rightarrow \Aut(H_1(S_g^b;\Z)).
\end{displaymath}
\bn  Then $G$ has a finite generating set consisting of transvections along classes $v$ represented by nonseparating curves $c \subseteq S_g \cut M$.
\end{lemma}

\begin{proof} Korkmaz~\cite[Theorem 3.1]{Korkmaz} has shown that if $S_{g'}^{b'}$ is a compact, oriented surface with genus $g' \geq 1$, then $\Mod(S_{g'}^{b'})$ has a finite generating set $\calD \subseteq \Mod(S_{g'}^{b'})$ consisting of Dehn twists along nonseparating curves.  Then $S \cut M$ is a compact, oriented surface, and has genus at least one by hypothesis.  Hence there is a finite set of Dehn twists along nonseparating curves $\calD \subseteq \Mod(S \cut M)$ such that $\calD$ generates $\Mod(S \cut M)$.  Then the image of $T_d \in \calD$ under the symplectic representation is a transvection along $[d]$.  Hence $G$ has a finite generating set transvections along elements represented by nonseparating curves, as desired.
\end{proof}

\bn We are now ready to conclude Section \ref{gengrpsection}.

\begin{proof}[Proof of Proposition \ref{gengrpprop}]
The proof proceeds by induction on $d$.

\p{Base case: $d = 0$} In this case, the first hypothesis says that if $M$ is the empty multicurve, then the cokernel $\bigoplus_{c \in \emptyset} V^{T_c} \rightarrow V$ is finite dimensional.  This is just saying that the cokernel of the zero map is finite dimensional, so in particular $V$ is finite dimensional.

\p{Inductive step: $d \geq 1$} Assume that the proposition holds for all $d' < d$.  We will show that it holds for $d$ as well.  Our aim is to show that if $G$ and $V$ satisfy the hypotheses of the proposition for $d$, then they satisfy the hypotheses of the proposition for $d -1$ as well.  The inductive hypothesis would then imply that $V$ is finite dimensional.  The second hypothesis, namely that the coinvariants module $V_{\Stab_G(M)}$ is finite dimensional for $|M| < d$ also holds for $|M| < d -1$, so it suffices to prove the first hypothesis for $d-1$.  It suffices to show that for $M$ a multicurve with:
\begin{itemize}
\item $|M|  \geq d - 1$,
\item $M$ is disjoint from $N$, and
\item $M \sqcup N$ is nonseparating,
\end{itemize}
\bn the cokernel of the map $\bigoplus_{c \in M} V^{T_c} \rightarrow V$ is finite dimensional.  This holds for $|M| \geq d$ by hypothesis, so it suffices to show that the result holds for $|M| = d-1$.  

Let $M \subseteq S_g \cut N$ be a nonseparating multicurve with $\left|M\right| = d-1$.  Let $G' = \im(\Mod(S_g \cut (M \sqcup N) \rightarrow \Sp(2g,\Z))$.  By Lemma \ref{spgenlemma}, $G"$ is generated by a finite set of transvections $\calF$ along a set of primitive elements $\calV$ such that each $v \in \calV$ has a nonseparating representative $c \subseteq S_g \cut M$.  Let $\calR$ be a set of such representatives, one for each element of $\calV$.  Now, for any $c \in \calR$, let $M_c = M \sqcup c$.  The multicurve $M_c$ satisfies:
\begin{itemize}
\item $|M_c| = 1 + |M| = d$,
\item $M_c$ is disjoint from $N$ since $M$ and $c$ are, and
\item $M_c \sqcup N$ is nonseparating since $c$ is nonseparating on $S_g \cut (M \sqcup N)$ by our choice of generating set.
\end{itemize}
\bn Therefore, the first hypothesis in the proposition applied to $V$ as a $G$ representation tells us that 
\begin{displaymath}
\cok(\bigoplus_{d \in M_c} V^{T_{[d]}} \rightarrow V)
\end{displaymath}
\bn is finite dimensional.  In particular, this implies that, if we let $V_M = \cok(\bigoplus_{d \in M} V^{T_d} \rightarrow V)$, then $V_M^{T_v}$ has finite codimension in $V_M$ for any $v \in \calV$.  Now, consider the filtration of $V_M$ given by ordering $\calV = \{v_1,\ldots, v_n\}$ and setting $W_{M,i} = \bigcap_{j=1}^iV_M^{T_{v_i}}$. Since $V_M^{T_{v_i}}$ has finite codimension in $V_M$, each $V_{M,i}$ has finite codimension in $V_{M,i-1}$.  In particular, this implies that $W_{M,n}$ has finite codimension in $V_M$.  Thus, it suffices to show that $W_{M,n}$ is finite dimensional.

\p{$W_{M,n}$ is finite dimensional} Consider the long exact sequence in $G'$--homology associated to the short exact sequence
\begin{displaymath}
0 \rightarrow W_{M,n} \rightarrow V_M \rightarrow V_M/W_{M,n} \rightarrow 0.
\end{displaymath}
\bn Part of this sequence is given by
\begin{displaymath}
H_1(G';V_M/W_{M,n}) \rightarrow H_0(G';W_{M,n}) \rightarrow H_0(G';V_M).
\end{displaymath}
\bn The group $G'$ is finitely presented by Lemma \ref{spgenlemma} and $V_M/W_{M,n}$ is finite dimensional by hypothesis, so $H_1(G';V_M/W_{M,n})$ is finite dimensional.  Furthermore, $H_0(G';V_M)$ is a quotient of $H_0(G';V)$.  This last vector space is finite dimensional by applying the second hypothesis of the proposition to the multicurve $M$, since we know $|M| = d- 1 < d$.  Therefore, the vector space $H_0(G';W_{M,n})$ is finite dimensional.  But now, $W_{M,n}$ is the intersection $\bigcap_{v \in \calV} V_M^{T_v}$.  Since the set $\calF = \{T_v:v \in \calV\}$ generates $G'$ by assumption, we conclude that $W_{M,n}$ is a trivial $G'$--representation, so $H_0(G';W_{M,n}) = W_{M,n}$.  Hence $W_{M,n}$ is finite dimensional and has finite codimension in $V_M$, so $V_M$ is finite dimensional as well.  The proof is now complete by the inductive hypothesis that the proposition holds for $d-1$.
\end{proof}

\section{Bestvina--Margalit tori}\label{homolcurvequotsection}

For the remainder of this section, fix a $g \geq 3$ and $a \subseteq S_g$ a nonseparating curve.  Let $X_g = \cC_{[a]}(S_g)/\cI_g$.  A Bestvina--Margalit torus (defined below), is a certain type of 2--torus embedded in $X_g$.  There is a subspace $\BM_2(X_g;\Q) \subseteq H_2(X_g;\Q)$ generated by the fundamental classes of the Bestvina--Margalit tori.  The goal in this section is to prove basic structural results about $X_g$ and Bestvina--Margalit tori.

\p{Bestvina--Margalit tori, by example} In unpublished work~\cite{BMtori}, Bestvina and Margalit associate to each 2-cell $\sigma$ in $X_g$ a corresponding dual cell.  The cells $\sigma$ and $\sigma'$ form a torus which we refer to as a \textit{Bestvina--Margalit torus}.  The idea is that $\sigma$ and $\sigma'$ have the same edges, except the edges are in a different cyclic order.  Let $x,y,z,w \in \cC_{[a]}(S_g)$ be as in Figure~\ref{dualcellpic}.  Let $\sigma \subseteq X_g$ be the 2-cell such that there is a lift $\widehat{\sigma} \subseteq C_{[a]}(S_g)$ with vertices $x,y,$ and $w$, and let $\sigma' \subseteq X_g$ be the 2--cell with a lift $\widehat{\sigma}'$ with vertices $x,z,$ and $w$.  Then the edges $e_{x,y}$ and $e_{y,w}$ are in the same orbit under the action of $\cI_g$ to $e_{z,w}$ and $e_{x,z}$, respectively~\cite[Proposition 3.29]{Putmanbook}.  However, the cells $\widehat{\sigma}$ and $\widehat{\sigma}'$ are not in the same orbit under the action of $\cI_g$, since they induce decompositions of $H_1(S_g;\Z)$ with different cyclic orders.  Hence the cells $\sigma$ and $\sigma'$ form a torus in $X_g$.  
\begin{figure}[ht]
\begin{tikzpicture}[scale=0.8]
\node[anchor = south west, inner sep  = 0] at (0,0){\includegraphics[scale=0.8]{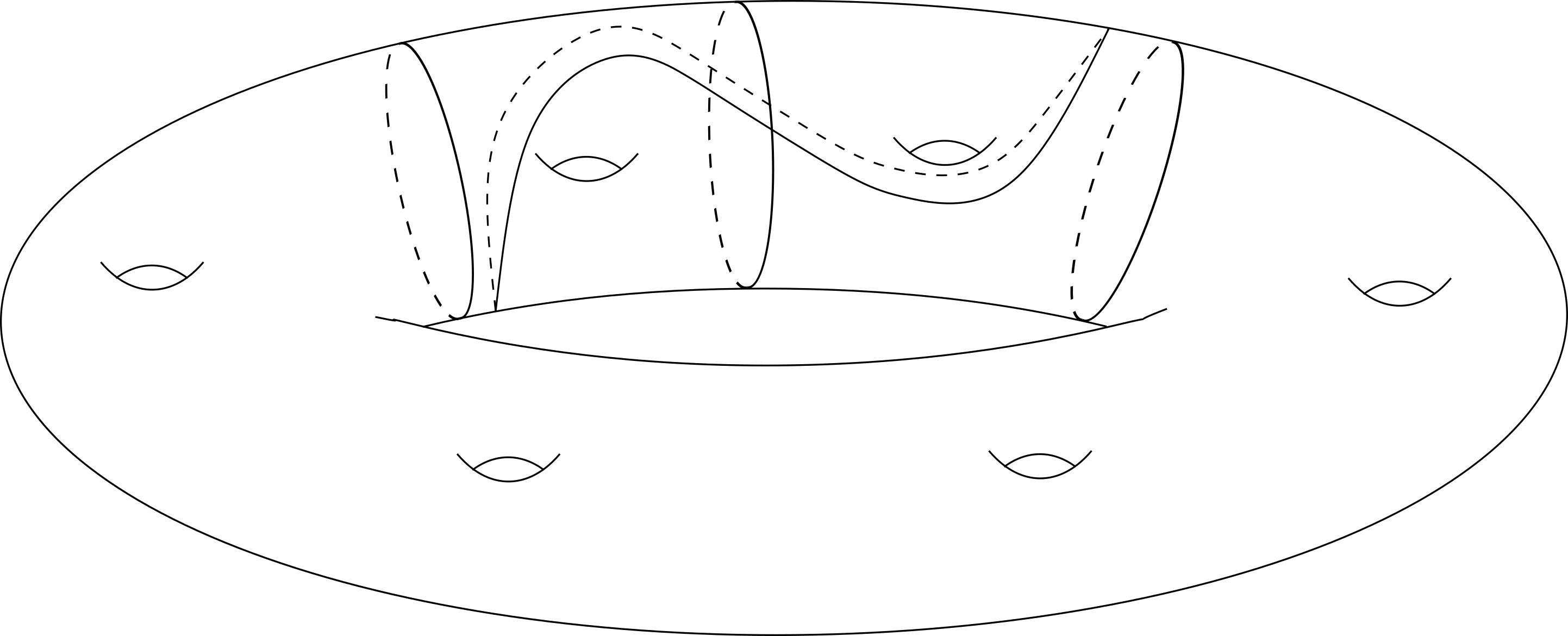}};
\node at (4.4,4.9){\large $x$};
\node at (7.4,6.5){\large $y$};
\node at (13.9,5){\large $w$};
\node at (9.2,4.8){\large $z$};
\end{tikzpicture}
\caption{The construction of dual cells.}\label{dualcellpic}
\end{figure}
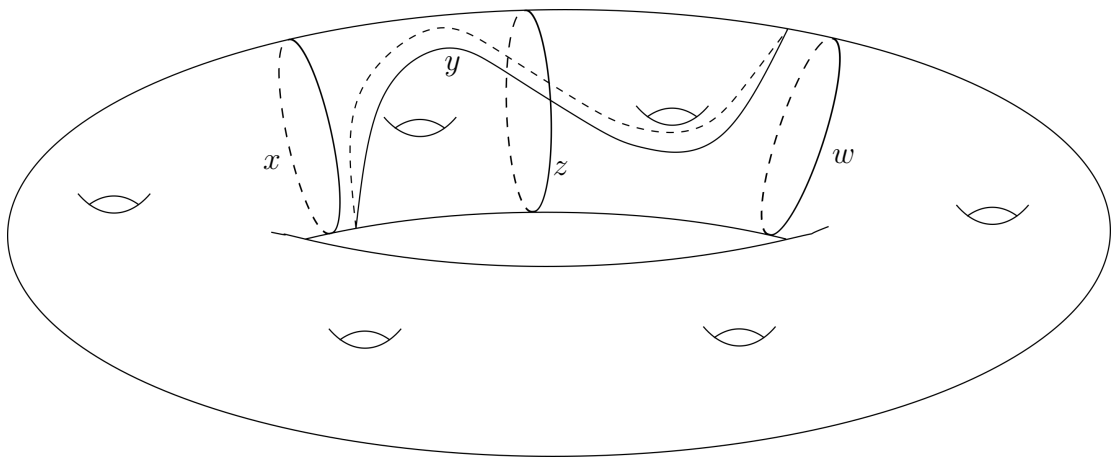
\bn We will denote this torus $\BM_{\sigma}$ and its fundamental class by $[\BM_\sigma]$.  In general, a \textit{Bestvina--Margalit} torus is any subcomplex of $X_g$ given by the union of a pair of 2--cells $\sigma$ and $\sigma'$ such that each edge of $\sigma$ is in the same $\cI_g$--orbit as an edge of $\sigma'$, but $\sigma \neq \sigma'$.  Let $\BM_2(X_g;\Q)$ denote the subspace of $H_2(X_g;\Q)$ generated by the set of fundamental classes $\{[\BM_{\sigma}]\}_{\sigma \in X_g^{(2)}}$.  This subspace plays a crucial role in the proof of Proposition~\ref{homolcurvequotprop}.   

\subsection{Structural results about \bm{$\cC_{[a]}(S_g)/\cI_g$}}\label{structsubsection} In this section we will prove Lemma~\ref{homolcurveconjstructurelemma}, which will tell us when two $k$--cells $\sigma, \sigma' \subseteq \cC_{[a]}(S_g)$ are in the same $\cI_g$--orbit for the action of $\cI_g$ on $\cC_{[a]}(S_g).$  It is known that two nonseparating curves $c$ and $d$ in $S_g$ are in the same orbit under the action of the Torelli group if and only if $c$ and $d$ are homologous (see, e.g., Putman~\cite[Lemma 6.2]{Putmantorelli}).  Hence the complex $X_g$ has one vertex.  Since $X_g$ has one vertex, there is a group generated by the edges of $X_g$ with relations given by 2--cells.  We will use this fact to denote the composition of two edges $y$ and $z$ that share a 2--cell by $yz$, where this edge is the third edge of the 2--cell.  This specifies an orientation on the edges, since there is one orientation of the edges $y$ and $z$ such that $y * z = yz$, where $*$ is the composition in the group generated by the edges on $X_g$.  If $e \subseteq \cC_{[a]}(S_g)$ is an edge, there is a bounding pair $T_{c,c'}$ taking one endpoint of $e$ to the other endpoint.  See Figure~\ref{hatchermargfig} for an example.
\begin{figure}[ht]
\begin{tikzpicture}[scale=0.9]
\node[anchor = south west, inner sep = 0] at (0,0){\includegraphics{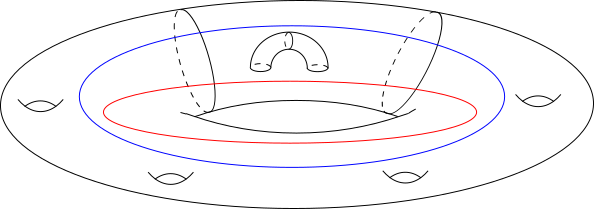}}[scale=0.9];
\node at (8.4,4){$v$};
\node at (8.4,5.6){$w$};
\node at (6.4,4.2){$c$};
\node at (12.8,4.2){$c'$};
\end{tikzpicture}
\caption{The bounding pair map $T_{c}T_{c'}^{-1}$ and edge $e = e_{v,w}$.}\label{hatchermargfig}
\end{figure}
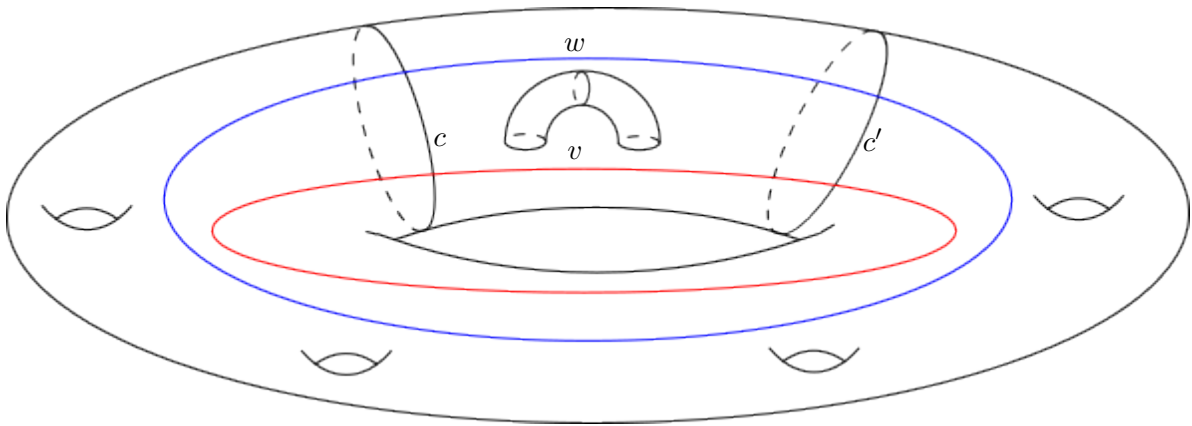
Lemma~\ref{homolcurveconjstructurelemma} describes the cells of the complex $X_g$.  Before stating and proving the lemma, we introduce some notation.  Let $\sigma \subseteq X_g$ be a $k$--cell with a lift $\widehat{\sigma} \subset \cC{[a]}(S_g)$, where $\widehat{\sigma}$ corresponds to a multicurve $a_0\sqcup\ldots \sqcup a_k$.  Let $S_0,\ldots,S_k$ be the connected components of $S_g \cut \widehat{\sigma}$.  Assume that $S_i \cap S_j \neq \emptyset$ if and only if $i-j \equiv \pm 1 \mod k+1$.  For each $S_i$, let $\calH_i^\sigma$ denote the image in $H_1(S_g;\Z)$ of the pushforward map $H_1(S_i;\Z) \rightarrow H_1(S_g;\Z)$.  Let $\calH(\sigma)$ denote the set of free abelian groups $\calH_0^\sigma,\ldots, \calH_k^\sigma$.  Associated to $\calH(\sigma)$ is a directed cycle $C$.  The vertices of $C$ are the connected components $S_i$ and there is a directed edge $S_i \rightarrow S_j$ if the following hold:
\begin{itemize}
\item $a_{\ell} \subseteq S_i \cap S_j$ for some $0 \leq \ell \leq k$, and
\item $a_{\ell}$ is oriented so that $S_i$ is on the left of $a_{\ell}$.
\end{itemize}
\bn  This directed cycle induces a cyclic ordering on the set $\calH(\sigma)$.  For convenience, we will index the $\calH_i$ so that the cyclic order is $\calH_0 < \calH_1 < \ldots < \calH_k < \calH_0$.  We will refer to this cyclically ordered set $\calH(\sigma)$ as the $\textit{cyclic decomposition of } H_1(S_g;\Z) {induced by } \widehat{\sigma}$.  We will say that the $\textit{genus of } \widehat{\sigma}$ is the cyclically ordered tuple $\left(g(\calH_0),\ldots, g(\calH_k)\right)$, and we will denote this tuple by $g(\widehat{\sigma})$.  We will prove the following result.

\begin{lemma}\label{homolcurveconjstructurelemma}
Let $\sigma$ and $\sigma'$ be two $k$--cells of $\cC_{[a]}(S_g)$.  Then there is an $f \in \cI_g$ such that $f\sigma = \sigma'$ if and only if $\sigma$ and $\sigma'$ induce the same decomposition of $H_1(S_g;\Z)$.
\end{lemma}

\begin{proof} We first prove the forwards implication and then we prove the backwards implication.

\p{The forward implication} Let $\sigma, \sigma' \subseteq C_{[a]}(S_g)$ be 2--cells in the same $\cI_g$ orbit.  Let $f \in \cI_g$ be an element with $f\sigma = \sigma'$.  Since $f \in \Mod(S_g)$, $f$ takes any $\calH \in \calH(\sigma)$ to some $\calH' \in \calH(\sigma')$, and since $f \in \cI_g$, we must have $\calH = \calH'$.  Therefore $\calH(\sigma) = \calH(\sigma')$ as unordered sets, so it remains to show that $f$ preserves the cyclic order on $\calH(\sigma)$ and $\calH(\sigma')$.  Let $b \subseteq S_g$ be an oriented curve such that for any vertex $d\in \sigma$, $b$ intersects $d$ once with signed intersection $-1$.  Then since $f(\sigma) = \sigma'$, the image $f(b)$ satisfies the property that for any $d' \in \sigma'$, $f(b)$ intersects $cd$ once with signed intersection number $-1$.  Now, for each $\calH_i \in \calH(\sigma)$, let $c_i$ be a nonseparating simple closed curve disjoint from $\sigma$ and not homologous to $[c]$ such that $c_i$ intersects $b$ once with signed intersection $-1$.  An example with $\sigma$ a 2--cell of the curves $\sigma,b$, and $c_i$ can be found in Figure~\ref{toralnonconjfig}, if we denote the vertices of $\sigma$ by $d_0, d_1,d_2$.

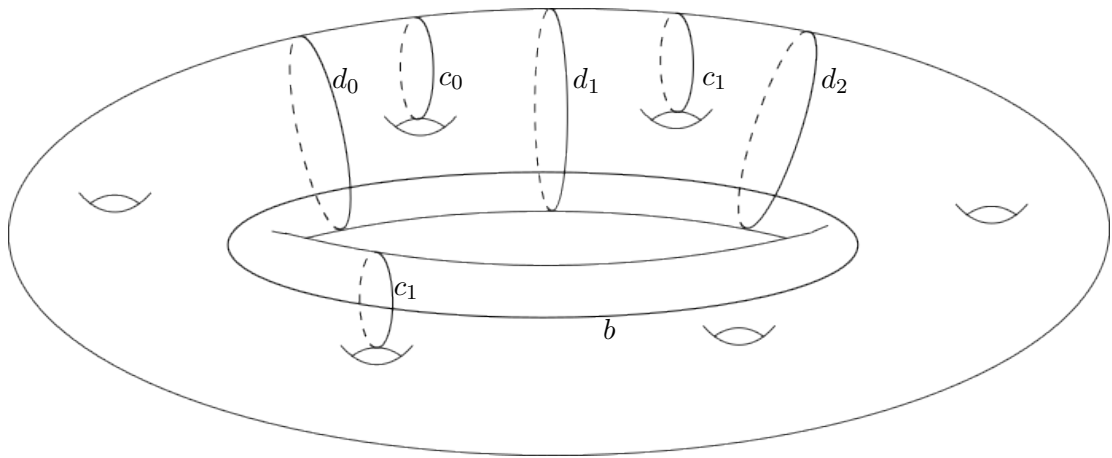
\begin{figure}[h]
\begin{tikzpicture}
\node[anchor = south west, inner sep = 0] at (0,0){\includegraphics[scale=0.8]{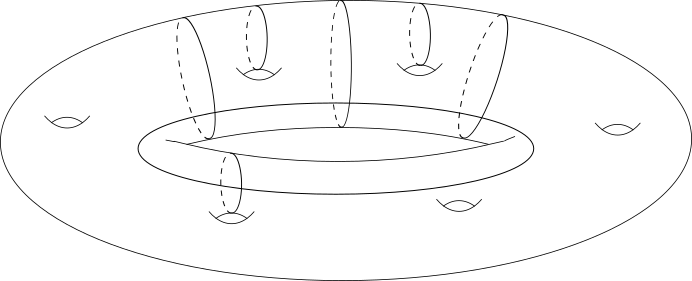}};
\node at (4.5,5) {$d_0$};
\node at (5.9,5) {$c_0$};
\node at (7.7,5) {$d_1$};
\node at (9.4,5) {$c_1$};
\node at (11,5) {$d_2$};
\node at (5.3,2.2) {$c_1$};
\node at (8, 1.7) {$b$};
\end{tikzpicture}
\caption{The curves $a_0,a_1,a_2$, $b$, and $c_i$}\label{toralnonconjfig}
\end{figure} 

Let $c'_i = f(c_i)$.  Then $[c'_i]= [c_i]$ since $f \in \cI_g$, and $[c'_i] \in \calH_i(\sigma')$.  Then the cyclic order of the elements of $\calH(\sigma)$ is the same as the cyclic order of the intersections of $c_i$ with $b$, where the cyclic order on the $c_i$ is induced by the orientation of $b$, so the lemma follows.

\p{The backwards implication} Suppose now that $\sigma$ and $\sigma'$ induce the same decomposition of $H_1(S_g;\Z)$.  Reindex $\calH(\sigma')$ so that $\calH_i^{\sigma} = \calH_i^{\sigma'}$ for all $0 \leq i \leq \dim(\sigma)$, and let $S_i^{\sigma}$, $S_i^{\sigma'}$ denote the connected components of $S_g \cut \sigma$ and $S_g \cut \sigma'$.  Let $\alpha_1,\beta_1,\ldots, \alpha_g, \beta_g$ be a symplectic basis for $H_1(S_g;\Z)$ such that:
\begin{itemize}
\item $[a] = \alpha_1$ and
\item there is a partition $\nu \vdash\{2,\ldots, g\}$ where $\nu$ has $k +1$ blocks such that, for each $P_i \in \nu$, the set $\{\alpha_j, \beta_j\}_{j \in P_i}$ is a symplectic basis for a maximal unimodular subgroup of $\calH_i^\sigma \in \calH(\sigma)$.
\end{itemize}
\bn For each $\calH_i^{\sigma}$ and $\calH_i^{\sigma'}$, let $\{a_j, b_j\}_{j \in P_i}$ and $\{a_j', b_j'\}_{P_i}$ be representatives in $S_i^{\sigma}$ and $S_i^{\sigma'}$ respectively for $\{\alpha_j,\beta_j\}_{j \in P_i}$ such that $|a_i \cap b_j| = |a'_i \cap b'_j| = \left| \langle \alpha_i,\beta_j \rangle\right|$ for any $1 \leq i,j \leq g$.  By the change of coordinates principle, there is an $f \in \Mod(S_g)$ taking $\sigma$ to $\sigma'$ that also takes $a_j \rightarrow a'_j, b_j \rightarrow b_j'$ for all $2 \leq j \leq g$.  Then $f_*(\alpha_i) = \alpha_i$ and $f_*(\beta_i) = \beta_i$ for $2 \leq i \leq g$.  Then $f$ takes a vertex of $\sigma$ to a vertex of $\sigma'$, so $f_*(\alpha_1) = \alpha_1$.  This implies that $\langle f_*\beta_1, \alpha_i \rangle = \langle f_*\beta_1, f_*(\beta_i)\rangle  = 0$ for any $2 \leq i \leq g$, and $\langle \alpha_1, f_*(\beta_1)\rangle = 1$.  Hence $f(\beta_1) = \beta_1 + n \alpha_1$ for some $n \in \Z$, so $f$ is not \textit{a priori} in the Torelli group, since it may act nontrivially on $\beta_1$.  We now find another element $h \in \Mod(S_g)$ such that $hf\sigma = \sigma'$ and $hf \in \cI_g$.

Let $a' \in \sigma'$ be a curve in the $k$--cell $\sigma'$.  Then $T_{a'}^{-n}f\sigma = T_{a'}^{-n}\sigma' = \sigma'$ since $f\sigma = \sigma'$ and $T_{a'}$ acts trivially on $\sigma'$ by construction.  Furthermore, for any $m \in \Z$, we have 
\begin{displaymath}
T_{[a']}(\beta_1 + m \alpha_1) = \beta_1 + m \alpha_1 + \langle[a'], \beta_1 + m \alpha_1\rangle [a'].
\end{displaymath}
\bn  Since $[a'] = \alpha_1$, we therefore have $T_{[a']}(\beta_1 + m \alpha_1) = \beta_1 + m \alpha_1 + 1 \alpha_1 = \beta_1 + (m+1)\alpha_1$.  Therefore $T_{[a']}^{-n}(\beta_1 + m\alpha_1) = \beta_1 + (m-n)\alpha_1$.  Therefore we have we have $T_{[a']}^{-n}f_*(\beta_1) = T_{[a']}^{-n}(\beta_1 + n \alpha_1) =  \beta_1 + n \alpha_1 - n \alpha_1 = \beta_1$.  Then $T_{a'}$ acts trivially on the set $\{\alpha_1,\alpha_2,\beta_2,\ldots, \alpha_g, \beta_g\}$ since $[a'] = \alpha_1$ trivially intersects every element in this set.   Hence $(T_a^{-n}f)_* $ fixes the set $\{\alpha_1,\beta_1,\ldots, \alpha_g, \beta_g\}$, so $T_{a}^{-n} f \in \cI_g$.  Hence $T_{a}^{-n}f$ is an element of the Torelli group taking $\sigma$ to $\sigma'$, as desired.
\end{proof}

\bn Given Lemma~\ref{homolcurveconjstructurelemma}, if $\sigma \subseteq X_g$ is a cell, the cyclically ordered decomposition $\calH(\widehat{\sigma})$ for $\widehat{\sigma}$ a lift of $\sigma$ to $\cC_{[a]}(S_g)$ depends only on $\sigma$, so we will use the notation $\calH(\sigma)$ to refer to the decomposition of $H_1(S_g;\Z)$ induced by $\widehat{\sigma}$, and refer to $\calH(\sigma)$ as the \textit{decomposition of $H_1(S_g;\Z)$ induced by $\sigma$}.

\subsection{Bestvina--Margalit tori}\label{bmsubsection}  We will now give a formal definition of a Bestvina--Margalit torus and prove a collection of results about fundamental classes of Bestvina--Margalit tori. 

\p{Bestvina--Margalit tori} Let $\sigma \subseteq X_g$ be a 2--cell.  There is a unique cell $\sigma'$ called the \textit{dual cell to $\sigma$} such that $\calH(\sigma) = \calH(\sigma')$ as unordered sets, but with $\calH(\sigma) \neq \calH(\sigma')$ as cyclically ordered sets.  This $\sigma'$ is unique by Lemma~\ref{homolcurveconjstructurelemma} since there are exactly two cyclic orders on the set of three elements.  The Bestvina--Margalit torus $\BM_{\sigma}$ is the union $\sigma \cup \sigma'$, and the fundamental class of $\BM_{\sigma}$ is denoted $[\BM_{\sigma}]$.  We have the following lemma that allows us to add fundamental classes of Bestvina--Margalit tori.

\begin{lemma}\label{lemmaaltcohomolauxadd}
Let $x,y,z$ be three edges in $X_g$ such that $\BM_{x,y}$, $\BM_{x,z}$ and $\BM_{y,z}$ are all Bestvina--Margalit tori.  Let $yz$ denote the third edge of a 2--cell $\sigma$ with $y,z \subseteq \sigma$ oriented so that in the group on the edges of $X_g$, we have $y * z = yz$.  A lift of a 3--cell $\tau \subseteq X_g$ containing these edges can be seen in Figure~\ref{lemmaaltcohomolauxaddfig}.  Then in $H_2(X_g;\Q)$, the relation $[\BM_{x,y}] + [\BM_{x,z}] = [\BM_{x,yz}]$ is satisfied.
\end{lemma}

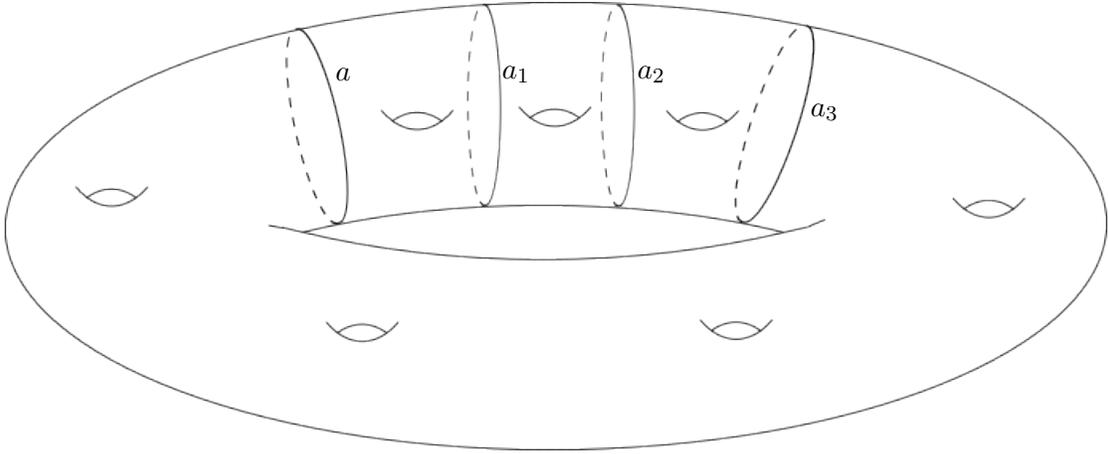
\begin{figure}[h]
\begin{tikzpicture}
\node[anchor = south west, inner sep = 0] at (0,0){\includegraphics[scale=0.8]{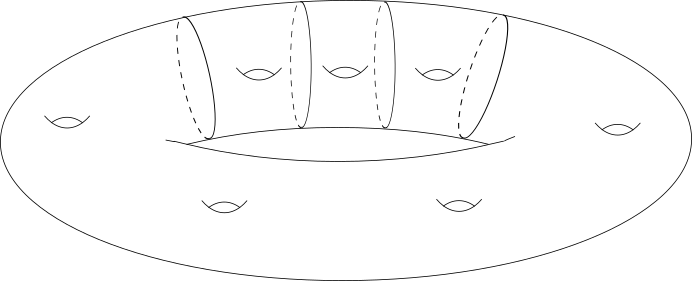}};
\node at (4.5,5){$a$};
\node at (6.8,5){$a_1$};
\node at (8.6,5){$a_2$};
\node at (10.9,4.5){$a_3$};
\end{tikzpicture}
\caption{A 3--cell containing lifts of the edges $x,y,z,yz$ as in Lemma~\ref{lemmaaltcohomolauxadd}.  $x$ lifts to $(a,a_1)$, $y$ lifts to $(a_1,a_2)$, $z$ lifts to $(a_2,a_3)$ and $yz$ lifts to $(a_1,a_3)$.}\label{lemmaaltcohomolauxaddfig}
\end{figure}

\bn Before proving the lemma, we will explain how to compute differentials in $X_g$.  If $\tau \subseteq \cC_{[a]}(S_g)$ is a $k$--cell with vertices $\{c_0,\ldots, c_k\}$, then the differential $\partial \tau$ is given by
\begin{displaymath}
\partial \tau = \sum_{i=0}^k (-1)^i (c_0,\ldots, \widehat{c_i}, \ldots, c_k)
\end{displaymath}
\bn where $(c_0, \ldots, \widehat{c_i}, \ldots, c_k)$ denotes the $(k-1)$--cell with vertices given by $c_0, \ldots, c_{i-1}, c_{i+1}, \ldots, c_k$.  This differential descends to $X_g$ by definition.  Now, if $\tau \subseteq X_g$ is a $k$--cell, Lemma~\ref{homolcurveconjstructurelemma} tells us that $\tau$ is determined by the decomposition
\begin{displaymath}
\calH(\tau) = \{\calH_0,\ldots, \calH_k\}.
\end{displaymath}
\bn Then, if $\overline{\tau}$ is a lift of $\tau$ to $\cC_{[a]}(S_g)$, forgetting a vertex $c \subseteq \overline{\tau}$, which yields a $k$--cell $\overline{\sigma}$, corresponds to replacing $\calH_i, \calH_{i+1}$ in $\calH(\tau)$ with $\calH_i + \calH_{i+1}$.  Therefore the differential $\partial \tau$ is given by
\begin{displaymath}
\partial \tau = \sum_{i=0}^k (-1)^i \sigma_i
\end{displaymath}
\bn where for $0 \leq i \leq k-1$, each $\sigma_i$ satisfies
\begin{displaymath}
\calH(\sigma_i) = \{\calH_0,\ldots, \calH_{i-1}, \calH_{i} + \calH_{i+1}, \calH_{i+2},\ldots, \calH_k\}.
\end{displaymath}
\bn We have that $\calH(\sigma_k)= \{\calH_1,\ldots, \calH_{k-1}, \calH_k + \calH_0\}$.

\begin{proof}[Proof of Lemma~\ref{lemmaaltcohomolauxadd}]
We will find 3--dimensional subcomplex $\overline{L} \subseteq X_g$ such that the boundary realizes the relation in the lemma.  Let $\tau$ be a 3--cell with $\calH(\tau) = \{\calH_0, \calH_1, \calH_2, \calH_3\}$ such that $\calH_0 \in \calH(x)$, $\calH_1 \in \calH(x)$, and $\calH_2 \in \calH(x)$.  The space $\overline{L}$ is given as the union of the following 3--cells:
\begin{itemize}
\item $\tau_0 = \tau$,
\item $\tau_1$ that satisfies $\calH(\tau_1) = \{\calH_1, \calH_0, \calH_2, \calH_3\}$, and
\item $\tau_2$ that satisfies $\calH(\tau_2) = \{\calH_1, \calH_2, \calH_0, \calH_3\}$.
\end{itemize}
\bn We can compute the differentials of each 3--cell above as follows, signed appropriately to yield the desired relation in the lemma.  For notational convenience, we will conflate $\sigma$ and $\calH(\sigma)$ for $\sigma \subseteq X_g$.
\begin{itemize}
\item $\partial \tau_0 = \{\calH_0 + \calH_1, \calH_2, \calH_3\} - \{\calH_0, \calH_1 + \calH_2, \calH_3\} + \{\calH_0, \calH_1, \calH_2 + \calH_3\} - \{\calH_0 + \calH_3, \calH_1, \calH_2\}$,
\item $-\partial \tau_1 = -\{\calH_0 + \calH_1, \calH_2, \calH_3\} + \{\calH_1, \calH_0 + \calH_2, \calH_3\} - \{\calH_1, \calH_0, \calH_2 + \calH_3\} + \{\calH_1 + \calH_3, \calH_0, \calH_2\}$, and
\item $\partial \tau_2 = \{\calH_1 + \calH_2, \calH_0, \calH_3\}-\{\calH_1, \calH_2 + \calH_0, \calH_3\} + \{\calH_1, \calH_2, \calH_0 + \calH_3\} - \{\calH_1 + \calH_3, \calH_2, \calH_0\}.$ 
\end{itemize}
\bn Now, we see that in the sum $\partial \tau_0 - \partial \tau_1 + \partial \tau_2$, the term $\{\calH_0 + \calH_1, \calH_2, \calH_3\}$ in $\partial \tau_0$ cancels with $-\{\calH_0 + \calH_1, \calH_2, \calH_3\}$ in $-\partial \tau_1$, the term $-\{\calH_0 + \calH_3, \calH_1, \calH_2\}$ in $\partial \tau_0$ cancels with $\{\calH_1, \calH_2, \calH_0 + \calH_3\}$ in $\partial \tau_2$, and $\{\calH_1, \calH_0 + \calH_2, \calH_3\}$ in $-\partial \tau_1$ cancels with $-\{\calH_1, \calH_2 + \calH_0, \calH_3\}$ in $\partial \tau_2$.  This means that
\begin{align*}
\partial \tau_0 - \partial \tau_1 + \partial \tau_2 =& -\{\calH_0, \calH_1 + \calH_2, \calH_3\} +\{H_0, \calH_1, \calH_2+ \calH_3\}\\
& - \{\calH_1, \calH_0, \calH_2 + \calH_3\} + \{\calH_1 + \calH_3, \calH_0, \calH_2\} \\
&+\{\calH_1 + \calH_2, \calH_0, \calH_3\} - \{\calH_1 + \calH_3, \calH_2, \calH_0\}.\\
\end{align*}
\bn Now, by rearranging terms, we have
\begin{align*}
\partial \tau_0 - \partial \tau_1 + \partial \tau_2 =& -\{\calH_0, \calH_1 + \calH_2, \calH_3\} +\{\calH_1 + \calH_2, \calH_0, \calH_3\}\\
& +\{H_0, \calH_1, \calH_2+ \calH_3\}- \{\calH_1, \calH_0, \calH_2 + \calH_3\}  \\
& + \{\calH_1 + \calH_3, \calH_0, \calH_2\}- \{\calH_1 + \calH_3, \calH_2, \calH_0\}.\\
\end{align*}
\bn Then each pair of two cells term on the right side of the above equation is a Bestvina--Margalit torus, so we have
\begin{displaymath}
\partial \tau_0 - \partial \tau_1 + \partial \tau_2 = -[\BM_{x,yz}] + [\BM_{x,y}] + [\BM_{x,z}]
\end{displaymath}
\bn and therefore the relationship given in the lemma holds.
\end{proof}

\bn We now prove Lemma~\ref{bmfinquotsplit}, which says when two Bestvina--Margalit tori are equal in $X_g$.

\begin{lemma}\label{bmfinquotsplit}
Let $\BM_{\sigma}$ and $\BM_{\tau}$ be two Bestvina--Margalit tori.  Then $\BM_{\sigma} = \BM_{\tau}$ if and only if $\calH(\sigma) = \calH(\tau)$ as unordered sets.  
\end{lemma}

\begin{proof}
We prove both directions.

\p{Backwards direction} We prove the contrapositive.  If $\BM_{\sigma} \neq \BM_{\tau}$, then $\sigma \neq \tau$ since every cell $\sigma$ has a unique dual 2--cell $\sigma'$ such that $\sigma \cup \sigma'$ is a Bestvina--Margalit torus.  This implies that $\calH(\sigma) \neq \calH(\tau)$ as decompositions of $[a]^{\perp}$, which implies in particular that $\calH(\sigma) \neq \calH(\tau)$ as unordered sets.

\p{Forwards direction}  Suppose that $\BM_{\sigma} = \BM_{\tau}$.  By Lemma~\ref{homolcurveconjstructurelemma}, $\sigma$ and $\tau$ are identified under the action of $\cI_g$ if and only if $\calH(\sigma) = \calH(\tau)$ as decompositions of $[a]^{\perp}$.  But then if $\calH(\sigma) \neq \calH(\tau)$, we must have $\calH(\sigma) = \calH(\tau')$, where $\tau'$ is the 2--cell dual to $\tau$, since there are only two cyclic orders on a set of three elements.
\end{proof}

\bn  In particular, if $\BM_{\sigma}$ is a Bestvina--Margalit torus, we will use the notation  $\calH(\BM_{\sigma})$ to denote the decomposition $\calH(\sigma)$ without the cyclic order, and we will refer to this decomposition as the \textit{decomposition of } $H_1(S_g;\Z)$ \textit{induced by $\BM_{\sigma}$}.  If $\sigma$ is a 2--cell with edges $x$ and $y$, we will denote the splitting $\calH(\BM_{\sigma})$ by $\calH(\BM_{x,y})$.  We describe how $\calH$ interacts with the addition of tori in the following lemma.  The idea is that if we have a 3--cell $\tau$ as in Figure~\ref{lemmaaltcohomolauxaddfig}, then we can describe the splitting induced by one torus corresponding to a 2--cell $\sigma \subseteq \tau$ in terms of the splitting induced by tori corresponding to other 2--cells contained in $\tau$.

\begin{lemma}\label{bmfinquotfundclassadd}
Let $g \geq \finalbound$ and let $a \subseteq S_g$ be a nonseparating simple closed curve.  Let $x,y,z \subseteq X_g$ be three edges such that $\BM_{x,y}$, $\BM_{x,z}$ and $\BM_{y,z}$ are all Bestvina--Margalit tori.  Suppose that we have
\begin{displaymath}
\calH(\BM_{x,y}) = \{\calH_0,\calH_1,\calH_2\} \text{ and } \calH(\BM_{x,z}) = \{\calH_0, \calH'_1, \calH'_2\}
\end{displaymath}
\bn  with $\calH_1 \cap \calH'_1 = \Z[a]$.  Suppose additionally that we have $\calH_0 \in \calH(x)$, $\calH_1 \in \calH(y)$ and $\calH'_1 \in \calH(z)$.  Let $yz$ denote the third edge of a 2--cell $\sigma$ with $y,z \subseteq \sigma$ with $yz$ oriented so that $y *z = yz$ in the group generated by edges of $X_g$.  Then we have
\begin{displaymath}
\calH(\BM_{x,yz}) = \{\calH_0, \calH_1 + \calH_1', \calH_2 \cap \calH'_2\}.
\end{displaymath}
\end{lemma}

\begin{proof}
We first want to compute the decomposition $\calH(yz)$.  The 2--cell $\sigma$ containing $y,z$ and $yz$ has $\calH_1, \calH'_1 \in \calH(\sigma)$.  This implies that the third element of $\calH(\sigma)$ is $\calH_1^{\perp} \cap (\calH'_1)^{\perp}$.  Hence 
\begin{displaymath}
\calH(yz) = \{\calH_1^{\perp} \cap (\calH'_1)^{\perp}, (\calH_1^{\perp} \cap (\calH'_1)^{\perp})^{\perp}\} = \{\calH_1^{\perp} \cap (\calH'_1)^{\perp}, \calH_1 + \calH_1'\}
\end{displaymath}
\bn and in particular, we have $\calH_1 + \calH'_1 \in \calH(yz)$.  Now, the splitting $\calH(\BM_{x,yz})$ contains $\calH_0$ and $\calH_1 + \calH_1'$, so the third element of $\calH(\BM_{x,yz})$ is given by
\begin{displaymath}
\calH_0^{\perp} \cap (\calH_1 + \calH_1')^{\perp} = (\calH_0 + \calH_1)^{\perp} \cap (\calH_0 + \calH_1')^{\perp} = \calH_2 \cap \calH_2',
\end{displaymath}
\bn so the lemma holds.
\end{proof}

\bn Before moving forward with the paper, we will briefly discuss a notation for cells in $X_g$, as well as the action of subgroups of $\Sp(2g,\Z)$ on $X_g$.

\p{Graphical representations for Bestvina--Margalit tori}  Let $x,y,z,yz \subseteq X_g$ be a set of edges as in Lemma~\ref{bmfinquotfundclassadd}.  Let $\widehat{\calH} = \{\calH_0, \calH_1, \calH_1', \calH_2 \cap \calH_2'\}$ be an unordered set.  Then $\widehat{\calH}$ corresponds to the labeled graph as in Figure~\ref{labelgraphex}.

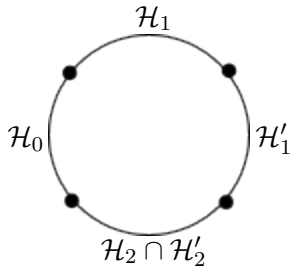
\begin{figure}[h]
\begin{tikzpicture}
\node[anchor = south west, inner sep = 0] at (0,0){\includegraphics{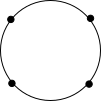}};
\node at (-0.3,1.3) {$\calH_0$};
\node at (1.4,2.9) {$\calH_1$};
\node at (3,1.3) {$\calH_1'$};
\node at (1.4,-0.2) {$\calH_2 \cap \calH_2'$};
\end{tikzpicture}
\caption{The graphical representation of $\widehat{\calH}$}\label{labelgraphex}
\end{figure}

The advantage of this notation is that the decompositions corresponding to $\calH(x,y)$, $\calH(x,z)$ and $\calH(x,yz)$ can be read off from the graph.  Each of these Bestvina--Margalit tori corresponds to the decomposition given by taking the span of the union of two subgroups not equal to $\calH_0$.  We will use this notation in Section~\ref{bmfinquotsection} to keep track of certain decompositions of $[a]^{\perp} \subseteq H_1(S_g;\Z)$.

\p{The action of $\Sp([a]^{\perp}, \Z)$ on $X_g$}  Throughout the rest of the paper, we will consider the action of $\Sp([a]^{\perp}, \Z)$ on $X_g$.  The subgroup $\Gamma_a \subseteq \Mod(S_g)$ consisting of all elements that act trivially on the class $[a] \in H_1(S_g;\Z)$ acts on $\cC_{[a]}(S_g)$.  We have $\cI_g \subseteq \Gamma_a$ by definition, and so the quotient $\Gamma_a/\cI_g = \Sp([a]^{\perp},\Z)$ acts on $X_g$.  

\section{Abelian cycles in $H_2(\cI_g;\Z)$}\label{abelcyclesection}

The goal of this section is to prove Proposition~\ref{abcycleprop}, which says that certain subspaces of $H_2(\cI_g;\Q)$ generated by abelian cycles of bounding pair maps can be generated by abelian cycles supported on finitely many subsurfaces.  

\p{Abelian cycles} Let $G$ be a group and ${c_1,\ldots, c_k} \subseteq G$ be a collection of pairwise commuting elements of infinite order.  Let $A = \langle c_1,\ldots, c_k\rangle \subseteq G$.  The abelian cycle $\ab(c_1,\ldots, c_k) \in H_k(G;\Q)$ is the image of the class $[c_1] \wedge \ldots \wedge [c_k] \in H_k(A;\Q)$ under the pushforward map $H_k(A;\Q) \rightarrow H_k(G;\Q)$.  If $k = 2$, we will denote the abelian cycle $\ab(c_1,c_2)$ by $[c_1, c_2]$.  If $G$ is a group, we let $H_k^{\ab}(G;\Q)$ denote the subspace of $H_k(G;\Q)$ generated by abelian cycles.

\p{Torelli group of a subsurface} Let $S$ be a compact surface and $\iota: S \hookrightarrow S_g$ be an embedding.  We will denote the genus of $S$ by $g(S)$.  We say that $\iota$ is \textit{clean} if no connected component of $S_g \setminus \iota(S)$ is a disk.  Let $\cI_g(S, S_g)$ denote the intersection $\iota_*(\Mod(S)) \cap \cI_g$, where the pushforward map $\iota_*$ is given by extending $\varphi \in \Mod(S)$ by the identity over $S_g \setminus \iota(S)$.  Let $\cI(\iota)$ denote the inclusion map $\cI(S, S_g) \rightarrow \cI_g$.

\p{Bounding pair maps} If $c \subseteq S_g$ is a curve, let $T_c$ denote the Dehn twist along $c$.  We say that $f \in \cI_g$ is a \textit{bounding pair map} if $f = T_{c}T_{c'}^{-1}$ for $c$ and $c'$ two disjoint homologous curves.  We will denote the bounding pair map $T_{c}T_{c'}^{-1}$ by $T_{c,c'}$.  Let $H_2^{\ab, \bp}(\cI(S, S_g);\Q)$ denote the subspace of $H_2(\cI(S,S_g);\Q)$ generated by abelian cycles $[T_{c,c'}, T_{d,d'}]$ such that the bounding pairs $c \cup c'$ and $d \cup d'$ are disjoint.

The main goal of the section is to prove the following result.

\begin{prop}\label{abcycleprop}
Let $g \geq \finalbound$ and let $M \subseteq S_g$ be a nonseparating multicurve with $\left|\pi_0(M) \right| =9$.  Let $\calV = \{[c]: c \in \pi_0(M)\}$.  The natural map
\begin{displaymath}
\bigoplus_{v \in \calV} H_2^{\ab, \bp}(\cI_g;\Q)^{T_v} \rightarrow H_2^{\ab,\bp}(\cI_g;\Q)
\end{displaymath}
\bn is surjective.
\end{prop}

\bn The idea of Proposition~\ref{abcycleprop} is that if there are enough curves on a surface $S_g$, then any abelian cycle of bounding pairs is a linear combination of abelian cycles that are each disjoint from at least one of the curves. 

\p{Outline of the proof of Proposition~\ref{abcycleprop}} If $G$ is a group acting on a vector space $V$ and $\calF \subseteq G$ is a subset, we let $V^{\calF}$ denote the subspace of $V$ fixed by all elements of $\calF$.  If $\calW \subseteq H_1(S_g;\Z)$, let $T_{\calW}$ denote the set of transvections along elements of $\calW$.  If $F \in \cI_g$, let $[F]$ denote the corresponding class in $H_1(\cI_g;\Q)$.  Let $[T_{d,d'}, T_{e,e'}] \in H_2^{\ab, \bp}(\cI_g;\Q)$ be an abelian cycle.  The proof of the proposition proceeds in two steps.
\begin{enumerate}
    \item There is a relation
    \begin{displaymath}
        [T_{d,d'}, T_{e,e'}] = \sum_{i=1}^k \lambda_i[T_{d,d'}, T_{f_i,f'_i}]
    \end{displaymath}
    \bn with $[T_{f_i, f'_i}] \in H_1(\cI_g;\Q)^{T_{\calV'}}$ for some $\calV' \subseteq \calV$ and $\left|\calV'\right| = 4$ and for all $1 \leq i \leq k$.
    \item If $[T_{d,d'}, T_{f_i,f'_i}]$ is an abelian cycle with $T_{f_i, f'_i} \in H_1(\cI_g;\Q)^{T_{\calV'}}$ for some $\calV' \subseteq V$ and $\left|\calV'\right| = 4$, then there is a relation
    \begin{displaymath}
    [T_{d,d'}, T_{f_i, f'_i}] = \sum_{j=1}^m \theta(\calV)_j [T_{h_j, h'_j}, T_{f_i, f'_i}]
    \end{displaymath}
    \bn with each $[T_{h_j, h'_j}] \in H_1(\cI_g;\Q)^{T_v}$ for some $v \in \calV'.$
\end{enumerate}

\bn Step (1) is the content of Lemma~\ref{abelcyclestabfourlemma} and Step (2) is the content of Lemma~\ref{abcyclefourstablemma}.  

\p{Outline of the section} We will begin by describing the Johnson homomorphism.  We will then prove Lemma~\ref{firsthomolratimlemma}, which describes the vector space $H_1(\cI(S, S_g);\Q)$ for $S \hookrightarrow S_g$ a clean embedding with $g(S) \geq 3$.  We then prove Lemma~\ref{unimodwedgethreelemma} which is a result about generating sets of exterior powers of quasi--unimodular lattices, and then use Lemma~\ref{unimodwedgethreelemma} to prove Lemma~\ref{abcyclefourstablemma}.  We also use Lemma~\ref{unimodwedgethreelemma} to prove Lemma~\ref{unimodwedgethreept2lemma}, which is another result about generating sets of exterior products of vector spaces.  We use Lemma~\ref{unimodwedgethreept2lemma} to prove Lemma~\ref{abelcyclestabfourlemma}. We then combine these lemmas to prove Proposition~\ref{abcycleprop}.

\p{The Johnson homomorphism} Let $a_1,b_1,\ldots, a_g, b_g$ be a symplectic basis for $H_1(S_g;\Z)$ and let $\omega = a_1 \wedge b_1 + \ldots + a_g \wedge b_g \in \midwedge^2 H_1(S_g;\Z)$.  There is an inclusion $H_1(S_g;\Z) \hookrightarrow \midwedge^3H_1(S_g;\Z)$ given by $[c] \rightarrow [c] \wedge \omega$.  Johnson~\cite{Johnsonhomomorphism} constructed a map 
\begin{displaymath}
    \tau_g: \cI_g \rightarrow \midwedge^3H_1(S_g;\Z)/H_1(S_g;\Z)
\end{displaymath}
\bn and showed that the pushforward map
\begin{displaymath}
    (\tau_g)_*: H_1(\cI_g;\Q) \rightarrow \midwedge^3H_1(S_g;\Q)/H_1(S_g;\Q)
\end{displaymath}
\bn is an isomorphism~\cite{JohnsonIII}.  We will make use of Lemma~\ref{bpjohnsoncomp}, which describes $\tau_g(T_{c,c'})$ for $T_{c,c'}$ a bounding pair map.  The lemma is due to Johnson~\cite[Lemma 4B]{Johnsonhomomorphism}.

\begin{lemma}\label{bpjohnsoncomp}
Let $g \geq 3$.  Let $T_{c,c'}$ be a bounding pair map and let $S$ be a connected component of $S_g \cut \left(c \cup c' \right)$.  Let $a_1,b_1,\ldots, a_h,b_h$ be a symplectic basis for a maximal nondegenerate subspace of $H_1(S;\Z)$.  We have
\begin{displaymath}
   \tau_g(T_{c,c'}) =  [c]\wedge \left(\sum_{i=1}^{h}a_i \wedge b_i\right).
\end{displaymath}
\end{lemma}

\bn We use Lemma~\ref{bpjohnsoncomp} to prove Lemma~\ref{firsthomolratimlemma}, which describes $H_1(\cI(S;S_g);\Z)$ for clean embeddings $S \hookrightarrow S_g$.  We first prove an auxiliary lemma, which is essentially just a corollary of a result of Putman~\cite[Theorem 1.3]{Putmantorelli}.

\begin{lemma}\label{putmancorrlemma}
Let $g \geq 4$ and let $\iota:S \hookrightarrow S_g$ be a clean embedding such that $g(S) \geq 3$.  Then $\cI(S, S_g)$ is generated by bounding pair maps.
\end{lemma}

\begin{proof}
A result of Putman~\cite[Theorem 1.3]{Putmantorelli} tells us that $\cI(S, S_g)$ is generated by bounding pair maps and Dehn twists along separating curves.  Hence it suffices to show that if $\delta \subseteq \iota(S)$ is a separating curve in $S_g$, then $T_{\delta}$ is a product of bounding pair maps.  Since $g(S) \geq 3$, there is an embedded copy of $S_1^2 \subseteq S$ such that the image of one boundary component of $S_1^2$ is the curve $\delta$.  Then by embedding a copy of $S_0^4 \hookrightarrow S_1^2$, we can use the lantern relation~\cite[Section 5.1]{FarbMarg} to rewrite $T_{\delta}$ as a product of three bounding pair maps, so the lemma follows.
\end{proof}

\bn We also need one more fact about quasi--unimodular lattices.

\begin{lemma}\label{wedgegenlemma}
Let $L$ be a quasi--unimodular lattice such that $g(L) \geq 1$.  The free abelian group $\midwedge^2 L$ is spanned by elements of the form $\gamma \wedge \delta$ with $\gamma, \delta$ primitive and $\langle \gamma, \delta\rangle = 1$.
\end{lemma}

\begin{proof}
Let $\calS \subseteq \midwedge^2L$ be the set of all elements $\gamma \wedge \delta$ as in the statement of the lemma.  Let 
\begin{displaymath}
\calB = \{\gamma_1,\delta_1, \gamma_2,\delta_2,\ldots, \gamma_{g(L)}, \delta_{g(L)}, \gamma_{g(L)+1}, \gamma_{g(L) + 2}, \ldots, \gamma_{\ell}\}
\end{displaymath}
\bn be a partial symplectic basis for $L$.  The group $\midwedge^2 L$ is generated by all elements of the form $r \wedge s$ for $r,s \in \calB$, so it suffices to show that any element $r \wedge s$ with $r,s \in \calB$ is a $\Z$--linear combination of elements as in the statement of the lemma.

We begin by proving this in the case $r = \gamma_i, s = \delta_j$.  We have $\gamma_i \wedge \delta_j = (\gamma_i + \gamma_j) \wedge \delta_j - (\gamma_j \wedge \delta_j)$.  Each of these latter elements is a wedge of primitive elements with algebraic intersection number one, so $r \wedge s$ is in the subgroup of $\midwedge^2 L$ generated by $\calS$.  

Now, suppose that $r \in \calB, s = \delta_i $ with $r \neq \gamma_i$.  We have $r \wedge s =  (r + \gamma_i) \wedge \delta_i - \gamma_i \wedge \delta_i$, and each of these latter elements are in $\calS$, since the only element of $\calB$ that $\delta_i$ intersects nontrivially is $\gamma_i$.  

Finally, suppose that $r = \gamma_i, s = \gamma_j$.  If $i$ or $j$ are less than $g(L)$, then without loss of generality we have $j \leq g(L)$.   Then we have $r \wedge s = gamma_i + \delta_j \wedge \gamma_j - \delta_j \wedge \gamma_j$, and these latter elements are all in $\calS$.  Otherwise, $i,j > g(L)$.  In this case, we have 
\begin{displaymath}
r \wedge s = (\gamma_i + \delta_{g(L)}) \wedge (\gamma_j + \gamma_{g(L)}) - \delta_{g(L)} \wedge \gamma_j - \delta_{g(L)} \wedge \gamma_j - \delta_{g(L)} - \gamma_{g(L)}.
\end{displaymath}
\bn The first and less elements are in $\calS$, and the middle two elements are $\Z$--linear combinations of elements in $\calS$ by the previous cases, so $r \wedge s \in \Span_{\Z}(\calS)$ and thus the lemma holds.
\end{proof}

\bn We now describe $H(\cI(S,S_g);\Q)$ for certain clean embeddings $\iota:S \hookrightarrow S_g$.

\begin{lemma}\label{firsthomolratimlemma}
Let $S$ be a compact surface with boundary with $\left|\pi_0(\partial S)\right| \leq 2$.  Let $\iota:S \hookrightarrow S_g$ be a clean embedding such that $S_g \cut S$ is connected. Suppose additionally that $g(S_g \cut S) \geq 1$.  Then there is a surjection
\begin{displaymath}
f:H_1(\cI_g(S, S_g);\Q) \twoheadrightarrow \midwedge^3 H_1(S;\Q)
\end{displaymath}
\bn such that the following diagram commutes:

\begin{center}
\begin{tikzcd}
H_1(\cI(S,S_g);\Q) \arrow[r,"f"] \arrow[d,"\cI(\iota)_*"] & \wedge^3 H_1(S;\Q) \arrow[d, "\wedge^3 \iota_*"] \\
H_1(\cI_g;\Q) \arrow[r, "\tau_g \otimes \Q"]& \wedge^3H_1(S_g;\Q)/H_1(S_g;\Q).
\end{tikzcd}
\end{center} 

\bn Furthermore, suppose that $g(S) \geq 3$.  Then $f$ is an isomorphism, and the map $\wedge^3 \iota_* \circ f$ is injective.
\end{lemma}

\begin{proof} We begin with the first part of the lemma.  We show that the image of the composition the composition
\begin{displaymath}
H_1(\cI_g(S,S_g);\Q) \xrightarrow{\cI(\iota)_*} H_1(\cI_g;\Q) \xrightarrow{\tau_g \otimes \Q} \midwedge^3H_1(S_g;\Q)/H_1(S_g;\Q)
\end{displaymath}
\bn is $\im(\midwedge^3 \iota_*)$, and that $\midwedge^3 \iota_*$ is an injection.  If these two statements hold, then $f$ is given by
\begin{displaymath}
f = (\midwedge^3 \iota_*)^{-1} |_{\im(\tau_g \otimes \Q \circ \cI(\iota)_*} \circ \tau_g \otimes \Q \circ \cI(\iota)_*.
\end{displaymath}
\bn We prove each containment in turn, and then prove injectivity.

\p{Showing that $\im(\tau_g \otimes \Q \circ \cI(\iota)_*) \subseteq  \im(\midwedge^3 H_1(S;\Q) \rightarrow \midwedge^3 H_1(S_g;\Q)/H_1(S_g;\Q))$} Lemma~\ref{putmancorrlemma} says that $\cI_g(S,S_g)$ is generated by bounding pair maps, so it suffices to show that if $T_{c,c'} \in \cI_g(S, S_g)$ is a bounding pair map, then $\tau_g \otimes \Q \circ \cI(\iota)_*(T_{c,c'}) \in \im(\midwedge^3 H_1(S;\Q) \rightarrow \midwedge^3 H_1(S_g;\Q)/H_1(S_g;\Q))$.  Since the bounding pair $c \cup c'$ is contained in $S$, we have $[c] \in \midwedge^3 H_1(S;\Q)$.  Then there must be a connected component $S'$ of $S_g \cut \left(c \cup c'\right)$ such that $S' \subseteq S$.  Let $\alpha_1,\beta_1,\ldots, \alpha_h, \beta_h$ be a symplectic basis for a maximal unimodular subgroup of $H_1(S';\Z)$.  Lemma~\ref{bpjohnsoncomp} says that $\tau_g \circ \cI(\iota)_*([T_{c,c'}]) = [c] \wedge \left(\alpha_1 \wedge \beta_1 + \ldots + \alpha_h \wedge \beta_h\right) \in \im(\midwedge^3 H_1(S;\Q) \rightarrow \midwedge^3 H_1(S_g;\Q)/H_1(S_g;\Q))$, so the $\subseteq$ containment holds.

\p{Showing that $\im(\tau_g \circ \cI(\iota)_*) \supseteq \im(\midwedge^3 H_1(S;\Q) \rightarrow \midwedge^3 H_1(S_g;\Q)/H_1(S_g;\Q))$} Observe that if $v_1, v_2, v_3 \in H_1(S;\Z)$ are primitive classes with $\langle v_1, v_2 \rangle = 1$ and $\langle v_1, v_3 \rangle = \langle v_2, v_3 \rangle = 0$, then Lemma~\ref{bpjohnsoncomp} says that if $c,c'$ is a bounding pair with $[c] = v_3$ and $v_1,v_2$ a symplectic basis for a maximal nondegenerate subgroup of the first homology of a connected component of $S \cut \left(c \cup c'\right)$, then $\tau_g(T_{c,c'}) = v_1 \wedge v_2 \wedge v_3 \in \im(\tau_g \circ \cI(\iota)_*)$.  Since the set of such triples of elements spans $\midwedge^3 H_1(S;\Z)$ by Lemma~\ref{wedgegenlemma}, we have $\im(\tau_g \circ \cI(\iota)_*) \supseteq \im(\midwedge^3 H_1(S;\Q) \rightarrow \midwedge^3 H_1(S_g;\Q)/H_1(S_g;\Q))$.  

\p{Showing that $\midwedge^3H_1(S;\Q)$ injects into $\midwedge^3 H_1(S_g;\Q)/H_1(S_g;\Q)$} We first show that the exterior power of pushforwards $\midwedge^3 H_1(S;\Q) \rightarrow \midwedge^3 H_1(S_g;\Q)$ is injective.  Since $S_g \cut S$ is connected, the pushforward $H_1(S;\Q) \rightarrow H_1(S_g;\Q)$ is injective, so $\midwedge^3 H_1(S;\Q) \rightarrow \midwedge^3 H_1(S_g;\Q)$ is injective as well.  Hence it suffices to show that $\im(\midwedge^3 H_1(S;\Q) \rightarrow \midwedge^3 H_1(S_g;\Q)) \cap H_1(S_g;\Q) \wedge \omega = 0$, where $\omega \in \midwedge^2 H_1(S_g;\Q)$ is the characteristic element.  Suppose otherwise, so there is some nonzero $\alpha \wedge \omega \in \im(\midwedge^3 H_1(S;\Q) \rightarrow \midwedge^3 H_1(S_g;\Q))$.  Since $\omega$ does not depend on the choice of symplectic basis, we can choose a symplectic basis $\alpha_1, \beta_1,\ldots, \alpha_g, \beta_g$ with $\alpha \wedge \omega = m \cdot \alpha_1 \wedge(\alpha_1 \wedge \beta_1 + \alpha_2 \wedge \beta_2 + \ldots + \alpha_g \wedge \beta_g) = m \cdot \alpha_1 \wedge (\alpha_2 \wedge \beta_2 + \ldots + \alpha_g \wedge \beta_g)$ for $m$ some positive integer.  But this implies that there are $(2g-1)$ linearly independent elements $\alpha_1, \alpha_2, \beta_2, \ldots, \alpha_g, \beta_g \in \im(H_1(S;\Q) \rightarrow H_1(S_g;\Q))$.  We have assumed that $g(S_g \cut S) \geq 1$, so $\dim(H_1(S;\Q)) \leq 2g - 2$, which is a contradiction.

We now prove the second part of the lemma, where we assume that $g(S) \geq 3$.  In this case, since $g(S) \geq 3$ and $\iota:S \hookrightarrow S_g$ is clean, a theorem of Putman~\cite[Theorem B]{Putmanjohnson} states that the pushforward map $\cI(\iota)_*$ is an injection.  Therefore the composition $(\midwedge^3 \iota_*)^{-1} |_{\im(\tau_g \otimes \Q \circ \cI(\iota)_*} \circ \tau_g \otimes \Q \circ \cI(\iota)_*$ is an injection.  By definition this is $f$, so $f$ is injective as well, and is therefore an isomorphism.
\end{proof}

\bn We will also prove Lemma~\ref{dehntransvectionswitch}, which is a result that describes how stabilizers of transvections interact with the abelianization of certain $\cI(S, S_g)$.  We first prove the following result.

\begin{lemma}\label{transvectionstablemma}
Let $g \geq 3$ and let $w \in H_1(S_g;\Z)$ be a nonzero primitive element.  The fixed space $H_1(\cI_g;\Q)^{T_w}$ is sent to $\im(\midwedge^3 [w]^{\perp} \rightarrow \midwedge^3 H_1(S_g;\Q)/H_1(S_g;\Q))$ under the Johnson homomorphism.
\end{lemma}

\begin{proof}
Since the Johnson homomorphism is a map of $\Sp(2g,\Z)$--representations, it suffices to show that $\left(\midwedge^3 H_1(S_g;\Q)/H_1(S_g;\Q)\right)^{T_w} = \im(\midwedge^3 [w]^{\perp} \rightarrow \midwedge^3 H_1(S_g;\Q)/H_1(S_g;\Q))$.  Since $T_w$ acts trivially on $[w]^{\perp}$ by the definition of transvection, the transvection $T_w$ acts trivially on $\im(\midwedge^3 [w]^{\perp} \rightarrow \midwedge^3 H_1(S_g;\Q)/H_1(S_g;\Q))$, and thus $\tau_g(H_1(\cI_g;\Q)^{T_w}) \supseteq \im(\midwedge^3 [w]^{\perp} \rightarrow \midwedge^3 H_1(S_g;\Q)/H_1(S_g;\Q))$.  It therefore remains to prove the other containment.  

Let $\calB = \{\alpha_1,\beta_1,\ldots, \alpha_g, \beta_g\} \in H_1(S_g;\Z)$ be a symplectic basis so that $\alpha_1 = w$.  The vector space $\midwedge^3 H_1(S_g;\Q)$ has a basis consisting of all pure wedge products of elements in $\calB$.  Then after modding out by $H_1(S_g;\Q) \rightarrow \midwedge^3 H_1(S_g;\Q)$, we see that any pure tensor of the form $\alpha_1 \wedge \beta_1 \wedge \gamma$ is a linear combination
\begin{displaymath}
\alpha_2 \wedge \beta_2 \wedge \gamma + \ldots + \alpha_g \wedge \beta_g \wedge \gamma.
\end{displaymath}
\bn Hence $\midwedge^3 H_1(S_g;\Q)/H_1(S_g;\Q)$ has a basis consisting of all pure tensors of elements in $\calB$ except those containing $\alpha_1 \wedge \beta_1$.  This set is spanning in $\midwedge^3 H_1(S_g;\Q)/H_1(S_g;\Q)$, and is therefore a basis by counting dimensions.  Now, $T_w$ acts trivially on any such wedge products except for those containing $\beta_1$.  For such a basis element $\beta_1 \wedge \gamma \wedge \delta$ for $\gamma, \delta \in \calB \setminus \{\alpha_1\}$, we have $T_w(\beta_1 \wedge \gamma \wedge \delta) = \beta_1 \wedge \gamma \wedge \delta + \alpha_1 \wedge \gamma \wedge \delta$.  Now, if we have a $\Q$--linear combination of such elements
\begin{displaymath}
\sum_{i=1}^n \lambda_i \beta_1 \wedge \gamma_i \wedge \delta_i,
\end{displaymath}
\bn then we have
\begin{displaymath}
T_w\left(\sum_{i=1}^n \lambda_i \beta_1 \wedge \gamma_i \wedge \delta_i\right) = \sum_{i=1}^n \lambda_i (\beta_1 + \alpha_1) \wedge \gamma_i \wedge \delta_i = (\beta_1 + \alpha_1) \wedge \sum_{i=1}^n \lambda_i \gamma_i \wedge \delta_i.
\end{displaymath}
\bn  Then to have $\sum_{i=1}^n \lambda_i \beta_1 \wedge \gamma_i \wedge \delta_i$ fixed by $T_w$, we must have
\begin{displaymath}
\sum_{i=1}^n \lambda_i \beta_1 \wedge \gamma_i \wedge \delta_i = (\beta_1 + \alpha_1) \wedge \sum_{i=1}^n \lambda_i \gamma_i \wedge \delta_i,
\end{displaymath}
\bn or equivalently we must have
\begin{displaymath}
\alpha_1 \wedge \sum_{i=1}^n \lambda_i \gamma_i \wedge \delta_i = 0
\end{displaymath}
\bn Since the span of the set $\{\gamma_i, \delta_i: 1 \leq i \leq n\}$ does not contain $\alpha_1$, we conclude that $\sum_{i=1}^n \lambda_i \gamma_i \wedge \delta_i = 0$.  In particular, this implies that only elements of the form $\midwedge^3 [w]^{\perp} \rightarrow \midwedge^3 H_1(S_g;\Q)/H_1(S_g;\Q)$ are fixed by $T_w$, as desired.
\end{proof}

\bn We apply Lemma~\ref{transvectionstablemma} in the following context.

\begin{lemma}\label{dehntransvectionswitch}
Let $g \geq 4$, and let $\iota:S \hookrightarrow S_g$ be a clean embedding satisfying the hypotheses of Lemma~\ref{firsthomolratimlemma}.  Let $M \subseteq S_g$ be a nonseparating multicurve, and let $\calW \subseteq H_1(S_g;\Z)$ be given by $\calW = \{[c]: c\in \pi_0(M)\}.$  Then, after passing to $\Q$--coefficients, the Johnson homomorphism $\tau_g$ restricts to an isomorphism
\begin{displaymath}
\im\left(H_1(\cI(S, S_g);\Q) \rightarrow H_1(\cI_g;\Q)\right)  \cap H_1(\cI_g;\Q)^{T_\calW} \rightarrow \midwedge^3 \im(H_1(S;\Q) \rightarrow H_1(S_g;\Q)) \cap \midwedge^3(\calW^{\perp})
\end{displaymath}
\end{lemma}

\begin{proof}
By Lemma~\ref{firsthomolratimlemma}, we have $\im\left(H_1(\cI(S, S_g);\Q) \rightarrow H_1(\cI_g;\Q)\right) \cong \im(\midwedge^3 \iota_*) \cong \midwedge^3 H_1(S;\Q)$.  Now, the Johnson homomorphism is $\Sp(2g,\Z)$--equivariant, so if $w \in \calW$, the fixed set $H_1(\cI_g;\Q)^{T_w}$ is isomorphic to $\im(\tau_g \otimes \Q)^{T_w}$.  Then by Lemma~\ref{transvectionstablemma}, we have 
\begin{displaymath}
\tau_g(H_1(\cI_g;\Q)^{T_w}) = \im(\midwedge^3 [w]^{\perp} \rightarrow \midwedge^3 H_1(S_g;\Q)/H_1(S_g;\Q)).
\end{displaymath}
\bn Then we have $H_1(\cI_g;\Q)^{T_{\calW}} = \bigcap_{w \in \calW} H_1(\cI_g;\Q)^{T_w}$ and $\midwedge^3 \calW^{\perp} = \bigcap_{w \in \calW} \midwedge^3 [w]^{\perp}$.  Therefore we have $\tau_g(H_1(\cI_g;\Q)^{T_{\calW}}) = \im(\midwedge^3 \calW^{\perp} \rightarrow \midwedge^3 H_1(S_g;\Q)/H_1(S_g;\Q))$.  Hence the intersection
\begin{displaymath}
\im\left(H_1(\cI(S, S_g);\Q) \rightarrow H_1(\cI_g;\Q)\right)  \cap H_1(\cI_g;\Q)^{T_\calW}
\end{displaymath}
\bn is sent to $\im(\midwedge^3 H_1(S;\Q) \rightarrow \im(\tau_g)) \cap \im(\midwedge^3 \calW^{\perp} \rightarrow \im(\tau_g))$ by the Johnson homomorphism, and this last intersection is equal to
\begin{displaymath}
\midwedge^3 (\im(H_1(S;\Q) \rightarrow H_1(S_g;\Q))) \cap \midwedge^3\calW^{\perp}
\end{displaymath}
\bn as desired.
\end{proof}

\bn The remaining steps to do now before proving Proposition~\ref{abcycleprop} are to prove Lemma~\ref{abcyclefourstablemma} and Lemma~\ref{abelcyclestabfourlemma}, which will allow us to rewrite classes $[T_{d,d'}, T_{e,e'}]$ as linear combinations of abelian cycles of bounding pairs maps contained in $H_2(\cI_g;\Q)^{T_v}$ for various primitive $v \in H_1(S_g;\Z)$.

\subsection{The proof of Lemma~\ref{abcyclefourstablemma}}

We now recall a basic fact from linear algebra.  
\begin{lemma}\label{unimodwedgethreelemma}
Let $(L, \langle \cdot, \cdot \rangle)$ be quasi--unimodular lattice.  Let $\calW = \{w_1,\ldots, w_4\} \subseteq L$ be a set of primitive elements such that the image of $\calW$ under the adjoint map $L \rightarrow \ho_{\Z}(L, \Z)$ is a linearly independent set.  Then the natural map
\begin{displaymath}
\varphi: \bigoplus_{w \in \calW} \midwedge^3\left(w^{\perp} \otimes \Q\right) \rightarrow \midwedge^3\left(L \otimes \Q\right)
\end{displaymath}
\bn is surjective.
\end{lemma}

\begin{proof}
Choose a basis $\calB = \{a_1,\ldots, a_n\}$ for $L \otimes \Q$ such that $\langle w_i, a_j\rangle = \delta_{ij}$.  Then the set $\calB_{i} = \{a_1,\ldots, \widehat{a_i},\ldots, a_n\}$ is a basis for $w^{\perp} \otimes \Q$ for each $w \in \calW$.  Now, the vector space $\midwedge^3 (L \otimes \Q)$ has a basis consisting of triples of elements in $\calB$.  Since $\left|\calW\right| \geq 4$, each of these triples lies in some $\calB_{i}$.  Therefore $\im(\varphi)$ contains a basis for $\midwedge^3(L \otimes \Q)$ and hence is surjective. 
\end{proof}

\bn For our purpose, the important takeaway from Lemma~\ref{unimodwedgethreelemma} is the following result.

\begin{lemma}\label{abcyclefourstablemma}
Let $M' \subseteq S_g$ be a nonseparating multicurve with $\left|\pi_0(M')\right| \geq 4$, and let $\calV' = \{[c]: c \in \pi_0(M')\}$.  Let $d,d' \subseteq S_g$ be a bounding pair such that $[T_{d,d'}] \in H_1(\cI_g;\Q)^{T_{\calV'}}$.  Let $[T_{d,d'}, T_{f,f'}] \in H_2^{\ab,\bp}(\cI_g;\Q)$ be an abelian cycle.  There is a $\Q$--linear relation
\begin{displaymath}
[T_{d,d'}, T_{f,f'}] = \sum_{\ell = 1}^m \lambda_{\ell} [T_{d,d'}, T_{f_\ell, f_\ell'}]
\end{displaymath}
\bn in $H_2^{\ab,\bp}(\cI_g;\Q)$ such that, for each $1 \leq \ell \leq m$, there is a $v \in \calV'$ such that $[T_{f_{\ell},f_{\ell}'}] \in H_1(\cI_g;\Q)^{T_v}$.
\end{lemma}

\begin{proof}
Let $S'$ and $S''$ be the two connected components of $S_g \cut \left(d \cup d' \right)$.  Without loss of generality, we may assume that $f,f' \subseteq S'$.  Indeed, if $f,f'$ are in distinct connected components of $S_g \cut \left(d \cup d' \right)$, then $f$ and $f'$ are homologous to $d$, so we have a relation
\begin{displaymath}
[T_{d,d'}, T_{f,f'}] = [T_{d,d'}, T_{f,d}] + [T_{d,d'}, T_{d, f'}]
\end{displaymath}
\bn We now have two cases.

\p{Case 1: $\calV' \not \subseteq \im(H_1(S';\Z) \rightarrow H_1(S_g;\Z))$} In this case, there is a $v \in \calV'$ such that $ v \in \im(H_1( S'';\Z) \rightarrow H_1(S_g;\Z))$.  Then $T_{f,f'} \in H_1(\cI_g;\Z)^{T_v}$, so we are done.

\p{Case 2: $\calV' \subseteq \im(H_1(S';\Q))$} We see that since $\calV'$ is linearly independent and the elements of $\calV'$ have pairwise trivial algebraic intersection, we have $g(S') \geq 4$.  Lemma~\ref{firsthomolratimlemma} says that $H_1(\cI(S', S_g);\Q) \cong \midwedge^3 H_1(S';\Q)$.  Lemma~\ref{unimodwedgethreelemma} implies that there is a surjection
\begin{displaymath}
\bigoplus_{v \in \calV'}\midwedge^3 v^{\perp} \cap \midwedge^3 H_1(S';\Q) \rightarrow \midwedge^3 H_1(S';\Q).
\end{displaymath}
\bn  Then Lemma~\ref{dehntransvectionswitch} applied to each term on the summand on the left implies there is a surjection
\begin{displaymath}
\bigoplus_{v \in \calV'}H_1(\cI_g;\Q)^{T_v} \cap H_1(\cI(S',S_g);\Q) \rightarrow H_1(\cI(S', S_g);\Q).
\end{displaymath}
\bn Hence there is a relation in $H_1(\cI(S', S_g);\Q)$ given by
\begin{displaymath}
[T_{f,f'}] = \sum_{r = 1}^n\lambda_r[F_r]
\end{displaymath}
\bn where each $F_r \in H_1(\cI(S', S_g);\Q)^{T_v}$ for some $v \in \calV'$.  Then each $F_r$ is supported on $S'$, so each $F_r$ commutes with $T_{d,d'}$.  Therefore, there is a relation in $H_2(\cI_g;\Q)$ given by
\begin{displaymath}
[T_{d,d'}, T_{f,f'}] = \sum_{r = 1}^n\lambda_r[T_{d,d'}, F_r]
\end{displaymath}
\bn where each $[F_r] \in H_1(\cI(S',S_g);\Q)^{T_v}$ for some $v \in \calV'$.  

We now show that each $[F_r]$ can be written as a sum of classes represented by bounding pair maps.  Choose an $r$ with $1 \leq r \leq n$ and let $v \in \calV'$ such that $[F_r] \in H_1(S';\Q)^{T_v}$. Lemma~\ref{dehntransvectionswitch} says that $\midwedge^3 H_1(\cI(S',S_g);\Q)^{T_v} \cong\midwedge^3 v^{\perp} \cap \midwedge^3 H_1(S';\Q)$, and this isomorphism is induced by restricting the Johnson homomorphism.  Then Lemma \ref{bpjohnsoncomp} and Lemma \ref{wedgegenlemma} combine to tell us that $H_1(\cI(S' \cut c, S_g);\Q)^{T_v}$ is generated by classes represented by bounding pair maps, since each element in the basis for $\midwedge^3 v^{\perp} \cap H_1(S';\Q)$ in Lemma \ref{wedgegenlemma} is the image under $\tau_g$ of a bounding pair map.  Therefore we may rewrite $[F_r]$ as a linear combination of classes $[T_{f_{1,r}, f_{1,r}'}], \ldots, [T_{f_{\ell_r,r}, f_{m_r, r}'}] \in H_1(\cI(S', S_g);\Q)^{T_v}$.  Since each $T_{f_{\ell, r}, f'_{\ell, r}}$ is supported on $S'$, each $T_{f_{\ell, r}, f'_{\ell, r}}$ commutes with $T_{d,d'}$.  Therefore we have a relation
\begin{displaymath}
[T_{d,d'}, T_{f,f'}] = \sum_{r = 1}^n \lambda_r \sum_{\ell = 1}^{m_r} [T_{d,d'}, T_{f_{\ell, r}, f'_{\ell, r}}]
\end{displaymath}
\bn which by forgetting the $r$ index yields a relation
\begin{displaymath}
[T_{d,d'}, T_{f,f'}] = \sum_{\ell = 1}^m \lambda_{\ell} [T_{d,d'}, T_{f_\ell, f_\ell'}].
\end{displaymath}
\bn Now, since each $T_{f_{\ell}, f_\ell'}$ is stabilized by some $T_v \in T_{\calV'}$, we have each $[T_{f_{\ell}, f'_{\ell}}]\in H_1(\cI_g;\Q)^{T_v}$ for some $v \in \calV'$, as desired.
\end{proof}

\subsection{The proof of Lemma~\ref{abelcyclestabfourlemma}} 

\bn We begin by extending Lemma~\ref{unimodwedgethreelemma} into the following.

\begin{lemma}\label{unimodwedgethreept2lemma}
Let $L$ be a quasi--unimodular lattice.  Let $\calW = \{w_1,\ldots, w_7\}$ be a set of primitive elements in $L$ such that the image of $\calW$ under the adjoint map $L \rightarrow \ho_{\Z}(L,\Z)$ is $\Q$--linearly independent.  Then the natural map
\begin{displaymath}
\varphi: \bigoplus_{\calW' \subseteq\calW: \left|\calW' \right| = 4} \midwedge^3\left(\left(\calW'\right)^{\perp} \otimes \Q\right) \rightarrow \midwedge^3\left(L \otimes \Q\right)
\end{displaymath}
\bn is surjective.
\end{lemma}

\begin{proof}
We will prove inductively that the natural maps
\begin{displaymath}
    \varphi_{k}: \bigoplus_{\calW' \subseteq\calW: \left|\calW' \right| = k} \midwedge^3\left(\left(\calW'\right)^{\perp} \otimes \Q\right) \rightarrow \midwedge^3\left(L \otimes \Q\right)
\end{displaymath}
\bn are surjective for $k \leq 4$.

\p{Base case: $k = 1$} This follows by applying Lemma~\ref{unimodwedgethreelemma} with any subset $\calW' \subseteq \calW$ with $\left|\calW'\right| = 4$.  

\p{Inductive step: $2 \leq k \leq 4$}  Assume that the map $\varphi_{k'}$ is surjective for all $1 \leq k' < k$.  We will show that $\varphi_k$ is surjective as well.  Since $\varphi_{k'}$ is surjective, it suffices to show that, for any $\calW'' \subseteq \calW$ with $\left|\calW''\right| = k-1$, we have
\begin{displaymath}
    \im(\varphi_{k}) \supseteq \im\left(\midwedge^3\left(\calW''\right)^{\perp}\otimes \Q \rightarrow \midwedge^3 L \otimes \Q \right). 
\end{displaymath}
\bn Since $k \leq 4$ and $\left|\calW''\right| < k$, we have $\left|\calW \setminus\calW''\right| \geq 7-k > 7-4 = 3$, so $\left|\calW \setminus \calW'' \right| \geq 4$.  Hence the natural map
\begin{displaymath}
    \varphi_{\calW''}: \bigoplus_{w \in \calW \setminus \calW''}\midwedge^3\left(\calW'' \cup \{w\}\right)^{\perp} \otimes \Q \rightarrow \midwedge^3\left(\calW''\right)^{\perp} \otimes \Q
\end{displaymath}
\bn is surjective by Lemma~\ref{unimodwedgethreelemma}.  But then $\im(\varphi_{k'} \circ \varphi_{\calW''}) \subseteq \im(\varphi_k)$ for any $\calW'' \subseteq \calW$, so $\im(\varphi_{k'}) \subseteq \im(\varphi_k)$.  Since $\im(\varphi_{k'}) = \midwedge^3 L \otimes \Q$ by the inductive hypothesis, the proof is complete.
\end{proof}

\bn We use Lemma~\ref{unimodwedgethreept2lemma} to prove the following result about $H_2^{\ab,\bp}(\cI_g;\Q)$.

\begin{lemma}\label{abelcyclestabfourlemma}
Let $g \geq \finalbound$ and let $[T_{d,d'}, T_{e,e'}] \in H_2^{\ab,\bp}(\cI_g;\Q)$ be an abelian cycle.  Assume that one connected component of $S_g \cut \left(d \cup d'\right)$ has genus one.  Let $M \subseteq S_g$ be a nonseparating multicurve with $\left|\pi_0(M)\right| = 9$ and let $\calV = \{[c]: c \in \pi_0(M)\}$.  Then there is a linear relation
\begin{displaymath}
        [T_{d,d'}, T_{e,e'}] = \sum_{i=1}^k \lambda_i[T_{d,d'}, T_{f_i,f'_i}]
    \end{displaymath}
    \bn with each $T_{f_i, f'_i} \in H_1(\cI_g;\Q)^{T_{\calV'}}$ for some $\calV' \subseteq \calV$ and $\left|\calV'\right| = 4.$
\end{lemma}

\begin{proof}
The proof follows the same approach as Lemma~\ref{abcyclefourstablemma}.  Let $S \sqcup S = S_g \cut \left(d \cup d' \right)$ with $g(S) = 1$.  Since $e \cup e'$ is disjoint from $d \cup d'$, we have $e, e' \subseteq S_1$.  Since $g(S) = 1$, we must have $e,e' \subseteq S'$.  Now, choose a primitive element $[b] \in H_1(S_g;\Z)$ such that $[b]\in [e]^{\perp}$ and $\langle [b], [d] \rangle = 1$.  Since $g(S) = 1$ and $\langle [b], [d] \rangle = 1$, we have $g(H_1(S';\Z) \cap [b]^{\perp}) = g - 2$.  Then since the elements of $\calV$ have pairwise trivial algebraic intersection, the projection of $\calV$ to $H_1(S';\Z) \cap [b]^{\perp}$ must contain a set of at least 7 elements whose image under the adjoint map $H_1(S';\Z) \cap [b]^{\perp} \rightarrow \ho_{\Z}(H_1(S';\Z), \Z)$ is $\Q$--linearly independent.  Let $\overline{\calV}$ be the image of the projection of $\calV$ to $H_1(S';\Z) \cap [b]^{\perp}$.  Let $\calB \subseteq \overline{\calV}$ be a set of at least 7 linearly independent elements.  By applying Lemma~\ref{unimodwedgethreept2lemma} to $H_1(S';\Q)$ and $\calB$, there is a surjection

\begin{displaymath}
\bigoplus_{\calB' \subseteq \calB:\left|\calB'\right| = 4} \midwedge^3 \left(\calB'\right)^{\perp} \rightarrow \midwedge^3 H_1(S';\Q).
\end{displaymath}
\bn Let $\calV_{\calB}$ be the preimage of $\calB$ under the projection map $\calV \rightarrow \overline{\calV}$.  Since for any $v \in \calV$, we have $v^{\perp} \cap H_1(S';\Q) = \left(\proj_{H_1(S';\Q) \cap [b]^{\perp}}v\right)^{\perp} \cap H_1(S';\Q)$, we therefore have a surjection
\begin{displaymath}
\bigoplus_{\calV' \subseteq \calV_{\calB}:\left|\calV'\right| = 4} \midwedge^3 \left(\calV'\right)^{\perp} \cap H_1(S';\Q) \rightarrow \midwedge^3 H_1(S';\Q).
\end{displaymath}
\bn By construction, the complement $S_g \cut S'$ is connected and has $g(S_g \cut S') \geq 1$.  Hence by Lemma~\ref{firsthomolratimlemma}, there is an isomorphism $H_1(\cI(S', S_g);\Q) \cong \midwedge^3H_1(S';\Q)$.  Furthermore, by applying Lemma~\ref{dehntransvectionswitch}, we have $\midwedge^3 \left(\calV'\right)^{\perp} \cap H_1(S';\Q) \cong  \midwedge^3 H_1(S';\Q)^{T_{\calV'}}$.  Then, since $T_{e,e'} \in \cI(S', S_g)$, we have a relation in $H_1(\cI(S', S_g);\Q)$ given by
\begin{displaymath}
[T_{e,e'}] = \sum_{i=1}^k \lambda_i [F_i]
\end{displaymath}
\bn where each $[F_i] \in H_1(\cI_g;\Q)^{T_{\calV'}}$ for some $\calV' \subseteq \calV_{\calB}$ with $\left|\calV'\right| = 4$.  Since any $f \in \cI(S', S_g)$ commutes with $T_{d,d'}$, we have a relation in $H_2(\cI_g;\Q)$ given by
\begin{displaymath}
[T_{d,d'}, T_{e,e'}] = \sum_{i=1}^k \lambda_i[T_{d,d'}, F_i]
\end{displaymath}
\bn where each $[F_i] \in H_1(\cI(S', S_g);\Q)^{T_{\calV'}}$ for some $\calV' \subseteq \calB$ with $\left|\calV'\right| = 4$.  By Lemma~\ref{dehntransvectionswitch}, we have an isomorphism $H_1(\cI(S',S_g);\Q)^{T_{\calV'}} \cong \midwedge^3 \im(H_1(S';\Q) \rightarrow H_1(S_g;\Q)) \cap \midwedge^3 (\calV')^{\perp}$.  Then Lemma~\ref{wedgegenlemma} gives us a spanning set for the vector space $\midwedge^3 \im(H_1(S';\Q) \rightarrow H_1(S_g;\Q)) \cap \midwedge^3 (\calV')^{\perp}$, and by Lemma~\ref{bpjohnsoncomp} each element of this spanning set is given by applying $\tau_g$ to a bounding pair map supported on a bounding pair $f \cup f' \subseteq S'$.  Hence there is a relation in $H_2(\cI_g;\Q)$ given by
\begin{displaymath}
        [T_{d,d'}, T_{e,e'}] = \sum_{i=1}^k \lambda_i[T_{d,d'}, T_{f_i,f'_i}]
    \end{displaymath}
    \bn such that each $T_{f_i, f'_i} \in H_1(\cI_g;\Q)^{T_{\calV'}}$ for some $\calV' \subseteq \calV$ with $\left|\calV'\right| = 4$, so the lemma holds.
\end{proof}

\bn We will need one more auxiliary result, which well help us apply Lemma~\ref{abelcyclestabfourlemma}.

\begin{lemma}\label{genusoneabelcyclelemma}
Let $g \geq \finalbound$ and let $[T_{d,d'}, T_{e,e'}] \in H_2^{\ab,\bp}(\cI_g;\Q)$.  Then there is a linear relation
\begin{displaymath}
[T_{d,d'}, T_{e,e'}] = \sum_{j = 1}^r [T_{d_j, d'_j}, T_{e,e'}]
\end{displaymath}
\bn in $H_2^{\ab, \bp}(\cI_g;\Q)$ such that , for each $1 \leq j \leq r$, at least one connected component of $S_g \cut \left(d_j \cup d'_j\right)$ has genus one.
\end{lemma}

\begin{proof}
Let $S \cup S'$ be the connected components of $S_g \cut (d \cup d')$.  If $e, e'$ are contained in different connected components of $S_g \cut \left(d \cup d' \right)$, we have a relation $T_{e,e'} = T_{e,d} T_{d,e'}$.  Hence there is a relation
\begin{displaymath}
    [T_{d,d'},T_{e,e'}] = [T_{d,d'}, T_{e,d}] + [T_{d,d'}, T_{d,e'}]
\end{displaymath}
\bn in $H_2^{\ab, \bp}(\cI_g;\Q)$.  Therefore without loss of generality, we may assume that $e,e' \subseteq S$.  Now, let $d_0, \ldots,d_k \subseteq S'$ be curves such that the following hold:
\begin{itemize}
    \item $d_0 = d$,
    \item $d_k = d'$, 
    \item $d_i$ and $d_j $ are disjoint for all $0 \leq i < j \leq k$,
    \item $d_i$ is homologous to $d$ for all $0 \leq i \leq k$, and
    \item the connected components of $S' \cut \left(\bigcup_{0 \leq i \leq k} d_i \right)$ all have genus one.
\end{itemize}
 \bn By construction, we have $T_{d,d'} = \prod_{0 \leq i \leq k - 1}T_{d_i, d_{i+1}}$.  Furthermore, each bounding pair $d_i \cup d_{i+1}$ is contained in $S'$, and hence each $T_{d_i, d_{i+1}}$ commutes with $T_{e,e'}$.  Therefore there is a relation in $H_2^{\ab, \bp}(\cI_g;\Q)$ given by
 \begin{displaymath}
     [T_{d,d'}, T_{e,e'}] = \sum_{i=0}^{k-1}[T_{d_i, d_{i+1}}, T_{e,e'}],
 \end{displaymath}
\bn so the lemma holds.
\end{proof}

\bn We now conclude Section~\ref{abelcyclesection}.  

\begin{proof}[Proof of Proposition~\ref{abcycleprop}]
Let $[T_{d,d'}, T_{e,e'}] \in H_2^{\ab, \bp}(\cI_g;\Q))$ be an abelian cycle.  The proof of the proposition proceeds in two steps.
\begin{enumerate}
    \item There is a relation
    \begin{displaymath}
        [T_{d,d'}, T_{e,e'}] = \sum_{i=1}^k \lambda_i[T_{d,d'}, T_{f_i,f'_i}]
    \end{displaymath}
    \bn with each $[T_{f_i, f'_i}] \in H_1(\cI_g;\Q)^{T_{\calV'}}$ with $\calV' \subseteq \calV$ and $\left|\calV'\right| = 4.$
    \item If $[T_{d,d'}, T_{f_i,f'_i}]$ is an abelian cycle with $T_{f_i, f'_i} \in H_1(\cI_g;\Q)^{T_{\calV'}}$ with $\calV' \subseteq V$ and $\left|\calV'\right| = 4$, then there is a relation
    \begin{displaymath}
    [T_{d,d'}, T_{f_i, f'_i}] = \sum_{j=1}^m \lambda_j [T_{h_j, h'_j}, T_{f_i, f'_i}]
    \end{displaymath}
    \bn with each $[T_{h_j, h'_j}] \in H_1(\cI_g;\Q)^{T_v}$ for some $v \in \calV'.$
\end{enumerate}

\bn Lemma~\ref{genusoneabelcyclelemma} says that we may assume without loss of generality that at least one connected component of $S_g \cut(d \cup d')$ has genus one.  Then Step (1) is the content of Lemma~\ref{abelcyclestabfourlemma} and Step (2) is the content of Lemma~\ref{abcyclefourstablemma}.
\end{proof}

\section{Finiteness of coinvariants in $\BM_2(X_g;\Q)$}\label{bmfinquotsection}

The main goal of this section is to prove the following result. 
\begin{lemma}\label{bmfinquotlemma}
Let $g \geq \finalbound$ and $a \subseteq S_g$ a nonseparating curve.  Let $[b] \in H_1(S_g;\Z)$ be a homology class with $\langle [a], [b] \rangle = 1$.  Let $b \subseteq S_g$ be a representative of $[b]$ with geometric intersection $\left|a \cap b \right| = 1$.  Let $M \subseteq S_g \cut (a \cup b)$ be a nonseparating multicurve with $\left|\pi_0(M)\right| = 9$.  Let $G$ be the image of the map $\Mod(S_g \cut \left(a \cup b \cup M\right)) \rightarrow \Sp(2g,\Z)$.  The vector space
\begin{displaymath}
H_0(G;\BM_2(X_g;\Q))
\end{displaymath}
\bn is finite dimensional.
\end{lemma}  
\p{The outline of Section~\ref{bmfinquotsection}} Let $g \geq \finalbound$ and $a,b,M$ be as in the statement of Lemma~\ref{bmfinquotlemma} and let $\calV = \{[c]: c \in \pi_0(M)\}$.  We will begin by describing some algebraic invariants of Bestvina--Margalit tori, which record how the elements in $\calV$ project onto the different subgroups in $\calH(\BM_{\sigma})$.  We will then prove Lemma~\ref{bmconjlemma}, which allows us to use these algebraic invariants to determine when two Bestvina--Margalit tori $\BM_{\sigma}$ and $\BM_{\tau}$ are in the same $G$--orbit for certain subgroups $G  \subseteq \Sp(2g,\Z)$.  We will then prove Lemma~\ref{bmfinquotfundclassaddpt2}, which describes how these algebraic invariants interact with addition of fundamental classes.  We will then prove Lemma~\ref{bmfincokprimred} and Lemma~\ref{bmfinquotalgint}.   These two results will allow decompose a class $[\BM_{\sigma}]$ into ``simpler" classes, where ``simpler" means roughly that the algebraic invariants of the torus $\BM_{\sigma}$ are bounded.  These last two lemmas make up the main work of the section.

\p{Algebraic invariants of Bestvina--Margalit tori} We begin by describing the invariants of Bestvina--Margalit tori that we will use to prove Lemma~\ref{bmfinquotlemma}.  Let $\calH(\BM_\sigma) = \{\calH_0^\sigma, \calH_1^\sigma, \calH_2^\sigma\}$ as in Section~\ref{homolcurvequotsection}.  For each $v_i \in \calV$, let $\rk^{\calV}(v_{i,k}^\sigma)$ denote the maximal $m \in \Z$ such that that $\proj_{\calH_k^\sigma \cap [b]^{\perp}}(v_i) = m w$ for some nonzero $w \in \calH_k^\sigma$.  For each $v_i, v_j \in \calV$, let
\begin{displaymath}
\theta(\calV)_{i,j,k}(\sigma) = \langle \proj_{\calH_k^\sigma \cap [b]^{\perp}}(v_i), \proj_{\calH_k^\sigma \cap [b]^{\perp}}(v_j)\rangle.
\end{displaymath} 

\bn We now prove a result that describes the orbits of Bestvina--Margalit tori under the action of stabilizer subgroups of $\Sp(2g,\Z)$. 

\begin{lemma}\label{bmconjlemma}
Let $M \subseteq S_g$ be a nonseparating multicurve disjoint from $a$ and $b$.  Let $\calV = \{[c]: c \in \pi_0(M)\}$, and assume that the elements of $\calV$ are indexed as  $\calV = \{v_1,\ldots, v_n\}$.  Let $G$ be the group $\im(\Mod(S_g \cut \left(a \cup b \cup M\right)) \rightarrow \Sp(2g,\Z))$.  Let $\sigma, \tau \subseteq X_g$ be 2--cells and let $\BM_{\sigma}, \BM_{\tau}$ be the corresponding Bestvina--Margalit tori.  The tori $\BM_{\sigma}$ and $\BM_{\tau}$ are in the same $G$--orbit if, after possibly reindexing $\calH(\BM_{\tau})$, we have:
\begin{itemize}
\item $g(\sigma) = g(\tau)$,
\item $\rk^{\calV}(v_{i,k}^\sigma)=\rk^{\calV}(v_{i,k}^\tau)$ for $1 \leq i \leq n, 0 \leq k \leq 2,$ and
\item $\theta(\calV)_{i,j,k}(\sigma )= \theta(\calV)_{i,j,k}(\tau)$ for $1 \leq i,j \leq n, 0 \leq k \leq 2.$
\end{itemize}
\end{lemma}

\begin{proof}
Let $\BM_{\sigma}, \BM_{\tau}$ be two Bestvina--Margalit tori that satisfy the hypotheses of the lemma.  Because $g(\sigma) = g(\tau)$, there is an $f \in \Stab_{\Sp(2g,\Z)}([a], [b])$ satisfying $f\BM_{\sigma} = \BM_{\tau}$.  Therefore, after possibly reindexing, we have $f(\calH_k^\sigma) = \calH_k^\tau$ for every $0 \leq k\leq 2$. 

We claim that $f$ sends $\proj_{\calH_k^\sigma \cap [b]^{\perp}}(v_i)$ to $\proj_{\calH_k^{\tau} \cap [b]^{\perp}}(v_i)$ for every $i,k$.  Indeed, such an element exists since $g(\calH(\sigma)) = g(\calH(\tau))$ and $\BM_{\sigma}$ and $\BM_{\tau}$ have $\rk^{\calV}(v_{i,k}^\sigma) = \rk^{\calV}(v_{i,k}^\tau)$ and $\theta(\calV)_{i,j,k}(\sigma) = \theta(\calV)_{i,j,k}(\tau)$.  Hence for each $\calH_k(\tau)$ there is an $f_k \in \Stab_{\Sp(2g,\Z)}([a],[b])$ which is the identity on $\calH_{k'}(\tau)$ for $k' \neq k$ and which takes $f(v_{i,k}^{\sigma})$ to $v_{i,k}^{\tau}$.  Then we can replace $f$ with $f_0f_1f_2f$, which satisfies $f\proj_{\calH_k^\sigma \cap [b]^{\perp}}(v_i) = \proj_{\calH_k^\tau \cap [b]^{\perp}}(v_i)$.  But then this new $f$ fixes every $v_i$, since
\begin{displaymath}
f (v_i) = \sum_{k \in \{1,2,3\}} f(\proj_{\calH_k^{\sigma} \cap [b]^{\perp}}(v_i)) = \sum_{k \in \{1,2,3\}} \proj_{\calH_k^{\tau} \cap [b]^{\perp}}(v_i) = v_i.
\end{displaymath}
\bn It now remains to show that $f \in G$, i.e., that there is some $F \in \Mod(S_g \cut (a \cup b \cup M))$ such that the image of $F$ under the symplectic representation is $f$.  Since the symplectic representation $\Mod(S_g) \rightarrow \Sp(2g,\Z)$ is surjective, there is some $F' \in \Mod(S_g)$ such that $F'$ is sent to $f$ by the symplectic representation.  Since $\calV$ can be represented by $M$ disjoint from $a$ and $b$, we can extend $\calV \cup \{a,b\}$ to a symplectic basis $\calB = \{\alpha_1,\beta_1,\ldots, \alpha_g, \beta_g\}$.  Choose a set of curves $\widehat{\calB}$ such that:
\begin{itemize}
\item  $\pi_0(M), \{a,b\} \subseteq \widehat{\calB}$,
\item  $\calB = \{[c]: c \in \widehat{\calB}\}$, and
\item $\left|c \cap c'\right| = \left|\langle [c], [c'] \rangle\right|$ for all $c,c' \in \widehat{\calB}$.  
\end{itemize}
\bn  Let $\widehat{\calB}'$ be another set of representatives for $\calB$ satisfying the above conditions, except we have 
\begin{displaymath}
\pi_0(F'(M)), \{F'(a), F'(b)\} \subseteq \widehat{\calB}'.
\end{displaymath}
\bn  But now by the change of coordinates principle, there is some $F'' \in \Mod(S_g)$ taking $\widehat{\calB}' \rightarrow \widehat{\calB}$ in such a way that $F''_*$ acts trivially on $\calB$.  But then $F'' \in \cI_g$, so the image of $F'' \cdot F'$ under the symplectic representation is the same as $F'$, which is $f$.  Then we have $F'' \cdot F'(M) = M$, $F'' \cdot F'(a) = a$, and $F'' \cdot F'(b) = b$.  Hence $F'' \cdot F' \in \Mod(S_g \cut (M \cup a \cup b))$ and $F'' \cdot F'$ maps to $f$ under the symplectic representation, so $f \in G$, as desired.
\end{proof}

\bn We now prove a lemma that describes how the $\theta(\calV)$ and $\rk^{\calV}$ interact with the Bestvina--Margalit tori configured as in Lemma~\ref{bmfinquotfundclassaddpt2}. 

\begin{lemma}\label{bmfinquotfundclassaddpt2}
Let $g \geq \finalbound$, $a,b, M \subseteq S_g$ and $\calV \subseteq H_1(S_g;\Z)$ be as in Lemma~\ref{bmconjlemma}.  Let $x,y,z \subseteq X_g$ and $\calH_0, \calH_1, \calH_2, \calH'_1, \calH'_2$ be as in Lemma~\ref{bmfinquotfundclassadd}.  For any choice of $* = y,z$, and $yz$, let $\sigma_*$ denote the 2--cell containing $x$ and $*$.  Let $v_{i,k}^*$ denote $\proj_{\calH_i \in \calH(\sigma_*) \cap [b]^{\perp}}(v_i)$, and let $\theta(\calV)_{i,j,k}(*)$ denote $\theta(\calV)_{i,j,k}(\sigma_{*})$.  Then the following equalities hold for all $1 \leq i,j \leq 9$:
\begin{multicols}{2}
\begin{enumerate}
\item $\rk^{\calV}(v_{i,0}^{y}) = \rk^{\calV}(v_{i,0}^z) = \rk^{\calV}(v_{i,0}^{yz})$,
\item $\rk^{\calV}(v_{i,1}^{yz}) = \gcd(\rk^{\calV}(v_{i,1}^y), \rk^{\calV}(v_{i,1}^z))$,
\item $\rk^{\calV}(v_{i,2}^y) = \gcd(\rk^{\calV}(v_{i,1}^z), \rk^{\calV}(v_{i,2}^{yz}))$,
\item $\rk^{\calV}(v_{i,2}^z) = \gcd(\rk^{\calV}(v_{i,1}^y), \rk^{\calV}(v_{i,2}^{yz}))$,
\item $\theta(\calV)_{i,j,0}(y) = \theta(\calV)_{i,j,0}(z)= \theta(\calV)_{i,j,0}(yz)$ ,
\item $\theta(\calV)_{i,j,1}(y) + \theta(\calV)_{i,j,1}(z)= \theta(\calV)_{i,j,1}(yz)$,
\item $\theta(\calV)_{i,j,1}(z) + \theta(\calV)_{i,j,2}(yz) = \theta(\calV)_{i,j,2}(y)$, and
\item $\theta(\calV)_{i,j,1}(y) + \theta(\calV)_{i,j,2}(yz) = \theta(\calV)_{i,j,2}(z)$.
\end{enumerate}
\end{multicols}
\end{lemma}

\begin{proof}
Since we have chosen the multicurve $M$ to be disjoint from $a$ and $b$, we have $\calV \subseteq [b]^{\perp}$.  Then, if $X, Y \subseteq [a]^{\perp}$ are two quasi--unimodular lattices with $X \cap Y = \Z[a]$ and $\rk^{\calV}(X)= 2g(X) + 1$, $\rk^{\calV}(Y) = 2g(Y) + 1$, we see that for any $v \in [a]^{\perp} \cap [b]^{\perp}$, we have
\begin{displaymath}
\proj_{(X + Y) \cap [b]^{\perp}}(v) = \proj_{X \cap[b]^{\perp}}(v) + \proj_{Y \cap[b]^{\perp}}(v).
\end{displaymath}
\bn Now, Lemma~\ref{bmfinquotfundclassadd} says that the following hold:
\begin{enumerate}
\item $\calH_0^{y} = \calH_0^z = \calH_0^{yz}$,
\item $\calH_1^y + \calH_1^{z} = \calH_1^{yz}$,
\item $\calH_1^y + \calH_2^{yz} = \calH_2^{z}$, and
\item $\calH_1^z + \calH_2^{yz} = \calH_2^{y}$.
\end{enumerate}
Hence the above observation about projections tells us that the following equalities among $v_{i,k}^{y}$, $v_{i,k}^z$ and $v_{i,k}^{yz}$ hold for all $1 \leq i \leq 9$:
\begin{multicols}{2}
\begin{enumerate}[label=(\alph*)]
\item $v_{i,0}^{y} = v_{i,0}^z = v_{i,0}^{yz}$,
\item $v_{i,1}^{y} + v_{i,1}^{z} = v_{i,1}^{yz}$,
\item $v_{i,1}^y + v_{i,2}^{yz} = v_{i,2}^z$, and
\item $v_{i,1}^z + v_{i,2}^{yz} = v_{i,2}^y$.
\end{enumerate}
\end{multicols}
\bn The relations in the statement of the lemma are derived as follows.

\p{Relations (1) and (5)}  Since $v_{i,0}^y = v_{i,0}^z = v_{i,0}^{yz}$ by equality (a), we must have $\rk^{\calV}(v_{i,0}^y) = \rk^{\calV}(v_{i,0}^z) = \rk^{\calV}(v_{i,0}^{yz})$, and similarly for $\theta(\calV)_{i,j,0}(y) = \theta(\calV)_{i,j,0}(z) = \theta(\calV)_{i,j,0}(yz).$ 

\p{Relation (2) and (6)}  We have $v_{i,1}^y + v_{i,1}^z =v_{i,1}^{yz}$ by relation (b).  If $w_{i,1}^*$ is a primitive class with $\rk^{\calV}(v_{i,1}^*) w_{i,1}^* = v_{i,1}^*$ for some choice of $* = y,z,yz$, then by relation (b) we have $\rk^{\calV}(v_{i,1}^y) w_{i,1}^y + \rk^{\calV}(v_{i,1}^z)w_{i,1}^z = v_{i,1}^{yz}$.  By Lemma~\ref{bmfinquotfundclassadd}, the group $\calH_1^{x,yz} \in \calH(\sigma_{yz})$ is given by $\calH_1^{x,y} + \calH^{x,z}_1$ with $\calH_1^{x,y} \cap \calH^{x,z}_1 = \Z[a]$, so we may represent $w_{i,1}^y$ and $w_{i,1}^z$ using disjoint curves $c_y, c_z \subseteq S_g$.  Then the homology class
\begin{displaymath}
    \rk^{\calV}(v_{i,1}^y)/\gcd(\rk^{\calV}(v_{i,1}^y), \rk^{\calV}(v_{i,1}^z))\cdot [c_y] + \rk^{\calV}(v_{i,1}^z)/\gcd(\rk^{\calV}(v_{i,1}^y), \rk^{\calV}(v_{i,1}^z))\cdot [c_z]
\end{displaymath}
\bn is primitive, since by definition $\rk^{\calV}(v_{i,1}^y)/\gcd(\rk^{\calV}(v_{i,1}^y), \rk^{\calV}(v_{i,1}^z))$ and $\rk^{\calV}(v_{i,1}^z)/\gcd(\rk^{\calV}(v_{i,1}^y), \rk^{\calV}(v_{i,1}^z))$ are relatively prime.  Therefore $\rk^{\calV}(v_{i,1}^{yz}) = \gcd(\rk^{\calV}(v_{i,1}^y), \rk^{\calV}(v_{i,1}^z))$.  For relation (6), relation (b) implies that
\begin{displaymath}
    \theta(\calV)_{i,j,1}^{yz} = \langle v_{i,1}^{y} + v_{i,1}^{z}, v_{j,1}^{y} + v_{j,1}^{z} \rangle.
\end{displaymath}
\bn As above, $v_{i,1}^y$ and $v_{j,1}^z$, can be represented by multiples of disjoint curves, and similarly for $v_{j,1}^z$ and $v_{i,1}^y$.  Therefore we have
\begin{align*}
    \theta(\calV)_{i,j,1}^{yz} &= \langle v_{i,1}^{y}, v_{j,1}^{y}  \rangle + \langle v_{i,1}^{y} , v_{j,1}^{z} \rangle + \langle v_{i,1}^{z}, v_{j,1}^{y} \rangle + \langle v_{i,1}^{z}, v_{j,1}^{z} \rangle\\
    &= \langle v_{i,1}^{y}, v_{j,1}^{y}\rangle + \langle v_{i,1}^{z}, v_{j,1}^{z} \rangle\\
    &= \theta(\calV)_{i,j,1}(y) + \theta(\calV)_{i,j,1}(z).
\end{align*}

\p{Relations (3) and (7)}  These two relations follow from relation (c) using an argument similar to that for relations (2) and (6).

\p{Relations (4) and (8)}  These two relations follow from relation (d) using an argument to that for relations (2) and (6).
\end{proof}

\subsection{The proof of Lemma~\ref{bmfincokprimred}}

\bn We begin by proving an auxiliary result, which we will need to prove Lemma~\ref{bmfincokprimred} and Lemma~\ref{bmfinquotalgint}.

\begin{lemma}\label{bmfinquotgenusincrease}
Let $g \geq \finalbound$, $a,b, M \subseteq S_g$, and $\calV \subseteq H_1(S_g;\Z)$ be as in Lemma~\ref{bmconjlemma}.  Let $\sigma \subseteq X_g$ be a 2--cell.  Suppose that there is a $1 \leq m \leq 9$ such that the following hold:
\begin{itemize}
\item $\rk^{\calV}(v_{i,k}^{\sigma}) \leq 1$ for all $1 \leq i < m$ and $0 \leq k \leq 2$,
\item $\rk^{\calV}(v_{m,2}^{\sigma}) \geq 2$,
\item $\rk^{\calV}(v_{m,2}^{\sigma}) = \max\{\rk^{\calV}(v_{m,k}^{\sigma}): 0 \leq k \leq 2\}$, and
\item $\rk^{\calV}(v_{m,1}^{\sigma}) \neq 0, \rk^{\calV}(v_{m,2}^{\sigma})$.
\end{itemize}
\bn Then there is a relation in $\BM_2(X_g;\Q)$ given by
\begin{displaymath}
[\BM_{\sigma}] = \sum_{\iota=1}^q \lambda_{\iota}[\BM_{\sigma_{\iota}}]
\end{displaymath}
\bn such that, for all $1 \leq \iota \leq q$:
\begin{enumerate}
\item we have $\rk^{\calV}(v_{i,k}^{\sigma_{\iota}}) \leq 1$ for all $1 \leq i < m$ and $0 \leq k \leq 2$,
\item the set $\{\rk^{\calV}(v_{m,k}^{\sigma_{\iota}}):0 \leq k \leq 2\}$ is upper bounded by $\rk^{\calV}(v_{m,2}^{\sigma})$,
\item the set $\{\rk^{\calV}(v_{m,k}^{\sigma_{\iota}}):0 \leq k \leq 2\}$ has no more elements equal to $\rk^{\calV}(v_{m,2}^{\sigma_{\iota}})$ than the does the set $\{\rk^{\calV}(v_{m,k}^{\sigma}):0 \leq k \leq 2\}$,
\item if $\max_{0\leq k \leq 2}\{\rk^{\calV}(v_{m,k}^{\sigma_{\iota}})\} = \rk^{\calV}(v_{m,2}^{\sigma})$, we have $0 < \rk^{\calV}(v_{m,1}^{\sigma_{\iota}}) < \max\{\rk^{\calV}(v_{m,k}^{\sigma_{\iota}}):0 \leq k \leq 2\}$, and
\item we have $g(\calH_1^{\sigma_{\iota}}) \geq 10$.
\end{enumerate}
\end{lemma}

\begin{proof}
If $g(\calH_1^{\sigma}) \geq 10$, then we are done by taking $q = 1$, $\lambda_1 = 1$ and $[\BM_{\sigma_1}] = [\BM_\sigma]$.  Otherwise, since $g \geq \finalbound$, there is a $\kappa = 0,2$ such that $g(\calH_\kappa^{\sigma}) \geq 19$.  Then since $\left|\calV\right| = 9$, the fact that $g(\calH_{\kappa}^{\sigma}) \geq 19$ implies that there is a primitive $\calH \subseteq \calH_\kappa^{\sigma}$ such that the following hold:
\begin{multicols}{2}
\begin{itemize}
\item $a \in \calH$, 
\item $v_{i,\kappa}^{\sigma} \in \calH^{\perp}$ for all $1 \leq i \leq 9$, 
\item $g(\calH) = 9$,
\item $g(\calH) + g(\calH_1^{\sigma}) \geq 10$, and  
\item $2 *g(\calH) + 1 = \rk^{\calV}(\calH)$.
\end{itemize}
\end{multicols}
\bn Let $\widehat{\calH} = \{\calH_0^{\sigma}, \calH, \calH^{\perp} \cap \calH_1^{\sigma}, \calH_2^{\sigma}\}$.  We can describe $\widehat{\calH}$ using the graphical notation of Section~\ref{homolcurvequotsection} as in Figure~\ref{bmfinquotgenusincreasefig}.
\begin{figure}[h]
\begin{tikzpicture}
\node[anchor = south west, inner sep = 0] at (0,0){\includegraphics{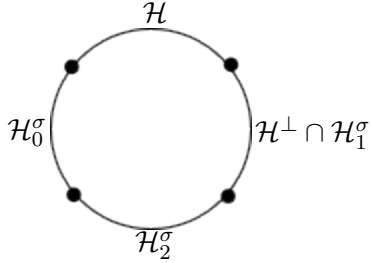}};
\node at (-0.3,1.3) {$\calH_0^{\sigma}$};
\node at (1.4,2.9) {$\calH$};
\node at (3.5,1.3) {$\calH^{\perp} \cap \calH_1^{\sigma}$};
\node at (1.4,-0.2) {$\calH_2^{\sigma}$};
\end{tikzpicture}
\caption{The graphical representation of $\widehat{\calH}$}\label{bmfinquotgenusincreasefig}
\end{figure}

Let $z \subseteq X_g$ be the edge with $\calH \in \calH(z)$.  Let $y \subseteq X_g$ the edge with $\calH^{\perp} \cap \calH_1^{\sigma} \in \calH(y)$.  Let $x \subseteq X_g$ be the edge with $\calH_{\kappa'}^{\sigma} \in \calH(y)$, where $0 \leq \kappa' \leq 2$ and $\kappa' \neq 1,\kappa$.  After orienting $y$ and $z$ correctly, we let $yz$ denote the third edge of a 2--cell in $X_g$ containing $y$ and $z$.  We have $\calH + \calH^{\perp} \cap \calH_1^{\sigma} = \calH_1^{\sigma} \in \calH(yz)$, so $x$ and $yz$ are two edges of the 2--cell $\sigma$.  Now, Lemma~\ref{lemmaaltcohomolauxadd} says that there is a relation
\begin{displaymath}
[\BM_{x,yz}] = [\BM_{x,y}] + [\BM_{x,z}].
\end{displaymath}
\bn Now, let $\sigma_y, \sigma_z \subseteq X_g$ be 2--cells containing $x$ and then $y$ and $z$ respectively.  We will now show that $\sigma_{z}$ and $\sigma_{y}$ desired properties for $\sigma_{\iota}$.

\p{The 2--cell $\sigma_z$ satisfies the desired properties of $\sigma_{\iota}$} Assume that $\calH(\sigma_z)$ is indexed so that $\calH = \calH_\kappa^{\sigma_z}$, $\calH_1^{\sigma_z} = \calH^{\perp} \cap \calH_{\kappa}^{\sigma} + \calH_1^{\sigma}$, and $\calH_{\kappa'}^{\sigma_z} = \calH_{\kappa'}^{\sigma}$.  By our assumption that $v_{i,\kappa}^{\sigma} \in \calH^{\perp}$ for all $1 \leq i \leq 9$, the following hold:
\begin{multicols}{2}
\begin{enumerate}
\item $v_{i,\kappa}^{\sigma_z} = 0$ for all $1 \leq i \leq 9$,
\item $v_{i,1}^{\sigma_z} = v_{i,\kappa}^{\sigma} + v_{i,1}^{\sigma}$ for all $1 \leq i \leq 9$, and
\item $v_{i,\kappa'}^{\sigma_z} = v_{i,\kappa'}^{\sigma}$ for all $1 \leq i \leq 9$.
\end{enumerate}
\end{multicols}
\bn Then we can compute $\rk^{\calV}(v_{i,\kappa}^{\sigma_z})$ as follows.
\begin{enumerate}
\item For $1 \leq i < m$ and $0 \leq k \leq 2$, we have $v_{i,k}^{\sigma_z}$ primitive.  The class $v_{i,k}^{\sigma}$ is primitive for all $1 \leq i <m$ and $0 \leq k \leq 2$ by hypothesis, so $v_{i,\kappa'}^{\sigma_z}$ is primitive.  Additionally, for each $1 \leq i < m$, the class $v_{i,\kappa}^{\sigma}$ can be represented by a simple closed curve disjoint from a representative for the class $v_{i,1}^{\sigma}$, so $v_{i, 1}^{\sigma_z}$ is primitive.  Hence $\rk^{\calV}(v_{i,k}^{\sigma_z}) \leq 1$ for $1 \leq i < m$ and $0 \leq k \leq 2$.
\item For $i = m$, $\rk^{\calV}(v_{i,\kappa}^{\sigma_z}) = 0$, $\rk^{\calV}(v_{i, \kappa'}^{\sigma_z}) = \rk^{\calV}(v_{i, \kappa'}^{\sigma_z})$, and $\rk^{\calV}(v_{i, 1}^{\sigma_z}) = \gcd(\rk^{\calV}(v_{i, \kappa}^{\sigma}), \rk^{\calV}(v_{i, 1}^{\sigma}))$.
\end{enumerate}
\bn The computation of $\rk^{\calV}(v_{i,k}^{\sigma_z})$ above implies that $\sigma_z$ satisfies property (1) of $\sigma_{\iota}$.  Then the computation of $\rk^{\calV}(v_{m,k}^{\sigma_z})$ implies that $\sigma_z$ satisfies properties (2)--(4) of $\sigma_{\iota}$.  Indeed, properties (2) and (3) follow from the fact that $\gcd(\rk^{\calV}(v_{i, \kappa}^{\sigma}), \rk^{\calV}(v_{i, 1}^{\sigma})) \leq \rk^{\calV}(v_{i, \kappa}^{\sigma}), \rk^{\calV}(v_{i, 1}^{\sigma})$.  Property (4) follows from the fact that $\rk^{\calV}(v_{i,1}^{\sigma_z}) = \gcd(\rk^{\calV}(v_{i,\kappa}^{\sigma}), \rk^{\calV}(v_{i,1}^{\sigma})) \leq \rk^{\calV}(v_{i,1}^{\sigma}) < \rk^{\calV}(v_{i,2}^{\sigma})$.  Finally, property (5) follows from the fact that $g(\calH_1^{\sigma_z}) = g(\calH_1^{\sigma}) + g(\calH_{\kappa}^{\sigma}) - g(\calH)$.  Indeed, we have assumed that $g(\calH_{\kappa}^{\sigma})  \geq 19$, so $g(\calH_1^{\sigma_z}) \geq 19 - g(\calH)$.  Then since $g(\calH) = 9$ by hypothesis, we have $g(\calH_1^{\sigma_z}) \geq 19 - 9 \geq 10$, so $\sigma_z$ satisfies property (5) of $\sigma_{\iota}$.  

\p{The 2--cell $\sigma_{y}$ satisfies the desired properties of $\sigma_{\iota}$}  Assume that $\calH(\sigma_{y})$ is indexed so that $\calH_{1}^{\sigma_{y}} = \calH_1^{\sigma} + \calH$, $\calH_{\kappa'}^{\sigma_{y}} = \calH_{\kappa}^{\sigma}$, and $\calH_{\kappa}^{\sigma_{y}} = \calH^{\perp} \cap \calH_{\kappa}^{\sigma_{y}}$.  The assumption that $v_{i,1}^{\sigma} \in \calH^{\perp}$ for all $1 \leq i \leq 9$ implies that $v_{i,k}^{\sigma} = v_{i,k}^{\sigma_{y}}$ for all $1 \leq i \leq 9$, $0 \leq k \leq 2$.  Therefore since $\sigma$ satisfies properties (1)--(4) of $\sigma_{\iota}$ by hypothesis, then $\sigma_{\iota}$ does as well.  Then $g(\calH_1^{\sigma_{y}}) = g(\calH_1^{\sigma}) + g(\calH) \geq 1 + 9 = 10$ by assumption, so $\sigma_{y}$ satisfies property (5) of $\sigma_{\iota}$.

Now, we have shown that $[\BM_{x,yz}] = [\BM_{x,y}] + [\BM_{x,z}]$.  By our choice of $x$ and $y$, we have $\BM_{x,yz} = \BM_{\sigma}$.  Therefore we have a relation
\begin{displaymath}
[\BM_{x,y}] = [\BM_{x,z}] - [\BM_{x,yz}].
\end{displaymath}
\bn By taking $s = 2$, $\tau_1 = \sigma_{z}$, $\tau_2 = \sigma_{yz}$, $\lambda_1= 1$, and $\lambda_2 = -1$, the proof is complete.
\end{proof}

\bn We now prove Lemma~\ref{bmfincokprimred}, which says that any fundamental class $[\BM_{\sigma}]$ is a linear combination of classes $[\BM_{\tau_\ell}]$ with $\rk^{\calV}(v_{i,k}^{\tau_{\ell}}) \leq 1.$

\begin{lemma}\label{bmfincokprimred}
Let $g \geq \finalbound$, $a,b, M \subseteq S_g$ and $\calV \subseteq H_1(S_g;\Z)$ be as in Lemma~\ref{bmconjlemma}.  Let $\BM_{\sigma}$ be a Bestvina--Margalit torus.  Then there is a collection of 2--cells $\tau_1,\ldots, \tau_p \subseteq X_g$ that satisfy
\begin{displaymath}
[\BM_{\sigma}] = \sum_{\ell=1}^p \lambda_\ell [\BM_{\tau_\ell}]
\end{displaymath}
\bn and such that $\rk^{\calV}(v_{i,k}^{\tau_\ell}) \leq 1$ for every $1 \leq \ell \leq p$, $1 \leq i \leq 9$ and $k = 0,1,2$.
\end{lemma}

\begin{proof}
We induct on the number $m$ with $0 \leq m \leq 9$ such that we can write $[\BM_{\sigma}]$ as a linear combination of classes $[\BM_{\tau_\ell}]$ that satisfy, for $1 \leq \ell \leq p$ and $1 \leq i \leq m$, the inequality $\rk^{\calV}(v_{i,k}^{\tau_s}) \leq 1$.  

\p{Base case: $m = 0$}  In this case, the result holds trivially, since there are no $i$ with $1 \leq i$ and $i \leq m$.

\p{Inductive step: $m \geq 1$}  Inductively, assume that $[\BM_\sigma]$ can be written as a linear combination of classes $[\BM_{\tau_\ell}]$ that satisfy, for $1 \leq \ell \leq p$ and $1 \leq i \leq m$, the inequality $\rk^{\calV}(v_{i,k}^{\tau_s}) \leq 1$.  Hence without loss of generality, we may assume that $\BM_{\sigma}$ satisfies $\rk^{\calV}(v_{i,k}^{\sigma}) \leq 1$ for $1 \leq i < m$.  We will show that $[\BM_{\sigma}]$ is a linear combination of fundamental classes $[\BM_{\tau}]$, each of which has $\rk^{\calV}(v_{i,k}^{\tau}) \leq 1$ for all $1 \leq i \leq m$ and $0 \leq k \leq2$.  We will perform a double induction on two quantities associated to $\sigma$:
\begin{itemize}
\item $\maxrk_m(\sigma) = \max_{0 \leq k \leq 2}\{\rk^{\calV}(v_{m,k}^\sigma)\}$, and
\item $\nummaxrk_m(\sigma) = \left|\{k: 0 \leq k \leq 2, \rk^{\calV}(v_{m,k}^\sigma) = \maxrk_m(\sigma)\}\right|$.
\end{itemize}
\bn In particular, we suppose that $\sigma$ is a 2--cell with $\rk^{\calV}(v_{i,k}^{\sigma}) \leq 1$ for all $1 \leq i < m$ and $0 \leq k \leq 2$.   We suppose that for each 2--cell $\tau$ with: 
\begin{itemize}
\item $\rk^{\calV}(v_{i,k}^{\tau}) \leq 1$ for $1 \leq i < m$ and $0 \leq k \leq 2$, and either:
\begin{itemize}
\item $\maxrk_m(\tau) < \maxrk_m(\sigma)$, or
\item $\maxrk_m(\tau) \leq \maxrk_m(\sigma)$ and $\nummaxrk_m(\tau) < \nummaxrk_m(\sigma)$,
\end{itemize}
\end{itemize}
\bn we know that $[\BM_{\tau}]$ is a $\Q$--linear combination of classes $[\BM_{\tau_{\ell}}]$ such that $\rk^{\calV}(v_{i,k}^{\tau_{\ell}}) \leq 1$ for all $1 \leq i \leq m$ and $0 \leq k \leq 2$.  We will show that this implies that the same holds for $[\BM_{\sigma}]$.  

\p{Base case: $\maxrk_m(\sigma) = 1$}  In this case, $\BM_{\sigma}$ satisfies $\rk^{\calV}(v_{i,k}^{\sigma}) \leq 1$ for all $1 \leq i \leq m$ and $0 \leq k \leq 2$, so the inductive step for the induction on $m$ is complete.

\medskip

\bn\textit{Inductive step for $\maxrk_m(\sigma)$ and $\nummaxrk_m(\sigma)$: The inductive hypothesis for the induction on $m$ holds for all $\tau \subseteq X_g$ with:
\begin{itemize}
\item either $\maxrk_m(\tau) < \maxrk_m(\sigma)$ or
\item both $\maxrk_m(\tau) \leq \maxrk_m(\sigma)$ and $\nummaxrk_m(\tau) < \nummaxrk_m(\sigma)$.
\end{itemize}} 
\bn We will show that 
\begin{displaymath}
[\BM_{\sigma}] = \sum_{\ell=1}^{p} \lambda_{\ell}[\BM_{\tau_{\ell}}]
\end{displaymath}
\bn such that the following hold:
\begin{itemize}
\item at least one of the following holds:
\begin{itemize}
\item $\maxrk_m(\sigma_{\iota}) <\maxrk_m(\sigma)$ or 
\item $\maxrk_m(\sigma_{\iota}) \leq \maxrk_m(\sigma)$ and $\nummaxrk_m(\sigma_{\iota}) < \nummaxrk_m(\sigma)$, 
\end{itemize}
\item $\lambda_{\ell} \in \Q$, and
\item $\rk^{\calV}(v_{i,k}^{\tau_{\ell}}) \leq 1$ for all$1 \leq i < m$ and $k = 0,1,2$.
\end{itemize}
\bn  This completes the proof, since every $\tau_{\ell}$ as above satisfies the inductive hypothesis for the induction on $m$ by the inductive hypothesis for the induction on $\maxrk_m(\sigma)$ and $\nummaxrk_m(\sigma)$.  

 If $\rk^{\calV}(v_{m,k}^{\sigma}) = 0$ for two distinct choices of $k$, then $\rk^{\calV}(v_{m,k}^{\sigma}) = 1$ for the third choice of $k$, since $v_m = v_{m,0}^{\sigma} + v_{m,1}^{\sigma} + v_{m,2}^{\sigma}$ and $v_m$ is primitive by assumption.  Otherwise, reindex $\calH(\BM_{\sigma})$ such that $\rk^{\calV}(v_{m,2}^{\sigma})$ is maximal among all $\rk^{\calV}(v_{m,k}^\sigma)$.  Note that at least one remaining $\rk^{\calV}(v_{m,i}^{\sigma})$ must have $\rk^{\calV}(v_{m,i}^{\sigma}) \neq 0, \neq \rk^{\calV}(v_{m,2}^{\sigma})$ since $v_m$ is primitive by assumption. Hence, we may further reindex so that $\rk^{\calV}(v_{m,1}^\sigma) \neq 0, \rk^{\calV}(v_{m,2}^\sigma)$.  Then $\sigma$ satisfies the hypothesis of Lemma~\ref{bmfinquotgenusincrease}, so the conclusion of Lemma~\ref{bmfinquotgenusincrease} implies that we may rewrite $[\BM_{\sigma}]$ as a linear combination of fundamental classes of Bestvina--Margalit tori such that each tori $\BM_{\sigma_{\iota}}$ satisfies the following:
\begin{enumerate}
\item $\rk^{\calV}(v_{i,k}^{\sigma_{\iota}}) \leq 1$ for all $1 \leq i < m$ and $0 \leq k \leq 2$,
\item $\maxrk_m(\sigma_{\iota}) \leq \maxrk_m(\sigma)$,
\item if $\maxrk_m(\sigma_{\iota}) = \maxrk_m(\sigma)$, then $\nummaxrk_m(\sigma_{\iota}) \leq \nummaxrk_m(\sigma)$,
\item if $\maxrk_m(\sigma_{\iota}) = \rk^{\calV}(v_{m,2}^{\sigma})$, then $0 < \rk^{\calV}(v_{m,1}^{\sigma_{\iota}}) < \rk^{\calV}(v_{m,2}^{\sigma_{\iota}})$, and
\item $g(\calH_1^{\sigma_{\iota}}) \geq 10$.
\end{enumerate}
\bn Therefore we may assume without loss of generality that $g(\calH_1^{\sigma}) \geq 10$.

Now, since $g(\calH_1^{\sigma}) \geq 10$ and $\left|\calV \right| \leq 9$, there must be some nonzero primitive $h \in \calH_1^{\sigma} \cap [b]^{\perp}$ such that there is a with multicurve $M_1 \subseteq S_g$ where $\{[c]: c \in \pi_0(M_1)\} = \{h, w_{1,1}^\sigma, \ldots, w_{9,1}^{\sigma}\}$, where $w_{i,1}^{\sigma}$ is a primitive element with $v_{i,1}^{\sigma} = \lambda w_{i,1}^{\sigma}$ for some $\lambda \in \Z$.  For $1 \leq i \leq 9$, let $h_i = v_{i,1}^{\sigma} - h$.  Since $g(\calH_1^{\sigma}) \geq 10$, there is a primitive subgroup $\calH \subseteq \calH_1^{\sigma}$ such that the following hold:
\begin{multicols}{2}
\begin{itemize}
\item $a \in \calH$,
\item $2g(\calH) + 1 = \rk(\calH)$, 
\item $h_i \in \calH$ for all $1 \leq i \leq 9$, and
\item $h \in \calH^{\perp}$.
\end{itemize}
\end{multicols}
\bn Let $\widehat{\calH} = \{\calH_0^{\sigma}, \calH, \calH_1^{\sigma} \cap \calH^{\perp}, \calH_2^{\sigma}\}$, which can be graphically represented as in Figure~\ref{bmfincokprimredfig}. 

\begin{figure}[h]
\begin{tikzpicture}
\node[anchor = south west, inner sep = 0] at (0,0){\includegraphics{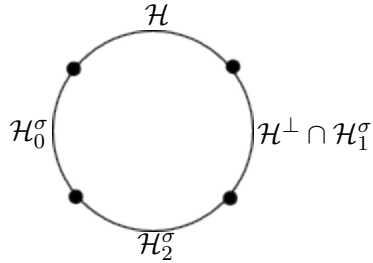}};
\node at (-0.3,1.3) {$\calH_0^{\sigma}$};
\node at (1.4,2.9) {$\calH$};
\node at (3.5,1.3) {$\calH^{\perp} \cap \calH_1^{\sigma}$};
\node at (1.4,-0.2) {$\calH_2^{\sigma}$};
\end{tikzpicture}
\caption{The graphical representation of $\widehat{\calH}$}\label{bmfincokprimredfig}
\end{figure}

 Let $y \subseteq X_g$ be the unique edge such that $\calH \in \calH(y)$.  Let $x$ be the edge of $\sigma$ with $\calH_0^{\sigma} \in \calH(x)$.  There is a unique edge $z \subseteq X_g$ such that:
\begin{itemize}
\item $y$ and $z$ are two edges in a 2--cell $\tau$ and
\item the third edge of $\tau$, denoted $yz$, is the unique edge with $\calH_1^{\sigma} \in \calH(yz)$.
\end{itemize}
\bn  Let $\sigma_y$ and $\sigma_z$ denote 2--cells with $x,y \subseteq \sigma_y$, $x,z \subseteq \sigma_z$.  Assume that $\calH(\sigma_y)$ is indexed so that $\calH_0^{\sigma_y} = \calH_0^{\sigma}$, $\calH_1^{\sigma_y} = \calH$, and $\calH_2^{\sigma_y} = \calH^{\perp} \cap \calH_1^{\sigma} + \calH_2^{\sigma}$.  Our choice of $\calH$ implies that $\proj_{\calH_1^{\sigma_y} \cap [b]^{\perp}}(v_i) = h_i$ for all $1 \leq i \leq 9$.  Since $h$ is nonzero and primitive by assumption, the element $h_i$ is nonzero and primitive as well for an $1 \leq i \leq 9$, so $\rk^{\calV}(v_{i,1}^{\sigma_y}) = 1$ for all $1 \leq i \leq 9$.  Assume now that $\calH(\sigma_z)$ is indexed such that $\calH_0^{\sigma_z} = \calH_0^{\sigma}$, $\calH_1^{\sigma_z}= \calH^{\perp} \cap \calH_1^{\sigma}$, and $\calH_2^{\sigma_z} = \calH + \calH_2^{\sigma}$.  This means that $\proj_{\calH_1^{z} \cap [b]^{\perp}}(v_i) = h$ for all $1 \leq i \leq 9$, so $\rk^{\calV}(v_{i,1}^{\sigma_z}) = 1$ for all $1 \leq i \leq 9$, since $h$ is primitive by assumption.  Furthermore, for such $y$ and $z$, after possibly reorienting $\sigma_y$ and $\sigma_z$, we have 
\begin{displaymath}
[\BM_{\sigma}] = [\BM_{\sigma_y}] + [\BM_{\sigma_z}].
\end{displaymath}
\bn We have assumed that $\calH(\sigma_y)$ and $\calH(\sigma_z)$ are indexed so that $\calH_0^{\sigma_y} = \calH_0^{\sigma_z} = \calH_0^{\sigma}.$  Relations (1)--(4) of Lemma~\ref{bmfinquotfundclassaddpt2} and the above computations of $\rk^{\calV}(v_{i,1}^{\sigma_y})$ and $\rk^{\calV}(v_{i,1}^{\sigma_z})$ imply that the following hold for all $1 \leq i \leq 9$:
\begin{enumerate}
\item $\rk^{\calV}(v_{i,0}^{\sigma_y}) = \rk^{\calV}(v_{i,0}^{\sigma_z}) = \rk^{\calV}(v_{i,0}^{\sigma})$,
\item $\rk^{\calV}(v_{i,1}^{\sigma}) = \gcd(\rk^{\calV}(v_{i,1}^{\sigma_y}), \rk^{\calV}(v_{i,1}^{\sigma_z})) = 1$,
\item $\rk^{\calV}(v_{i,2}^{\sigma_y}) = \gcd(\rk^{\calV}(v_{i,1}^{\sigma_z}), \rk^{\calV}(v_{i,2}^{\sigma})) = \gcd(1, \rk^{\calV}(v_{i,2}^{\sigma})) = 1$, and
\item $\rk^{\calV}(v_{i,2}^{\sigma_z}) = \gcd(\rk^{\calV}(v_{i,1}^{\sigma_y}), \rk^{\calV}(v_{i,2}^{\sigma})) = \gcd(1, \rk^{\calV}(v_{i,2}^{\sigma})) = 1$.
\end{enumerate}

\bn  We have assumed that $1 \leq \rk^{\calV}(v_{m,1}^{\sigma}) < \rk^{\calV}(v_{m,2}^{\sigma})$, so we must have:
\begin{itemize}
\item $1 = \rk^{\calV}(v_{m,2}^{\sigma_y}) < \rk^{\calV}(v_{m,2}^\sigma)$ and
\item $1 = \rk^{\calV}(v_{m,2}^{\sigma_z})  <\rk^{\calV}(v_{m,2}^\sigma)$.
\end{itemize}
\bn Then relations (1)--(4) together imply the following:
\begin{itemize}
\item For any pair $i,k$ with $1 \leq i \leq 9$ and $0 \leq k \leq 2$ with $\rk^{\calV}(v_{i,k}^{\sigma}) \leq 1$, we have $\rk^{\calV}(v_{i,k}^{\sigma_y}), \rk^{\calV}(v_{i,k}^{\sigma_z}) \leq 1$, and
\item for any pair $i,k$ with $1 \leq i \leq 9$ and $0 \leq k \leq 2$ with $\rk^{\calV}(v_{i,k}^{\sigma} \geq 1$, we have $\rk^{\calV}(v_{i,k}^{\sigma_y}), \rk^{\calV}(v_{i,k}^{\sigma_z}) \leq \rk^{\calV}(v_{i,k}^{\sigma})$.
\end{itemize}
\bn Therefore, since we have assumed that $\rk^{\calV}(v_{m,2}^{\sigma})$ was maximal over all $\rk^{\calV}(v_{m,k}^{\sigma})$, we see that the two sets $\{\rk^{\calV}(v_{i,k}^{\sigma_y})\}_{1 \leq i \leq m, k = 0,1,2}$ and $\{\rk^{\calV}(v_{i,k}^{\sigma_z})\}_{1 \leq i \leq m, k = 0,1,2}$ each have strictly fewer elements equal to or exceeding $\rk^{\calV}(v_{m,2}^{\sigma})$ than does the set $\{\rk^{\calV}(v_{i,k}^{\sigma})\}_{1 \leq i \leq m, k = 0,1,2}$.  Hence for $y$, we have:
\begin{itemize}
\item $\rk^{\calV}(v_{i,k}^{\sigma_y}) \leq 1$ for all $1 \leq i < m$ and $0 \leq k \leq 2$, and
\item either:
\begin{itemize}
\item $\maxrk_m(\sigma_y) < \maxrk_m(\sigma)$, or
\item $\maxrk_m(\sigma_y) \leq \maxrk_m(\sigma)$ and $\nummaxrk_m(\sigma_y) < \nummaxrk_m(\sigma)$
\end{itemize}
\end{itemize}
\bn and similarly for $z$.  Therefore we have
\begin{displaymath}
[\BM_{\sigma}] = [\BM_{\sigma_y}] + [\BM_{\sigma_z}]
\end{displaymath}
\bn such that $\sigma_y$ and $\sigma_z$ satisfy the inductive hypothesis for the induction on $\maxrk_m$ and $\nummaxrk_m$.  The inductive hypothesis for $\maxrk_m$ and $\nummaxrk_m$ says that there are relations
\begin{displaymath}
[\BM_{\sigma_y}] = \sum_{\ell_y = 1}^{p_y} \lambda_{\ell_y} [\BM_{\tau_{\ell_y}}] \text{ and }[\BM_{\sigma_z}] = \sum_{\ell_z = 1}^{p_z} \lambda_{\ell_z} [\BM_{\tau_{\ell_z}}] 
\end{displaymath}
\bn such that for each $1 \leq \ell_y \leq p_y$ and for each $1 \leq i \leq m$, $0 \leq k \leq 2$, we have $\rk^{\calV}(v_{i,k}^{\tau_{\ell_y}} )\leq 1$, and similarly for $z$.  By combining these two relations, we have
\begin{displaymath}
[\BM_{\sigma}] = \sum_{\ell_y = 1}^{p_y} \lambda_{\ell_y} [\BM_{\tau_{\ell_y}}] + \sum_{\ell_z = 1}^{p_z} \lambda_{\ell_z} [\BM_{\tau_{\ell_z}}] = \sum_{\ell = 1}^p \lambda_\ell [\BM_{\tau_\ell}]
\end{displaymath}
\bn where for each $1 \leq \ell \leq p$ and $1 \leq i \leq m$ and $0 \leq k \leq 2$, we have $\rk^{\calV}(v_{i,k}^{\tau_{\ell}}) \leq 1$, so the inductive step is complete.
\end{proof}

\subsection{The proof of Lemma~\ref{bmfinquotalgint}}

\bn We are now almost ready to show that fundamental classes $[\BM_\sigma]$ with $\left|\rk^{\calV}(v_{i,k}^{\sigma})\right| \leq 1$ can be written as linear combinations of classes $[\BM_{\tau_{s}}]$ with $\left| \theta(\calV)_{i,j,k}(\tau_s)\right| \leq 1$.  We first prove the following lemma.

\begin{lemma}\label{bmfinquotalggenusincrease}
Let $g \geq \finalbound$, $a,b, M \subseteq S_g$ and $\calV \subseteq H_1(S_g;\Z)$ be as in Lemma~\ref{bmconjlemma}.  Let $\BM_{\sigma} \subseteq X_g$ be a Bestvina--Margalit torus such that $\rk^{\calV}(v_{i,k}^{\sigma}) \leq 1$ for every $1 \leq i \leq 9$ and $0 \leq k \leq 2$.  Assume that not all pairs $1 \leq i,j \leq 9$ satisfy $\left|\theta(\calV)_{i,j,k}(\sigma)\right| \leq 1$.  Choose a pair $(i',j')$ with $1 \leq i' < j' \leq 9$ such that, after possibly reindexing $\calH(\BM_\sigma)$, $\left|\theta(\calV)_{i',j',1}(\sigma)\right| \geq 2$ and $\left|\theta(\calV)_{i',j',1}(\sigma)\right|$ is maximal in the set $\left\{\left|\theta(\calV)_{i',j',k}(\sigma)\right|: 0 \leq k \leq 2\right\}$.  Then $\BM_{\sigma}$ is a $\Q$--linear combination of classes $[\BM_{\sigma_{\iota}}]$ for $1 \leq \iota \leq s$ such that the following hold for all $1 \leq \iota \leq s$:
\begin{enumerate}
\item $\rk^{\calV}(v_{i,k}^{\sigma_{\iota}}) \leq 1$ for all $1 \leq i \leq 9$ and $0 \leq k \leq 2$,
\item the set $\left\{\left|\theta(\calV)_{i',j',k}(\sigma_{\iota})\right|:0 \leq k \leq 2\right\}$ is bounded above by $\left|\theta(\calV)_{i',j',1}(\sigma)\right|$ and has no more  elements equal to $\left|\theta(\calV)_{i',j',1}(\sigma)\right|$ than does the set $\left\{\left|\theta(\calV)_{i',j',k}(\sigma)\right|: 0 \leq k \leq 2\right\}$,
\item if the set $\{\left|\theta(\calV)_{i',j', k}(\sigma_{\iota})\right|: 0 \leq k \leq 2\}$ has as many elements equal to $\left|\theta(\calV)_{i',j',1}(\sigma)\right|$ as does the set $\{\left|\theta(\calV)_{i',j',k}(\sigma)\right|: 0 \leq k \leq 2\}$, then $\left|\theta(\calV)_{i',j',1}(\sigma_{\iota})\right| = \left|\theta(\calV)_{i',j',1}(\sigma)\right|$,
\item for every pair $1 \leq i < j \leq 9$ with $\left|\theta(\calV)_{i,j,k}(\sigma)\right| \leq 1$ for every $0 \leq k \leq 2$, we have $\left|\theta(\calV)_{i,j,k}(\sigma_{\iota})\right| \leq 1$ for all $0 \leq k \leq 2$, and
\item $g(\calH_1^{\sigma_{\iota}}) \geq 11$.
\end{enumerate}
\end{lemma}

\begin{proof}
If $g(\calH_1^{\sigma}) \geq 11$, then the lemma is trivially true by taking $s = 1$ and $\BM_{\sigma_{1}} = \BM_{\sigma}$.  Otherwise, since $g \geq \finalbound$ there is a $k = 0,2$ with $g(\calH_k^{\sigma}) \geq 21$.  Without loss of generality, assume that $g(\calH_0^{\sigma}) \geq 21$.  Since $g(\calH_0^{\sigma}) \geq 21$ and $\left|\calV\right| = 9$, there is a primitive subgroup $\calH \subseteq \calH_0^{\sigma}$ such that the following hold:
\begin{multicols}{2}
\begin{itemize}
\item $a \in \calH$, 
\item $v_{i,k}^{\sigma} \in \calH^{\perp}$ for all $1 \leq i \leq 8$, 
\item $g(\calH) - g(\calH_1^{\sigma}) \geq 9$,
\item $g(\calH) \geq 11$, and  
\item $2 *g(\calH) + 1 = \rk^{\calV}(\calH)$.
\end{itemize}
\end{multicols}
\bn Let $\widehat{\calH} = \{\calH_0^{\sigma}, \calH, \calH^{\perp} \cap \calH_1^{\sigma}, \calH_2^{\sigma}\}$.  We can describe $\widehat{\calH}$ using the graphical notation of Section~\ref{homolcurvequotsection} as in Figure~\ref{bmfinquotalggenusincreasefig}.  
\begin{figure}[h]
\begin{tikzpicture}
\node[anchor = south west, inner sep = 0] at (0,0){\includegraphics{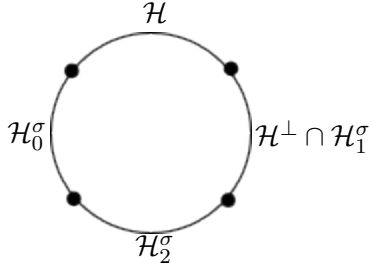}};
\node at (-0.3,1.3) {$\calH_0^{\sigma}$};
\node at (1.4,2.95) {$\calH$};
\node at (3.5,1.3) {$\calH^{\perp} \cap \calH_1^{\sigma}$};
\node at (1.4,-0.2) {$\calH_2^{\sigma}$};
\end{tikzpicture}
\caption{The graphical representation of $\widehat{\calH}$}\label{bmfinquotalggenusincreasefig}
\end{figure}

Now, let $z \subseteq X_g$ be the unique edge with $\calH \in \calH(z)$.  Let $y \subseteq X_g$ be the unique edge with $\calH^{\perp} \cap \calH_0^{\sigma} \in \calH(y)$.  Reorient $y$ and $z$ so that $yz$ is the unique edge with $\calH_0^{\sigma} \in \calH(yz)$.  Let $x \subseteq X_g$ be the unique edge with $\calH_2^{\sigma} \in \calH(x)$.  For $* = y,z$, let $\sigma_* \subseteq X_g$ be a 2--cell containing $x$ and $*$.  By Lemma~\ref{lemmaaltcohomolauxadd}, there is a relation
\begin{displaymath}
[\BM_{\sigma }] = [\BM_{\sigma_y}] + [\BM_{\sigma_z}]. 
\end{displaymath}
\bn We now show that $\BM_{\sigma_y}$ and $\BM_{\sigma_z}$ have the desired properties of the tori $\BM_{\sigma_{\iota}}$.

\p{$\BM_{\sigma_z}$ satisfies properties (1)-(5) of $\sigma_{\iota}$} Assume that $\calH(\sigma_z)$ is indexed so that $ \calH =\calH_0^{\sigma_z}$, $\calH^{\perp} \cap \calH_0^{\sigma} + \calH_1^{\sigma} = \calH_1^{\sigma_z}$, and $\calH_2^{\sigma} = \calH_2^{\sigma_z}$.  Now, we have chosen $\calH$ such that $\proj_{\calH \cap [b]^{\perp}}(v_i) = 0$ for all $1 \leq i \leq 9$.  Furthermore, for $k = 1$, we have $\proj_{\calH_1^{\sigma_z} \cap [b]^{\perp}}(v_i) = v_{i,0}^{\sigma} + v_{i,1}^{\sigma}$.  For $k = 2$, we have $\proj_{\calH_2^{\sigma_z} \cap [b]^{\perp}}(v_i) = v_{i,2}^{\sigma}$.  Each $v_{i,k}^{\sigma}$ is primitive by our hypothesis that $\rk^{\calV}(v_{i,k}^{\sigma}) \leq 1$, so each $v_{i,k}^{\sigma_z}$ is primitive as well.  Therefore $\BM_{\sigma_z}$ satisfies property (1) of $\sigma_{\iota}$.  Furthermore, the following hold for $\sigma_z$:
\begin{enumerate}
\item $\theta(\calV)_{i,j,0}(\sigma_z) = 0$ for all $1 \leq i,j \leq 9$,
\item $\left|\theta(\calV)_{i,j,1}(\sigma_z)\right| = \left|\theta(\calV)_{i,j,1}(\sigma) + \theta(\calV)_{i,j,0}(\sigma) \right|$ for all $1 \leq i, j \leq 9$, 
\item $\theta(\calV)_{i,j,2}(\sigma_z) = \theta(\calV)_{i,j,2}(\sigma)$ for all $1 \leq i,j \leq 9$, and
\item $g(\calH_1^{\sigma_z}) = g(\calH_1^{\sigma}) + g(\calH_0^{\sigma}) - g(\calH)$.
\end{enumerate}
\bn Now, observe that since $\langle v_i, v_j \rangle = 0$, if $1 \leq i < j \leq 9$ is a pair such that $\theta(\calV)_{i,j,k}(\sigma) \in \{-1,0,1\}$ for every $0 \leq k \leq 2$, then either $\theta(\calV)_{i,j,k}(\sigma) = 0$ for $0 \leq k \leq 2$, or $\{\theta(\calV)_{i,j,k}(\sigma): 0 \leq k \leq 2\} = \{-1,0,1\}$.  Therefore, relations (1), (2) and (3) in the above list for relations among $\theta(\calV)_{i,j,k}(\sigma_z)$ imply that $\BM_{\sigma_z}$ satisfies property (4) of the lemma.  For property (2), observe that the fact that $\langle v_{i'},v_{j'} \rangle = 0$ implies that 
\begin{displaymath}
\theta(\calV)_{i',j',0}(\sigma) + \theta(\calV)_{i',j',1}(\sigma) + \theta(\calV)_{i',j',2}(\sigma) = 0
\end{displaymath}
\bn  Since $\left|\theta(\calV)_{i',j',1}(\sigma)\right|$ is maximal among $\left|\theta(\calV)_{i',j',k}(\sigma)\right|$, the integers $\theta(\calV)_{i',j',1}(\sigma)$ and $\theta(\calV)_{i',j',0}(\sigma)$ must have opposite signs.  Therefore $\left|\theta(\calV)_{i',j',1}(\sigma_z)\right| = \left|\theta(\calV)_{i',j',1}(\sigma) + \theta(\calV)_{i',j',1}(\sigma)\right| \leq \left|\theta(\calV)_{i',j',1}(\sigma)\right|$, so relations (1), (2), and (3) imply that $\BM_{\sigma_z}$ satisfies property (2) of $\sigma_{\iota}$.  For property (3), relations (1)-(3) imply that  the only way for the set $\left\{\left|\theta(\calV)_{i',j',k}(\sigma_{z})\right|: 0 \leq k \leq 2\right\}$ to have as many elements equal to $\left|\theta(\calV)_{i',j',1}(\sigma)\right|$ as does $\left\{\left|\theta(\calV)_{i',j',k}(\sigma)\right|: 0 \leq k \leq 2\right\}$ is for $\theta(\calV)_{i',j',1}(\sigma_z) = \theta(\calV)_{i',j',1}(\sigma)$, so property (3) must hold.  Then since $g(\calH_0^{\sigma}) - g(\calH) \geq 21 - 11 = 10$ by assumption, relation (4) in the list of relations above implies that $g(\calH_1^{\sigma_z}) = g(\calH_0^{\sigma}) - g(\calH) + g(\calH_1^{\sigma}) \geq 10 + 1 \geq 11$, so $\sigma_z$ satisfies property (5) of $\sigma_\iota$.  

\p{$\BM_{\sigma_y}$ satisfies hypothesis (1)-(5) of $\sigma_{\iota}$} Assume that $\calH(\sigma_y)$ is indexed so that $\calH^{\perp} \cap \calH_0^{\sigma} = \calH_0^{\sigma_y}$, $\calH + \calH_1^{\sigma} = \calH_1^{\sigma_y}$, and $\calH_2^{\sigma_y} = \calH_2^{\sigma}$.  We have $\proj_{\calH^{\perp} \cap \calH_0(\sigma) \cap [b]^{\perp}}(v_i) = v_{i,0}^{\sigma}$ for all $1 \leq i \leq 9$ by our choice of $\calH$.  Since we have $\calH_1^{\sigma_y} = \calH + \calH_1^{\sigma}$, our choice of $\calH$ implies that $\proj_{\calH_1^{\sigma_y} \cap [b]^{\perp}}(v_i) = v_{i,1}^{\sigma}$.  Then we also have $\calH_2^{\sigma_y} = \calH_2^{\sigma}$ by construction, so the following hold:
\begin{enumerate}
\item $\theta(\calV)_{i,j,k}(\sigma) = \theta(\calV)_{i,j,k}(\sigma_y)$ for all $1 \leq i \leq 9$ and $0 \leq k \leq 2$, and
\item $g(\calH_1^{\sigma_y}) = g(\calH_1^{\sigma})+ g(\calH)$.
\end{enumerate}
\bn Relation (1) implies that $\sigma_y$ satisfies hypotheses (1), (2), (3), and (4) of $\sigma_{\iota}$.  Then since we have assumed that $g(\calH) \geq 10$ and we must have $g(\calH_1^{\sigma}) \geq 1$, we have $g(\calH_1^{\sigma_y}) \geq 1 + 10 = 11$, so $\BM_{\sigma_y}$ satisfies property (4) of $\BM_{\sigma_{\iota}}$.

\bn Now, we have shown that $[\BM_{\sigma}] = [\BM_{x,y}] + [\BM_{x,z}]$.  The lemma now follows by taking $s = 2$ and $\tau_1 = \sigma_{y}$, $\tau_2 = \sigma_z$.  
\end{proof}

\bn We now show that we can rewrite $[\BM_{\sigma}]$ as a linear combination of classes $[\BM_{\tau}]$ with $\left|\theta(\calV)_{i,j,k}(\tau)\right|$ bounded.

\begin{lemma}\label{bmfinquotalgint}
Let $g \geq \finalbound$, $a,b, M \subseteq S_g$ and $\calV \subseteq H_1(S_g;\Z)$ be as in Lemma~\ref{bmconjlemma}.  Let $\BM_{\sigma} \subseteq X_g$ be a Bestvina--Margalit torus such that $\rk^{\calV}(v_{i,k}^\sigma) \leq 1$ for every $1 \leq i \leq 9$ and $0 \leq k \leq 2$.  Then $\BM_{\sigma}$ is a linear combination of classes $[\BM_{\tau_s}]$ for $1 \leq s \leq m$ such that 
\begin{displaymath}
\left|\theta(\calV)_{i,j,k}(\tau_s)\right| \leq 1 \text{ and } \rk^{\calV}(v_{i,k}^{\tau_s}) \leq 1
\end{displaymath}
\bn for every $1 \leq i\leq j \leq 9$, $0 \leq k \leq 2$ and $1 \leq s \leq m$.
\end{lemma}

\begin{proof}
The proof follows by double induction on $i$ and $j$.  In particular, for some $(i,j)$, we assume that $\BM_{\sigma}$ has the property that 
\begin{displaymath}
\left| \theta(\calV)_{i',j',k}(\tau_s) \right| \leq 1
\end{displaymath}
\bn for every $0 \leq k \leq 2$ and for every $(i',j') < (i,j)$, where $<$ is the dictionary ordering.  We will show that $[\BM_{\sigma}]$ is a linear combination of fundamental classes $[\BM_{\tau_1}],\ldots ,[\BM_{\tau_m}]$ such that
\begin{displaymath}
\left| \theta(\calV)_{i',j',k}(\tau_s) \right| \leq 1
\end{displaymath}
\bn for all $0 \leq k \leq 2$, for all $1 \leq s \leq m$ and for all $(i',j') \leq (i,j)$, and $\rk^{\calV}(v_{i',k}^{\tau_s}) \leq 1$ for all $1 \leq i' \leq 9$ and $0 \leq k \leq 2$.

\p{Base case: $(i,j) = (1,1)$} We have $\left| \theta(\calV)_{1,1,k}(\sigma)\right| = 0$.

\p{Inductive step: $(i,j) > (1,1)$} If $j =i $ then we have $\theta(\calV)_{i,j,k}(\sigma) = 0$ for $0 \leq k \leq 2$, so assume $i < j$.  We will perform a double induction on the following quantities:
\begin{enumerate}
\item $\maxalg_{i,j}(\sigma) = \max_{0 \leq k \leq 2}\left\{\left|\theta(\calV)_{i,j,k}(\sigma)\right|\right\},$
\item $\nummaxalg_{i,j}(\sigma) = \left|\left\{0 \leq k \leq 2: \left|\theta(\calV)_{i,j,k}(\sigma)\right| = \maxalg_{i,j}(\sigma)\right\}\right|$.
\end{enumerate}
\bn Reindex $\calH(\sigma)$ so that $\left|\theta(\calV)_{i,j,1}(\sigma)\right|$ is maximal in the set $\{\left| \theta(\calV)_{i,j,k}(\sigma)\right|: 0 \leq k \leq 2\}$.  If $\left| \theta(\calV)_{i,j,1}(\sigma)\right| \leq 1$ we are done, so assume $\left| \theta(\calV)_{i,j,1}(\sigma)\right| \geq 2$.  Lemma~\ref{bmfinquotalggenusincrease} says that we can rewrite $[\BM_{\sigma}]$ as a linear combination of fundamental classes $[\BM_{\sigma_{\iota}}]$, where:
\begin{itemize}
\item  hypothesis (1) implies that $\rk^{\calV}_{i',k}(\sigma_{\iota})\leq 1$ for $1 \leq i' \leq 9$,
\item hypothesis (2) implies that $\theta(\calV)_{i',j',k}(\sigma_{\iota}) \leq 1$ for all $(i',j') < (i,j)$ and $0 \leq k \leq 2$,
\item hypothesis (3) implies that $\maxalg_{i,j}(\sigma_{\iota}) \leq \maxalg_{i,j}(\sigma)$,
\item hypothesis (3) implies that $\nummaxalg_{i,j}(\sigma_{\iota}) \leq \nummaxalg_{i,j}(\sigma_{\iota})$, 
\item hypothesis (4) implies that if both:
\begin{itemize}
\item $\maxalg_{i,j}(\sigma) = \maxalg_{i,j}(\sigma_{\iota})$ and 
\item $\nummaxalg_{i,j}(\sigma) = \nummaxalg_{i,j}(\sigma_{\iota})$,
\end{itemize}
\bn then $\left|\theta(\calV)_{i,j,1}(\sigma_{\iota})\right|$ is maximal among the set $\{\left|\theta(\calV)_{i,j,k}(\sigma_{\iota}) \right| 0 \leq k \leq 2 \}$, and 
\item hypothesis (4) implies $g(\calH_{1}^{\sigma_{\iota}}) \geq 11$.  
\end{itemize}
\bn Hence we may assume without loss of generality that $g(\calH_1^{\sigma}) \geq 10$.  Choose elements $v_{i,1}', v''_{i,1}, v'_{j,1}, v''_{j,1} \in \calH_1$ such that the following hold:
\begin{enumerate}
 \item $v_{i,1} = v'_{i,1} + v''_{i,1}$,
 \item $v_{j,1} = v'_{j,1} + v''_{j,1}$,
\item for all $(i,\ell)$ and $(\ell,j)$ with $(i,\ell), (\ell,j) < (i,j)$, we have
\begin{displaymath}
\theta(\calV)_{i,\ell,1}(\sigma) = \langle v'_{i,1}, v_{\ell,1}\rangle \text{ and } \theta(\calV)_{\ell,j,1}(\sigma) = \langle v_{\ell,1}, v'_{j,1}\rangle.
\end{displaymath}
\item $\left| \langle v'_{i,1}, v'_{j,1} \rangle \right| = \left| \theta(\calV)_{i,j,1}(\sigma) \right| - 1$,
\item $\langle v_{i,1}'', v_{j,1}'\rangle = \langle v_{j,1}'', v_{i,1}'\rangle = \langle v_{i,1}'', v_{i,1}'\rangle = \langle v_{j,1}'', v_{j, 1}'\rangle = 0$,
\item $\langle v_{i,1}'', v_{\ell,1}\rangle = \langle v_{j,1}'', v_{\ell,1}\rangle = 0$ for all $\ell \neq i,j$, and 
\item $\left| \langle v''_{i,1}, v''_{j,1} \rangle \right| = 1$
\end{enumerate}
\bn Such elements exist since $g(\calH) \geq 11$ and $\left|\calV\right| = 9$, so we can find $v''_{i,1}, v_{j,1}''$ with their desired algebraic intersections, and then $v'_{i,1} = v_{i,1} - v''_{i,1}$, $v_{j,1}' = v_{j,1} - v''_{j,1}$.  Since $g(\calH_1) \geq 11$ and $v_{i,1}$, $v_{j,1}$ are primitive by hypothesis, we may choose primitive subgroups $\calH'_1, \calH''_1 \subseteq \calH_1$ such that the following hold: 
\begin{enumerate}
\item for any $x \in \calH'_1$ and $y \in \calH''_1$, we have $\langle x, y \rangle = 0$,
\item $\calH''_1 = (\calH_1')^{\perp} \cap \calH_1$,
\item $\calH'_1 + \calH''_1 = \calH_1$,
\item $\calH_1' \cap \calH_1'' = \Z[a]$,
\item $v'_{i,1}, v'_{j,1}, v_{\ell,1} \in \calH'_1$ for all $\ell \neq i,j$ and
\item $v''_{i,1}, v''_{j,1} \in \calH_1''$.  
\end{enumerate}
\bn Let $\widehat{\calH} = \{\calH_0^{\sigma}, \calH_1', \calH''_1, \calH_2\}$.  We can describe $\widehat{\calH}$ using the graphical notation of Section~\ref{homolcurvequotsection} as in Figure~\ref{bmfinquotalgintfig}.  
\begin{figure}[h]
\begin{tikzpicture}
\node[anchor = south west, inner sep = 0] at (0,0){\includegraphics{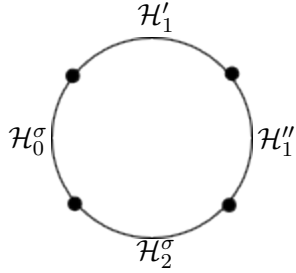}};
\node at (-0.3,1.3) {$\calH_0^{\sigma}$};
\node at (1.4,2.95) {$\calH'_1$};
\node at (3,1.3) {$\calH''_1$};
\node at (1.4,-0.2) {$\calH_2^{\sigma}$};
\end{tikzpicture}
\caption{The graphical representation of $\widehat{\calH}$}\label{bmfinquotalgintfig}
\end{figure}

Let $\BM_{\sigma'}$ denote the Bestvina--Margalit torus corresponding to $\{\calH_0, \calH_1', \calH_1'' + \calH_2\}$ and let $\BM_{\sigma''}$ denote the Bestvina--Margalit torus corresponding to $\{\calH_0, \calH''_1, \calH_1' + \calH_2\}$.  We will compute $\theta(\calV)_{i',j',k}(\sigma')$ and $\theta(\calV)_{i',j',k}(\sigma'')$ as follows.

\p{Computing $\theta(\calV)_{i',j',k}(\sigma')$} By our choice of $\calH_1'$, we have
\begin{enumerate}
\item $v_{\ell,0}^{\sigma'} = v_{\ell,0}^{\sigma}$ for all $1 \leq \ell \leq 9$,
\item $v_{\ell,1}^{\sigma'} = v_{\ell,1}^{\sigma}$ for all $ \ell \neq i,j$,
\item $v_{\ell,1}^{\sigma'} = v_{\ell,1}'$ for $\ell = 1,j$,
\item $v_{\ell,2}^{\sigma'} = v_{\ell,2}^{\sigma}$ for all $ \ell \neq i,j$, and
\item $v_{\ell,2}^{\sigma'} = v_{\ell,2} + v''_{\ell,1}$ for $\ell = 1,j$.
\end{enumerate}
\bn  Now, computation (1) implies that $\theta(\calV)_{i,j,0}(\sigma') = \theta(\calV)_{i,j,0}(\sigma)$ for all $1 \leq i,j \leq 9$.  Computations (2) and (3) and the fact that $ \langle v_{i,1}', v_{\ell,1}^{\sigma} \rangle = \langle v_{i,1}, v_{\ell,1}^{\sigma}\rangle$ and $\langle v_{\ell,1}^{\sigma}, v_{j,1}'\rangle = \langle v_{\ell,1}^{\sigma}, v_{j,1}\rangle$ for $\ell \neq i,j$ tell us  that $\theta(\calV)_{i',j',1}(\sigma') = \theta(\calV)_{i',j',1}(\sigma)$ for all pairs $(i',j') \neq (i,j)$.  Computation (3), our assumption that $\theta(\calV)_{i,j,1}(\sigma) \geq 1$, and the fact that $\langle v'_{i,1}, v'_{j,1}\rangle = \theta(\calV)_{i,j,1}(\sigma) - 1$ implies that $\left|\theta(\calV)_{i,j,1}(\sigma')\right| < \left|\theta(\calV)_{i,j,1}(\sigma)\right|$.  For $k = 2$, computations (4) and (5) say that $\theta(\calV)_{i',j',2}(\sigma') = \theta(\calV)_{i',j',2}(\sigma)$ for all $1 \leq i',j' \leq 9$.  Finally, we see that $\left|\theta(\calV)_{i,j,2}(\sigma')\right| < \left|\theta(\calV)_{i,j,2}(\sigma)\right|$ if $\left|\theta(\calV)_{i,j,2}(\sigma)\right| \geq 1$.  Indeed, since $\theta(\calV)_{i,j,1}(\sigma)$ is positive and has 
$\left|\theta(\calV)_{i,j,1}(\sigma)\right| \geq \left|\theta(\calV)_{i,j,2}(\sigma)\right|$, the fact that $\langle v_i, v_j \rangle = 0$ allows us to conclude that $\theta(\calV)_{i,j,2}(\sigma) \leq 0$.  Then $\langle v_{i,1}'', v_{j,1}''\rangle = 1$, so $\left|\theta(\calV)_{i,j,2}(\sigma')\right| = \left|\theta(\calV)_{i,j,2}(\sigma) + 1\right| < \left|\theta(\calV)_{i,j,2}(\sigma)\right|$ if $\left|\theta(\calV)_{i,j,2}(\sigma)\right| \leq 1$ by computation (5).  Hence $\sigma'$ has either $\maxalg(\sigma') < \maxalg(\sigma)$ or $\maxalg(\sigma') = \maxalg(\sigma)$ and $\nummaxalg(\sigma') < \nummaxalg(\sigma)$, so $\sigma'$ satisfies the inductive hypothesis for the induction on $\maxalg$ and $\nummaxalg$.

\p{Computing $\theta(\calV)_{i',j',k}(\sigma'')$} By our choice of $\calH_1''$, we have
\begin{enumerate}
\item $v_{\ell,0}^{\sigma''} = v_{\ell,0}^{\sigma}$ for all $1 \leq \ell \leq 9$,
\item $v_{\ell,1}^{\sigma''} = 0$ for all $ \ell \neq i,j$,
\item $v_{\ell,1}^{\sigma''} = v_{\ell,1}''$ for $\ell = 1,j$,
\item $v_{\ell,2}^{\sigma''} = v_{\ell,2}^{\sigma} + v_{\ell,1}^{\sigma}$ for all $ \ell \neq i,j$, and
\item $v_{\ell,2}^{\sigma''} = v_{\ell,2} + v'_{\ell,1}$ for $\ell = 1,j$.
\end{enumerate}

\bn Then we have $\theta(\calV)_{i',j',0}(\sigma'') = \theta(\calV)_{i',j',0}(\sigma)$ for all $1 \leq i',j' \leq 9$ by computation (1), $\theta(\calV)_{i',j',0}(\sigma'') = 0$ for $(i',j') \neq (i,j)$ by computation (2), and $\theta(\calV)_{i,j,1}(\sigma'') = 1$ by computation (3).  Then as in the previous paragraph, we see that $\theta(\calV)_{i',j',1}(\sigma)$ and $\theta(\calV)_{i',j',2}(\sigma)$ have opposite signs, so for any $1 \leq i',j' \leq 9$ we have either $\left|\theta(\calV)_{i',j',1}(\sigma'')\right| < \left|\theta(\calV)_{i',j',1}(\sigma)\right|$ or $\left|\theta(\calV)_{i',j',1}(\sigma'')\right| \leq 1$.  Hence $\sigma''$ satisfies the inductive hypothesis as well.

Since $[\BM_{\sigma}] = [\BM_{\sigma'}] + [\BM_{\sigma''}]$ by Lemma \ref{bmfinquotfundclassadd}, the inductive hypothesis applied to $\BM_{\sigma'}$ and $\BM_{\sigma''}$ completes the proof.
\end{proof}

\bn We now complete the section.

\begin{proof}[Proof of Lemma~\ref{bmfinquotlemma}]
By Lemmas~\ref{bmfincokprimred} and~\ref{bmfinquotalgint}, any fundamental class $[\BM_\tau]$ is a linear combination of fundamental classes of Bestvina--Margalit tori $\BM_{\sigma}$ with $\left|\rk^{\calV}(v_{i,k}^\sigma)\right| \leq 1$ and $\left|\theta(\calV)_{i,j,k}(\sigma)\right| \leq 1$ for every $1 \leq i,j \leq 8$ and $0 \leq k \leq 2$.  Hence by Lemma~\ref{bmconjlemma}, the vector space $H_0(G;\BM_2(X_g;\Q))$ is finite dimensional.
\end{proof}

\section{Finiteness of $H_2(\cC_{[a]}(S_g)/\cI_g;\Q)$}\label{finquotsectionpt2}

In this section, we will finish the proof of Proposition~\ref{homolcurvequotprop}.  If $w \in [a]^{\perp}$, let $X_g^w$ denote the full subcomplex of $X_g$ generated by edges $e$ such that the decomposition $\calH(e)$ is \textit{compatible} with $w$, i.e., $w \in \calH_i^e$ for some $\calH_i^e \in \calH(e)$.  The bulk of Section~\ref{finquotsectionpt2} is devoted to proving the following lemma.

\begin{lemma}\label{vstablemma}
Let $g \geq \finalbound$ and $a \subseteq S_g$ a nonseparating simple closed curve.  Let $c \subseteq S_g \cut a$ be a nonseparating simple closed curve, and let $w = [c]$.  The cokernel of the pushforward map $H_2(X_g^w;\Q) \rightarrow H_2(X_g;\Q)$ is a subquotient of $\BM_2(X_g;\Q)$.
\end{lemma}  

\bn For the remainder of this section, fix a $g$, $a$, and $c$ as in Lemma~\ref{vstablemma}.  Let $X_g^{w,2} \subseteq X_g$ be the subcomplex given by the union of all cells $\sigma$ such that $\dim(\sigma \cap X_g^w) \geq \dim(\sigma) - 1$.  Lemma~\ref{vstablemma} will proceed in two steps.

\p{Step (1)} We show that $\cok(H_2(X_g^{w,2};\Q) \rightarrow H_2(X_g;\Q))$ is generated by the images of fundamental classes of Bestvina--Margalit tori.

\p{Step (2)} We show that $\cok(H_2(X_g^w;\Q) \rightarrow H_2(X_g^{w,2};\Q))$ is generated by the images of fundamental classes of Bestvina--Margalit tori.  

\p{The outline of the proof of Lemma~\ref{bmfingenlemma}} Step (1) is recorded as Lemma~\ref{twotofullcoklemma}, and Step (2) is recorded as Lemma~\ref{vertstabcoklemma}.  We will prove Lemma~\ref{twotofullcoklemma} in Section~\ref{twotofullsubsection} and  Lemma~\ref{vertstabcoklemma} on Section~\ref{vstabsubsection}.  Additionally in Section~\ref{vstabsubsection}, we will assemble Lemmas~\ref{twotofullcoklemma} and~\ref{vertstabcoklemma} into the proof of Lemma~\ref{vstablemma}.  We will conclude with Section~\ref{homolcurvequotsubsection}, where we prove Lemma~\ref{bmfingenlemma} and Proposition~\ref{homolcurvequotprop}.

\subsection{Step (1) of the proof of Lemma~\ref{vstablemma}}\label{twotofullsubsection}
This step is recorded as the following lemma.

\begin{lemma}\label{twotofullcoklemma}
Let $g \geq \finalbound$ and $a \subseteq S_g$ a nonseparating simple closed curve.  Let $w \in [a]^{\perp}$ be a nonzero primitive homology class such that the image of $w$ under the adjoint map $[a]^{\perp} \rightarrow \ho_{\Z}([a]^{\perp},\Z)$ is nontrivial.  The cokernel of the pushforward map $H_2(X_g^{w,2};\Q) \rightarrow H_2(X_g;\Q)$ is generated by images of fundamental classes of Bestvina--Margalit tori.
\end{lemma}

\bn Before proving Lemma~\ref{twotofullcoklemma}, we will prove a collection of auxiliary lemmas.  The main goal is to prove Lemma \ref{bmcokpseudoproductlemma}, which describes the cokernel of the $H_2$--pushforward of certain subcomplexes of $X_g$.  Let $g, a$ and $w$ be as in the statement of Lemma~\ref{twotofullcoklemma}.  If $e \subseteq X_g$ is an edge with $w$ incompatible with $\calH(e)$ let $\widehat{U}_e$ denote the union of all 3-cells $\tau \subseteq X_g$ such that:
\begin{itemize}
\item  $e \subseteq \tau$,
\item  there is $\calH_0 \in \calH(e)$ such that $\calH_0 \in \calH(\tau)$, 
\item there is a unique edge $e' \subseteq \tau$ with $e' \subseteq X_g^w$, and
\item there is some $\calH' \in \calH(\tau)$ with $\calH' \in \calH(e')$.
\end{itemize}
\bn  Let $U_e$ denote the union of all 2--cells $\sigma \subseteq \widehat{U}_e$ such that the following hold:
\begin{itemize}
\item $\dim(\sigma \cap e) = 0$ and
\item $\dim(\sigma \cap X_g^w) = 1$.
\end{itemize}

\bn We will prove the following lemma.

\begin{lemma}\label{bmcokpseudoproductlemma}
Let $g \geq \finalbound$, $a \subseteq S_g$ and $w \in [a]^{\perp}$ be as in Lemma~\ref{vstablemma}.  Let $e,f \subseteq X_g$ be two edges such that $\calH(e)$ and $\calH(f)$ are incompatible with $w$.  Furthermore, assume that $g(e) = g(f) = \{1, g-2\}$.  Then the following hold:
\begin{enumerate}
\item the cokernel of the pushforward map
\begin{displaymath}
H_2(U_e;\Q) \rightarrow H_2(\widehat{U}_e;\Q)
\end{displaymath}
\bn is generated by the images of fundamental classes of Bestvina--Margalit tori,
\item the pushforward
\begin{displaymath}
H_1(U_e;\Q) \rightarrow H_1(\widehat{U}_e;\Q)
\end{displaymath}
\bn is an isomorphism, and
\item the pushforward $H_1(U_e \cap U_f;\Q) \rightarrow H_1(\widehat{U}_e \cap \widehat{U}_f;\Q)$ is surjective.
\end{enumerate}
\end{lemma}

\bn We begin by proving the following Lemma.

\begin{lemma}\label{bm3torilemma}
Let $\tau \subseteq X_g$ be a 3--cell.  Let $T$ denote the union of all $\tau' \subseteq X_g$ such that $\calH(\tau) = \calH(\tau')$ as unordered sets.  Then $T$ is a 3--torus.
\end{lemma}

\begin{proof}
This is the 3--dimensional analogue of the standard description of the 2--torus as a union of two 2--cells.
\end{proof}

\bn If $T$ is a torus as in the statement of Lemma \ref{bm3torilemma}, we will use $\calH(T)$ to denote the decomposition $\calH(\tau)$, except with the order forgotten.  In order to prove Lemma \ref{bmcokpseudoproductlemma}, we will prove the following result.

\begin{lemma}\label{threetoruslemma}
Let $g \geq \finalbound$, $a \subseteq S_g$ be a nonseparating simple closed curve, and $w \in [a]^{\perp}$ as in Lemma~\ref{vstablemma}.  Let $e \subseteq X_g$ be an edge with $\calH(e)$ not compatible with $w$.  Let $\tau, \tau' \subseteq \widehat{U}_e$ be two 3--cells.  Let $T$ denote the union of all 3--cells $\tau'' \subseteq \widehat{U}_e$ such that $H(\tau)= H(\tau'')$ as unordered sets, and similarly $T'$ and $\tau'$.  Then the following hold:
\begin{enumerate}
\item the pushforward $H_1(T \cap U_e;\Q) \rightarrow H_1(T;\Q)$ is an isomorphism and the cokernel $H_2(T \cap U_e;\Q) \rightarrow H_2(T;\Q)$ is generated by fundamental classes of Bestvina--Margalit tori,
\item if $T \cap T'$ contains an edge $f$ such that $f \neq e$ and $f \not \subseteq X_g^{w}$, then pushforward $H_1(T \cap T' \cap U_e;\Q) \rightarrow H_1(T \cap T';\Q)$ is a surjection, and
\item if $T \cap T'$ does not contain an $f$ as above, then the cokernel of the pushforward $H_1(T \cap T' \cap U_e;\Q) \rightarrow H_1(T \cap T';\Q)$ is generated by the image of the class $[e]$.
\end{enumerate}
\end{lemma}

\begin{proof}
We prove each case in turn.

\p{Case (1) } Let $\calH(\tau) = \{\calH_0, \calH_1, \calH_2, \calH_3\}$ such that $\calH_0 \in \calH(e)$.  Since $T$ is independent of the cyclic ordering on $\calH(\tau)$ by definition, we may assume without loss of generality that the decomposition of $[a]^{\perp}$ given by $\{\calH_2, \calH_2^{\perp}\}$ is compatible with $w$.  Given this, the intersection $T \cap U_e$ is given by 2--cells in $X_g$ with the following decompositions: $\{\calH_0 + \calH_1, \calH_2, \calH_3\}$, $\{\calH_0 + \calH_3, \calH_1, \calH_2\}$, $\{\calH_2, \calH_0 + \calH_1, \calH_3\}$, and $\{\calH_0 + \calH_3, \calH_2, \calH_1\}$.  This is a pair of Bestvina--Margalit tori corresponding to the unordered decompositions $\{\calH_0+ \calH_1, \calH_2, \calH_3\}$ and $\{\calH_1, \calH_2, \calH_0 + \calH_3\}$ that intersect in the edge $\{\calH_2, \calH_2^{\perp}\}$, so the map $H_1(T \cap U_e;\Q) \rightarrow H_1(T;\Q)$ is an isomorphism and the cokernel of $H_2(T \cap U_e;\Q) \rightarrow H_2(T;\Q)$ is generated by the image of the fundamental class $[\BM_{\sigma}]$ with $\calH(\sigma) = \{\calH_0, \calH_1 + \calH_2, \calH_3\}$.

\p{Case (2)} Reuse the indexing of the elements of $\calH(\tau)$.  Let $\calH(\tau') = \{\calH_0, \calH_1, \calH_2', \calH_3'\}$ where $\calH(f) = \{\calH_1, \calH_1^{\perp}\}$.  Then by the construction of $T \cap U_e$ in the previous part, we see that $T \cap T' \cap U_e$ also contains the edge corresponding to the splitting $\{\calH_0 + \calH_1, \calH_1^{\perp} \cap \calH_0^{\perp}\}$.  Therefore the class $[e] \in H_1(T \cap T';\Q)$ is in the image of the pushforward $H_1(T \cap T' \cap U_e;\Q) \rightarrow H_1(T \cap T';\Q)$, so the pushforward map is surjective.

\p{Case (3)} Reuse the indexing of $\calH(\tau)$ from the previous case.  If $\calH_2 \in \calH(\tau')$, then $T \cap T \cap U_e$ is the single edge corresponding to $\{\calH_2, \calH_2^{\perp}\}$, while the intersection $T \cap T$ is the Bestvina--Margalit torus containing this edge and $e$.  Otherwise, $T \cap T' = e$ and so $T \cap T' \cap U_e$ is the unique vertex of $X_g$.  In either case the lemma holds.
\end{proof}

\begin{proof}[Proof of Lemma~\ref{bmcokpseudoproductlemma}]
If $\sigma \subseteq U_e$ is a 2--cell with $e, \sigma \subseteq \tau \subseteq \widehat{U}_e$ for some 3--cell $\tau$, then the union
\begin{displaymath}
\bigcup_{\tau' \subseteq \widehat{U}_e: \calH(\tau') = \tau} \tau'
\end{displaymath}
\bn is a 3--torus.  Let $\widehat{\calR}_e$ denote the set of such 3--tori.  This set is a simplicial cover of $\widehat{U}_e$ by construction.  Let
\begin{displaymath}
\calR_e = \{T \cap U_e: T \in \widehat{\calR}_e\}.
\end{displaymath}
\bn By the definition of $U_e$, $\calR_e$ is a cover of $U_e$.  Let $\bE_{p,q}^r(\calR_e;\Q)$ and $\bE_{p,q}^r(\widehat{\calR}_e;\Q)$ denote the $\Cech$--to--singular spectral sequences for $\calR_e$ covering $U_e$ and $\widehat{\calR}_e$ covering $\widehat{U}_e$ respectively.  Then, for $k = 1,2$ the cokernel of $H_k(U_e;\Q) \rightarrow H_k(\widehat{U}_e;\Q)$ is noncanonically isomorphic to
\begin{displaymath}
\bigoplus_{p+q = k} \cok(\bE_{p,q}^\infty(\calR_e;\Q) \rightarrow \bE_{p,q}^{\infty}(\widehat{\calR}_e;\Q)).
\end{displaymath}
\bn Hence to prove the first two statements in the lemma, it suffices to prove the following.
\begin{enumerate}
\item The vector space $\cok(\bE_{0,2}^\infty(\calR_e;\Q) \rightarrow \bE_{0,2}^\infty(\widehat{\calR}_e;\Q))$ is generated by fundamental classes of Bestvina--Margalit tori, and the maps $\bE_{p,q}^\infty(\calR_e;\Q) \rightarrow \bE_{p,q}(\widehat{\calR}_e;\Q)$ are surjective for $p > 0$ and $p +q = 2$.  
\item The maps $\bE_{p,q}^\infty(\calR_e;\Q) \rightarrow \bE_{p,q}^{\infty}(\widehat{\calR}_e;\Q)$ are surjective for $p + q = 1$.  
\end{enumerate}

\bn We will prove each of these in turn.

\p{The proof of statement (1)} We first show that $\cok(\bE_{0,2}^\infty(\calR_e;\Q) \rightarrow \bE_{0,2}^\infty(\widehat{\calR}_e;\Q)$ is generated by the images of fundamental classes of Bestvina--Margalit tori.  Since the differentials out of $\bE_{0,2}^r(\calR_e;\Q)$ and $\bE_{0,2}^r(\widehat{\calR}_e;\Q)$ vanish for all $r \geq 1$, it suffices to show that $\cok(\bE_{0,2}^1(\calR_e;\Q) \rightarrow \bE_{0,2}^1(\widehat{\calR}_e;\Q))$ is generated by the images of fundamental classes of Bestvina--Margalit tori.  Hence it is enough to show that for any $T \in \widehat{\calR}_e$, the cokernel $H_2(T \cap U_e;\Q) \rightarrow H_2(T;\Q)$ is generated by the images of fundamental classes of Bestvina--Margalit tori.  This follows from case (1) of Lemma \ref{bmcokpseudoproductlemma}.

We now show that $\bE_{p,q}^\infty(\calR_e;\Q) \rightarrow \bE_{p,q}^\infty(\widehat{\calR}_e;\Q)$ is surjective for $p > 0$ and $p + q = 2$.  Consider the map of chain complexes
\begin{displaymath}
\bE_{*,1}^1(\calR_e;\Q) \rightarrow \bE_{*,1}^1(\widehat{\calR}_e;\Q).
\end{displaymath}
\bn By Lemma \ref{bmcokpseudoproductlemma}, the map $\bE_{0,1}^1(\calR_e;\Q) \rightarrow \bE_{*,1}^1(\widehat{\calR}_e;\Q)$ is injective. Hence it suffices to show that the vector space $H_1( \bE_{*,1}^1(\widehat{\calR}_e;\Q), \bE_{*,1}^1(\widehat{\calR}_e;\Q))$ is the zero space.  Let $C_*$ denote the quotient complex $\bE_{*,1}^1(\widehat{\calR}_e;\Q)/\bE_{*,1}^1(\calR_e;\Q)$.  By Lemma \ref{bmcokpseudoproductlemma}, if $T \cap T' \cap U_e$ contains an edge $f \not \subseteq X_g^w$, then the image of a class in $\bE_{*,1}^1(\widehat{\calR}_e;\Q)$ supported on the index $T \cap T'$ is 0 in $C_*$.  Then if no such $f$ exists, Lemma \ref{bmcokpseudoproductlemma} says that the cokernel $H_1(T \cap T'\cap U_e;\Q) \rightarrow H_1(T \cap T';\Q)$ is generated by the image of the class $[e]$.  Therefore it suffices to show that if $T, T' \in \widehat{\calR}_e$, then there is a sequence $T_0, \ldots, T_n \in \widehat{\calR}_e$ such that $T = T_0$, $T' = T_n$, and $T_i \cap T_{i+1}$ contains an edge $f_i \neq e$ and $f \not \subseteq X_g^w$.  This will imply that the element in $C_*$ given by $[e]$ supported on $T \cap T'$ is homologous to the element $[e]$ supported on $T_0 \cap T_1 + \ldots + T_{n-1} \cap T_{n}$, since the edge connecting $T$ to $T'$ is homologous to the path $T_0 \rightarrow T_1 \rightarrow \ldots \rightarrow T_n$ in the clique complex on the set $\widehat{\calR}_e$.  Since the latter classes are all trivial in $C_*$, this implies that $T \cap T'$ is trivial in $H_1(C_*)$ as well.

Let $\calH(e) = \{\calH_0, \calH_0^{\perp}\}$ such that $g(\calH_0) = g-2$.  Let $\calH(T) = \{\calH_0, \calH_1, \calH_2, \calH_3\}$ such that $\{\calH_3, \calH_3^{\perp}\}$ is compatible with $w$.  Choose $\calH_1' \subseteq \calH_1$ such that $g(\calH_1')= 1$, $\{\calH_1', (\calH_1')^{\perp}\}$ is incompatible with $w$, and $w^{\perp} \supseteq (\calH_1')^{\perp} \cap \calH_1$.  Let $\calH(T_1) = \{\calH_0, \calH_1', \calH_2, \calH_3 + (\calH_1 \cap (\calH_1')^{\perp})\}$.  Note that $T_1 \cap T_0$ contains the edge $\{\calH_2, \calH_2^{\perp}\}$ which is not $e$ and is not in $X_g^w$.  Now, since $g(\calH_1') = 1$ and $g \geq \finalbound$, there is some $\calH_1'' \in \calH(T')$ such that there is an edge $h \subseteq U_e$ with $h \not \subseteq X_g^w$ such that for some $\calH_h \in \calH$, we have $\calH_h \supseteq \calH_1', \calH_1''$, $\calH_h \subseteq \calH_0^{\perp}$, and $w \in \calH_h^{\perp} \cap \calH_0^{\perp}$.  Let $T_2 = \{\calH_0, \calH_1', \calH_h, \calH_h^{\perp} \cap \calH_0^{\perp} \cap (\calH_1')^{\perp}\}$ and $T_3 = \{\calH_0, \calH_1'', \calH_h, \calH_h^{\perp} \cap \calH_0^{\perp} \cap (\calH_1'')^{\perp}\}$.  Then $T_1 \cap T_2$ contains the edge $\{\calH_1', (\calH_1')^{\perp}\}$, $T_2 \cap T_3$ contains the edge $\{\calH_h, \calH_h^{\perp}\}$, and $T_3 \cap T'$ contains the edge $\{\calH_1'', (\calH_1'')^{\perp}\}$.  Hence the desired path between $T$ and $T'$ exists, so $H_1(C_*) = 0$ and thus statement (1) holds.  

For $p = 2$, $q = 0$, note that any intersection $T_0 \cap \ldots \cap T_k$ of $T_0, \ldots, T_k \in \widehat{\calR}_e$ contains the unique vertex of $X_g$, and in particular any such intersection is nonempty and connected.  The same holds for $\calR_e$, so we have $\bE_{2,0}^2(\calR_e;\Q) = \bE_{2,0}^2(\widehat{\calR}_e;\Q) = 0$, so the pushforward on the $r = \infty$ page is surjective.

\p{The proof of statement (2)}  As a consequence of Lemma \ref{bmcokpseudoproductlemma}, the map $\bE_{1,0}^1(\calR_e;\Q) \rightarrow \bE_{1,0}^1(\widehat{\calR}_e;\Q)$ is an isomorphism, and in particular surjective.  Hence the map $\bE_{1,0}^\infty(\calR_e;\Q) \rightarrow \bE_{1,0}^{\infty}(\widehat{\calR}_e;\Q)$ is surjective.  The case $p = 1$ and $q = 0$ follows by the same argument as $p = 2$ and $q = 0$, except with $2$ replaced by 1.

\p{The proof of statement (3)} The vector space $H_1(\widehat{U}_e \cap \widehat{U}_f;\Q)$ is generated by classes represented by edges, so it suffices to show that any class in $H_1(\widehat{U}_e \cap \widehat{U}_f;\Q)$ represented by an edge is in the image of the pushforward $H_1(U_e \cap U_f;\Q) \rightarrow H_1(\widehat{U}_e \cap \widehat{U}_f;\Q)$.  Now, by construction we have $U_e^{(1)} \cup e = \widehat{U}_e^{(1)}$, and similarly for $f$.  Hence if $h \subseteq \widehat{U}_e \cap \widehat{U}_f$ is an edge with $h \neq e,f$, the class $[h]$ is in the image of $H_1(U_e \cap U_f;\Q) \rightarrow H_1(\widehat{U}_e \cap \widehat{U}_f;\Q)$.  Hence it suffices to show that $[e], [f] \in H_1(U_e \cap U_f;\Q) \rightarrow H_1(\widehat{U}_e \cap \widehat{U}_f;\Q)$ if $e,f \subseteq \widehat{U}_e \cap \widehat{U}_f$.  If $e \subseteq \widehat{U}_e \cap \widehat{U}_f$, this implies that there is a 3--cell $\tau$ containing both $e$ and $f$.  But now every edge in $\tau$ besides $e$ and $f$ lies in both $U_e$ and $U_f$ by the above argument, so $[e]$ is a linear combination of classes in $U_e \cap U_f$, and the lemma holds.
\end{proof}

\bn We need one more auxiliary result before proving Lemma \ref{twotofullcoklemma}.

\begin{lemma}\label{onegenusreductionlemma}
Let $g \geq \finalbound$, $a \subseteq S_g$ be a nonseparating simple closed curve, and $w \in [a]^{\perp}$ such that $w$ is not sent to zero under the adjoint map $[a]^{\perp} \rightarrow \ho_{\Z}([a]^{\perp}, \Z)$.  Let $\sigma \subseteq X_g$ be a 2--cell such that $\sigma \not \subseteq X_g^{w,2}$.  Then $\sigma$ is homologous to a linear combination of 2--cells $\sigma_1,\ldots, \sigma_n$ such that each $\sigma_i \not \subseteq X_g^{w,2}$, and each $\sigma_i$ contains an edge $e_i \subseteq \sigma_i$ such that $\calH(e_i)$ is not compatible with $w$ and $g(e_i) = \{1, g-2\}$.
\end{lemma}

\begin{proof}
Inductively, it suffices to show that $\sigma$ is homologous to $\sigma_1 + \sigma_2 + \sigma_3 $ such that each $\sigma_i$ contains some edge $e_i \subseteq \sigma_i$ with $\min g(e_i) < g(\calH)$, where $\calH \in \calH(\sigma)$ is minimal among $\calH' \in \calH(\sigma)$ such that $\{\calH', (\calH')^{\perp}\}$ is not compatible with $w$.  Choose such an $\calH \subseteq \calH(\sigma)$.  Let $\calH = \calH_1 + \calH_2$ such that there are edges  $e_i \not \subseteq X_g^w$ with $\calH(e_i) = \{\calH_i, \calH_i^{\perp}\}$.  Let $\tau$ be the 3--cell containing $\sigma$, $e_1$ and $e_2$.  Then every other 2--cell in $\partial \tau$ besides $\sigma$ contains either $e_1$ or $e_2$.  We have $\min g(e_i) < g(\calH)$ and $e_i \not \subseteq X_g^w$, so the lemma holds.
\end{proof}

\bn We are now ready to move forward with Step (1) of the proof of Lemma~\ref{vstablemma}. 

\begin{proof}[Proof of Lemma~\ref{twotofullcoklemma}]
Let $\calE^w = \{e \subseteq X_g: \calH(e) \text{ is incompatible with } w \text{ and } g(e) = \{1, g-2\}\}$.  For each $e \in \calE^w$, let $U_e$ and $\widehat{U}_e$ be as above.  Let
\begin{displaymath}
\calU = \{U_e\}_{e \in \calE} \text{ and } \widehat{\calU} = \{\widehat{U}_e\}_{e \in \calE}.
\end{displaymath}
\bn Let $\calU^+ = \calU \cup \{X_g^w\}$ and $\widehat{\calU}^+ = \widehat{\calU} \cup \{X_g^w\}$.  The set $\calU^+$ covers $(X_g^{w,2})^{(2)}$, since any 2--cell $\sigma \subseteq X_g^{w,2}$ is either contained in $X_g^w$, or is not.  In the latter case, the cell $\sigma$ must be disjoint from some edge $e \in \calE^w$.  Then the pushforward $H_2(\bigcup_{\widehat{U} \in \widehat{\calU}} \widehat{U} \cup X_g^{w}) \rightarrow H_2(X_g;\Q))$ is a surjection as a consequence of Lemma \ref{onegenusreductionlemma}.  For either of the covers $* = \calU^+, \widehat{\calU}^+$, let $\bE_{p,q}^r(*;\Q)$ denote the $\Cech$--to--singular spectral sequence corresponding to the cover $*$.  The cokernel of the pushforward $H_2(X_g^{w,2};\Q) \rightarrow H_2(X_g;\Q)$ is noncanonically identified with the direct sum
\begin{displaymath}
\bigoplus_{p + q = 2} \cok(\bE_{p,q}^\infty(\calU^+;\Q) \rightarrow \bE_{p,q}^\infty(\widehat{\calU}^+;\Q)).
\end{displaymath}
\bn Therefore, it suffices to prove the following three facts:
\begin{enumerate}
\item The cokernel of the map $\bE_{0,2}^1(\calU;\Q) \rightarrow \bE_{0,2}^1(\widehat{\calU}, \Q)$ is generated by the images of fundamental classes of Bestvina--Margalit tori,
\item the cokernel of the map $\bE_{1,1}^2(\calU;\Q) \rightarrow \bE_{1,1}^2(\overline{\calU};\Q)$ is trivial, and
\item the cokernel of the map $\bE_{2,0}^3(\calU;\Q) \rightarrow \bE_{2,0}^3(\widehat{\calU};\Q)$ is trivial.
\end{enumerate}
\bn We prove each of these in turn.

\p{Proof of Fact (1)} This is the content of statement (1) of Lemma \ref{bmcokpseudoproductlemma}. 

\p{Proof of Fact (2)} By statements (2) and (3) respectively of Lemma \ref{bmcokpseudoproductlemma}, the maps $\bE_{0,1}^1(\calU^+;\Q) \rightarrow \bE_{0,1}^1(\widehat{\calU}^+;\Q)$ and $\bE_{1,1}^1(\calU^+;\Q) \rightarrow \bE_{1,1}^1(\widehat{\calU}^+;\Q)$ are an isomorphism and a surjection respectively.  Hence the map $\bE_{1,1}^2(\calU^+;\Q) \rightarrow \bE_{1,1}^2(\widehat{\calU}^+;\Q)$ is surjective.

\p{Proof of Fact (3)} For any choice of $e_0,\ldots, e_k \in \calE^w$, both $U_{e_0} \cap \ldots \cap U_{e_k}$ and $\widehat{U}_{e_0} \cap \ldots \cap \widehat{U}_{e_k}$ contain the unique vertex of $X_g$, and hence are connected.  The same applies if we include $X_g^w$, so $\bE_{2,0}^2(\calU^+;\Q) = \bE_{2,0}^2(\widehat{\calU}^+;\Q) = 0$.

\p{Completing the proof} Given the above three statements,  $\cok(H_2(X_g^{w,2};\Q) \rightarrow H_2(X_g;\Q))$ is generated by a quotient of the image $\bE_{0,2}^1(\widehat{\calU};\Q) \rightarrow \bE_{0,2}^1(\overline{\calU};\Q)$, and the image of this map  is generated by the images of the fundamental classes of Bestvina--Margalit tori, so the lemma holds.
\end{proof}

\subsection{Step (2) of the proof of Lemma~\ref{vstablemma}}\label{vstabsubsection}

We now prove Lemma~\ref{vertstabcoklemma}.  This will complete the proof of Lemma~\ref{vstablemma}.  We will also prove Lemma~\ref{edgeequivconnlemma}, which is an auxiliary result about the acyclicity of a complex where the vertices are edges $e \subseteq X_g^{w,2}.$

\begin{lemma}\label{vertstabcoklemma}
Let $g \geq \finalbound$ and $a \subseteq S_g$ a nonseparating simple closed curve.  Let $w \in [a]^{\perp}$ be a nonzero primitive homology class such that the image of $w$ under the adjoint map $\ho_{\Z}([a]^{\perp},\Z)$ is nontrivial.  The vector space
\begin{displaymath}
\cok(H_2(X_g^w;\Q) \rightarrow H_2(X_g^{w,2};\Q))
\end{displaymath}
\bn is generated by the images of fundamental classes of Bestvina--Margalit tori.
\end{lemma}
\bn We begin by defining an auxiliary complex.   Let $e \subseteq X_g$ be an edge with $g(e) = \{1, g-2\}$.  Let $A_e \subseteq H_1(X_g;\Q)$ be the affine space given by
\begin{displaymath}
A_e = \left\{[f] \in H_1(X_g;\Q): [f] - [e] \in H_1(X_g^w;\Q)\right\}.
\end{displaymath}
\bn Let $Y(e)$ denote the connected component containing the edge $e$ of the complex $C(e)$, where a $k$--cell of $C(e)$ is a set of \textit{ordered} $(k+1)$ edges $e_0,\ldots, e_k \subseteq X_g$ such that:
\begin{itemize}
\item $g(e_i) = \{1, g-2\}$ for every $0 \leq i \leq k$,
\item  $[e_i] \in A_e$ for every $0 \leq i \leq k$, and 
\item there is an edge $e'$ such that $[e'] \in A_e$ and such that $e'$ shares a 2--cell $\sigma_i \subseteq X_g^{w,2}$ with each $e_i$.
\end{itemize}

\p{Remark} Note that the cells of $Y(e)$ are ordered collections of vertices.  This is to avoid certain technical complications later in the section.

\medskip

\bn  We will prove the following auxiliary result.

\begin{lemma}\label{edgeequivconnlemma}
Let $g \geq \finalbound$ and $a \subseteq S_g$ a nonseparating simple closed curve.  Let $w \in [a]^{\perp}$ be a primitive nonzer class such that $w$ is not sent to zero under the adjoint map $[a]^{\perp} \rightarrow \ho_{\Z}([a]^{\perp}, \Z)$.  Let $e \subseteq X_g$ be an edge with $e \not \subseteq X_g^w$ and $g(e) = \{1,g-2\}$.  The complex $Y(e)$ is 1--acylic. 
\end{lemma}

\begin{proof}
Let $\calH(e) = \{V_0, V_1\}$.  Let $w_0, w_1$ be two nonzero elements in $[a]^{\perp}$ such that $w_i \in V_i$ and $w_0 + w_1 = w$.  We will prove the following.

\p{Claim} Let $f \subseteq X_g$ be an edge with $f$ a vertex of $Y(e)$ and let $\calH(f) = \left\{V_0^f, V_1^f\right\}$.  Then after possibly reindexing, we have $w_i \in V_i^f$ for $i = 0,1$.

\p{Proof of claim} Note that for any 2--cell $\sigma \subseteq X_g$ such that $w_0,w_1$ compatible with $\calH(\sigma)$, there is some $\calH \in \calH(\sigma)$ with $w_0,w_1 \in \calH$ or $\calH^{\perp}$, so such a $\sigma$ is contained in $X_{g}^{w,2}$.  Likewise, if $\sigma \subseteq X_g^{w,2}$ is a 2--cell containing $e$, then there is $\calH \in \calH(\sigma)$ compatible with $w$, so in particular we have $w_0,w_1$ both compatible with $\calH(\sigma)$.  Hence if $e_0,\ldots, e_k$ is a cell in $Y(e)$ with $e'$ as in the definition of $Y(e)$, then $e_0,\ldots, e_k$ and $e'$ are all compatible with $v_0$ and $v_1$, so the claim holds.

Now, given the claim, we see that since $g \geq \finalbound$, any triple of edges $e_0,e_1,e_2 \in Y(e)$ are the vertices of a 2--cell in $Y(e)$.  Indeed, for any three $e_0, e_1, e_2 \subseteq X_g$ with $e_0, e_1, e_2 \in Y(e)$, we have $w_0$ and $w_1$ compatible with $\calH(e_0)$, $\calH(e_1)$ and $\calH(e_2)$ by the claim.  Since $g \geq \finalbound$ and $g(e_k) = (1, g-2)$ for $k = 0,1,2$, we see that there is some primitive $\calH \subseteq [a]^{\perp}$ with $\calH^{\perp} \cap \calH = \Z[a]$, $\calH^{\perp} + \calH = [a]^{\perp}$, where $w_0 \in \calH$, $w_1 \in \calH^{\perp}$, and $\calH$ compatible with $\calH(e_k)$ for $k =0,1,2$.  Let $e' \subseteq X_g$ be the unique edge with $\calH \in \calH(e')$.  For each $e_k$, let $\sigma_k$ be a 2--cell containing $e_k$ and $e'$.  We see that since $w_0,w_1$ both compatible with $e_k$ and $e'$, then the third edge $z_k$ of $\sigma_k$ must have $w_0, w_1 \in \calH_0^{z_k}$ or $\calH_1^{z_k}$.  Hence $\sigma_k \in X_g^{w,2}$, so $[e'] \in A_e$ since $[e'] + [z_k] = [e_k]$ and $[z_k] \in H_1(X_g^w;\Q)$.  Therefore by the definition of $Y(e)$, we have that $e_0,e_1,e_2$ is a 2--cell.  This implies that $Y(e)$ is the 2--skeleton of a flag complex on the complete graph of the vertices of $Y(e)$, so in particular we have $H_1(Y(e);\Q) = 0$.
\end{proof}

\bn We are now ready to complete Step 2 of the proof of Lemma~\ref{vstablemma}.
\begin{proof}[Proof of Lemma~\ref{vertstabcoklemma}]
Let $\calE^w_1 = \{e \subseteq X_g: e \text{ is incompatible with } w, g(e) =\{1,g-2\}\}$.  Let $U_e \subseteq X_g^w$ be the subcomplex generated by all $2$--cells $\sigma \subseteq X_g^w$ such that $e$ and $\sigma$ are both faces of a $3$--cell $\tau$.  Let $\widehat{U}_e$ consist of the union of all $3$ cells $\tau$ as in the previous sentence.  For any such $\tau \subseteq \widehat{U}_e$, the union of $\tau$ with all $\tau_i$ that $H(\tau_i) = H(\tau)$ as unordered sets forms a $k+1$--torus, and this torus is naturally isomorphic to the product of $e$ and the minimal $k$--torus in $X_g$ containing $\sigma$. Hence, there is a natural isomorphism $\widehat{U}_e \cong e \times U_e$.  By construction, the collection
\begin{displaymath}
\overline{\calU} = \{\widehat{U}_e/U_e\}_{e \in \calE_1^w}
\end{displaymath}
\bn covers the 2--skeleton of $X_g^{w,2}/X_{w}$.  For any cover $* = \calU, \widehat{\calU}$, or $\overline{\calU}$, let
$\bE_{p,q}^r(*;\Q)$ denote the $\Cech$--to--singular spectral sequence corresponding to the cover $*$.  Since $\overline{\calU}$ covers the 2--skeleton of $X_g^{w,2}/X_{g}^w$, we have
\begin{displaymath}
\bigoplus_{p +q = 2}\bE_{p,q}^\infty(\overline{\calU};\Q) \rrightarrow H_{2}(X_g^{w,2}/X_{g}^w;\Q).
\end{displaymath}
\bn From the long exact sequence in homology for the pair $(X_g^{w,2}, X_g^w)$, we have an inclusion
\begin{displaymath}
\cok(H_2(X_g^w;\Q) \rightarrow H_2(X_g^{w,2};\Q)) \hookrightarrow H_2(X_g^{w,2}, X_g^{w};\Q).
\end{displaymath}
\bn  Then we have $H_2(X_g^{w,2}, X_g^w;\Q) = H_2(X_g^{w,2}/X_g^w;\Q)$.  Since $\bE_{p,q}^r(\overline{\calU};\Q) = 0$ for $p < 0$ or $q < 0$, the vector space $\bE_{0,2}^\infty(\overline{\calU};\Q)$ is a quotient of $\bE_{0,2}^1(\overline{\calU};\Q)$.  Therefore it is enough to prove the following three facts:
\begin{enumerate}
\item The image of $\bE_{0,2}^1(\widehat{\calU};\Q) \rightarrow \bE_{0,2}^1(\overline{\calU}, \Q)$ is generated by Bestvina--Margalit tori,
\item the vector space $\bE_{1,1}^2(\overline{\calU};\Q)$ is the $0$--space, and 
\item the vector space $\bE_{2,0}^2(\overline{\calU};\Q)$ is the $0$--space.
\end{enumerate}
\bn We prove each of these in turn.

\p{The proof of Fact (1)} Since $\widehat{U}_e \cong e \times U_e$, the cokernel $\cok(H_2(U_e;\Q) \rightarrow H_2(\widehat{U}_e;\Q))$ is isomorphic to $H_1(U_e;\Q) \otimes H_1(e;\Q)$ by the K\"unneth formula.  The tensor product of a class represented by an edge $f \subseteq U_e$ with $e$ is the Bestvina--Margalit torus containing $e$ and $f$, so Fact (1) holds.

\p{The proof of Fact (2)} Let $e_0,\ldots, e_k \in \calE^w$.  By construction, we have 
\begin{displaymath}
\dim\left(H_1\left(\widehat{U}_{e_0} \cap \ldots \cap \widehat{U}_{e_k}/U_{e_0}\cap \ldots \cap U_{e_k};\Q\right)\right) \leq 1
\end{displaymath}
\bn so $\bE_{*,1}^1(\overline{\calU};\Q)$ is the cellular chain complex of a simplicial complex $Z$, where the $k$--cells of $Z$ are sets of $k+1$ edges $e_0,\ldots, e_k \in \calE_1^w$ with 
\begin{displaymath}
\dim\left(H_1\left(\widehat{U}_{e_0} \cap \ldots \cap \widehat{U}_{e_k}/U_{e_0}\cap \ldots \cap U_{e_k};\Q\right)\right) = 1.
\end{displaymath}
\bn Now, note that if $U_{e_0},\ldots, U_{e_k}$ form a $k$--cell in $Z$, then there is an edge $f \subseteq X_g$ with $[f] \not \in H_1(X_g^w;\Q)$ such that $f$ and $e_i$ are two edges of a 2--cell $\sigma_i \subseteq X_g$ with the third edge in $X_g^w$.  Hence if $e \in \calE_1^w$, the path component $P_e$ of $Z$ containing $\widehat{U}_e/U_e$ has 2--skeleton canonically identified with the 2--skeleton of $Y(e)$.  Therefore $H_1(P_e;\Q) = 0$ by Lemma~\ref{edgeequivconnlemma}.   Therefore $H_1(Z;\Q) = 0$, so $\bE_{1,1}^2(\overline{\calU};\Q) = 0$ as desired. 

\p{The proof of Fact (3)} For any choice of $e_0,\ldots, e_k \in \calE^w$, both $U_{e_0} \cap \ldots \cap U_{e_k}$ and $\widehat{U}_{e_0} \cap \ldots \cap \widehat{U}_{e_k}$ contain the unique vertex of $X_g$, and hence are connected.  Therefore $\bE_{2,0}^2(\overline{\calU};\Q) = 0$.
\end{proof}

\bn We now prove Lemma~\ref{vstablemma}.

\begin{proof}[Proof of Lemma~\ref{vstablemma}]
There is a noncanonical surjection
\begin{align*}
\cok (H_2(X_g^w;\Q) & \rightarrow H_2(X_g^{w,2};\Q)) \bigoplus \cok (H_2(X_g^{w,2};\Q) \rightarrow H_2(X_g;\Q)) \\
&\twoheadrightarrow \cok(H_2(X_g^v;\Q) \rightarrow H_2(X_g;\Q)).
\end{align*}
\bn  Hence the lemma follows by Lemmas~\ref{twotofullcoklemma} and~\ref{vertstabcoklemma}.
\end{proof}

\subsection{The proof of Proposition~\ref{homolcurvequotprop}}\label{homolcurvequotsubsection}

We now conclude Section~\ref{finquotsectionpt2}.  We first connect the results of Section~\ref{abelcyclesection} with Lemma~\ref{vstablemma}.

\begin{lemma}\label{bpbmconnectionlemma}
Let $g \geq \finalbound$ and $a \subseteq S_g$ a nonseparating simple closed curve.  Let $w \in [a]^{\perp}$ be a nonzero primitive element such that $w$ is not sent to zero under the adjoint map $[a]^{\perp} \rightarrow \ho_{\Z}([a]^{\perp},\Z)$.  Let $\varphi$ be the composition
\begin{displaymath}
H_2^{\ab, \bp}(\cI_g;\Q) \rightarrow H_2(X_g;\Q) \rightarrow H_2(X_g;\Q))/H_2(X_g^w;\Q)
\end{displaymath}
\bn where the first map is the first map is the map in the five term exact sequence for the equivariant homology spectral sequence given by the action of $\cI_g$ on $\cC_{[a]}(S_g)$.  Then the map $\varphi$ is surjective.
\end{lemma}

\begin{proof}
By Lemma~\ref{vstablemma}, the quotient space $H_2(X_g;\Q)/H_2(X_g^w;\Q)$ is generated by the images of fundamental classes of Bestvina--Margalit tori.  Hence it suffices to show that any $[\BM_{\sigma}] \in H_2(X_g;\Q)$ is the image of some $[T_{c,c'}, T_{d,d'}] \in H_2^{\ab,\bp}(\cI_g;\Q)$ under the map $H_2^{\ab, \bp}(\cI_g;\Q) \rightarrow H_2(X_g;\Q)$.  Let $\BM_\sigma$ be a Bestvina--Margalit torus, and let $\widehat{\sigma}$ be a lift of $\sigma$ to $\cC_{[a]}(S_g)$ such that $a \in \widehat{\sigma}$.  Let $a_1, a_2$ be the other two vertices of $\widehat{\sigma}$.  Choose a curve $b \subseteq S_g$ such that the geometric intersections $\left|a \cap b\right| = \left|a_1 \cap b\right| = \left|a_2 \cap b\right|$ are all equal to 1.  Now, there are curves $b_1, b_2$ such that $b \cup b_1$ and $b \cup b_2$ are bounding pairs, and the corresponding bounding pair maps $T_{b, b_1}$ and $T_{b,b_2}$ both commute and take $a$ to $a_1$ and $a$ to $a_2$ respectively.  The construction of such bounding pairs can be seen in Figure~\ref{bpabelbmfigure}.
\begin{figure}[h]
\begin{tikzpicture}
\node[anchor = south west, inner sep = 0] at (0,0){\includegraphics{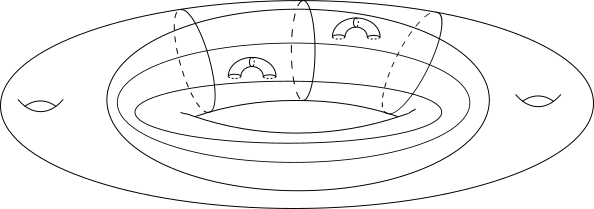}}[scale=0.9];
\node at (8.5,4){$b$};
\node at (5.4,3.8){$b_1$};
\node at (11.4,3.3){$b_2$};
\node at (6.4,5.0){$a_2$};
\node at (6.4,4.2){$a$};
\node at (6.4,3.1){$a_1$};
\end{tikzpicture}
\caption{The curves $a,a_1,a_2$ and bounding pairs $b \cup b_1$ and $b \cup b_2$}\label{bpabelbmfigure}
\end{figure}
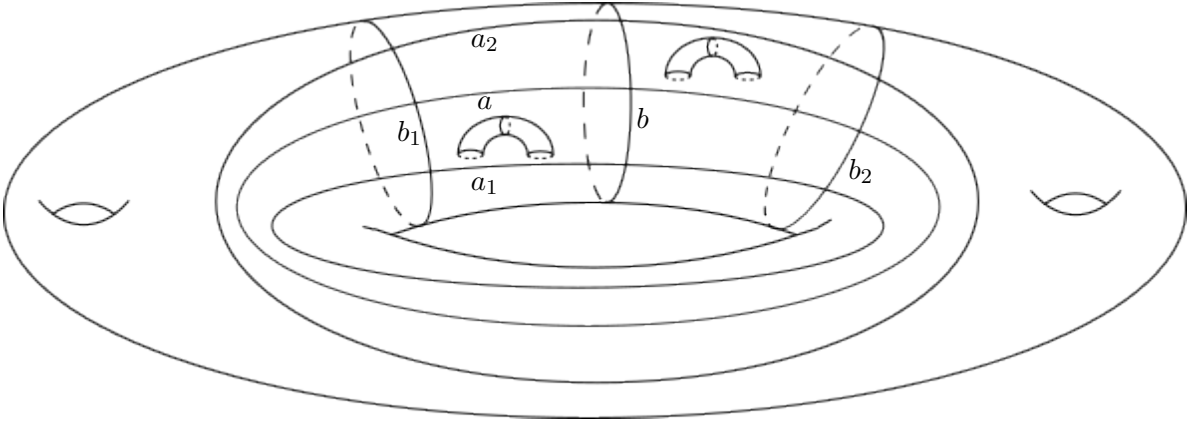
\bn Let $\gamma_i$ denote the loop in $\pi_1(X_g)$ given by the image of the edge in $\cC_{[a]}(S_g)$ connecting $a$ to $a_i$.  The bounding pair map $T_{b, b_i}$ is sent to the loop $\gamma_i$ under the natural map $\cI_g \rightarrow \pi_1(X_g)$ for $i = 1,2$.  Hence the image of the abelian cycle $[T_{b,b_1}, T_{b,b_2}]$ under the map $H_2(\cI_g;\Q) \rightarrow H_2(X_g;\Q)$ is the abelian cycle $[\gamma_1, \gamma_2] \in H_2(X_g;\Q)$, which is $[\BM_{\sigma}]$.
\end{proof}

\bn We now prove Lemma~\ref{bmfingenlemma}, which verifies that the first hypothesis of Proposition \ref{gengrpprop} holds for $\Sp([a]^{\perp},\Z)$, $H_2(X_g;\Q)$, and $d = 1$.  

\begin{lemma}\label{bmfingenlemma}
Let $g \geq \finalbound$ and let $w \in [a]^{\perp}$ be a primitive element such that the image of $w$ under the adjoint map $[a]^{\perp} \rightarrow \ho_{\Z}([a]^{\perp},\Z)$ is nonzero.  Then the vector space $\cok(H_2(X_g^w;\Q) \rightarrow H_2(X_g;\Q))$ is finite dimensional.
\end{lemma}

\begin{proof}[Proof of~\ref{bmfingenlemma}]
Let $d$ be a representative of $w$ disjoint from $a$, and let $G = \im(\Mod(S_g \cut (a \cup d)) \rightarrow \Sp(2g,\Z))$.  We will show that the $G$--representation $V = \cok(H_2(X_g^w;\Q) \rightarrow H_2(X_g;\Q))$ satisfies the hypotheses of Proposition \ref{gengrpprop} for $\delta = 9$, namely:
\begin{enumerate}
\item for $M \subseteq S_g \cut (a \cup d)$ a nonseparating multicurve with $|M| \geq 9$, the map
\begin{displaymath}
\bigoplus_{c \in M} V^{T_c} \rightarrow V
\end{displaymath}
\bn is surjective, and 
\item for $M \subseteq S_g \cut (a \cup d)$ a nonseparating multicurve with $|M| \leq 8$, and $G_M = \Stab_{G}(M)$, the coinvariants module
\begin{displaymath}
H_0(G_M;V)
\end{displaymath}
\bn is finite dimensional.
\end{enumerate}
\bn Then $V$ is finite dimensional by Proposition \ref{gengrpprop}.

\p{Hypothesis (1)} Lemma \ref{bpbmconnectionlemma} tells us that the map
\begin{displaymath}
H_2^{\ab,\bp}(\cI_g;\Q) \rightarrow V
\end{displaymath}
\bn is surjective.  Now Proposition \ref{abcycleprop} tells us that, if $M \subseteq S_g \cut (a \cup d)$ is a nonseparating multicurve with $|M| \geq 9$, the map
\begin{displaymath}
\bigoplus_{c \in M} H_2^{\ab,\bp}(\cI_g;\Q)^T_{c}  \rightarrow H_2^{\ab,\bp}(\cI_g;\Q)
\end{displaymath}
\bn is surjective.  Then the image $H_2^{\ab,\bp}(\cI_g;\Q)^{T_c} \rightarrow V$ is contained in $V^{T_c}$, so the map
\begin{displaymath}
\bigoplus_{c \in M} V^{T_c} \rightarrow V
\end{displaymath}
\bn is surjective, as desired.

\p{Hypothesis (2)} Let $M \subseteq S_g \cut (a \cut d)$ be a multicurve with $|M| \leq 8$.  The coinvariants module $\BM_2(X_g;\Q)_{G_M}$ is finite dimensional by Lemma \ref{bmfinquotlemma} applied to $M' = M \cup d$.  Now by Lemma \ref{vstablemma}, there is a surjection $\BM_2(X_g;\Q) \rightarrow V$, so there is a surjection $\BM_2(X_g;\Q)_{G_M} \rightarrow V_{G_M}$, so in particular $V_{G_M}$ is surjective.
\end{proof}

\begin{proof}[Proof of Proposition~\ref{homolcurvequotprop}]
Let $G = \im(\Mod(S_g \cut a) \rightarrow \Sp(2g,\Z))$.  We will use Proposition \ref{gengrpprop} applied to the $G$--representation $H_2(X_g;\Q)$ with $d = 1$to show that $H_2(X_g;\Q)$ is finite dimensional.  In particular, we will show that:
\begin{enumerate}
\item for any nonseparating $c \subseteq S_g \cut a$, we have $\cok(H_2(X_g;\Q)^{T_c} \rightarrow H_2(X_g;\Q))$ finite dimensional, and
\item the coinvariants module $H_2(X_g;\Q)_G$ is finite dimensional.
\end{enumerate}

\p{Hypothesis (1)} This is exactly the content of Lemma \ref{bmfingenlemma}.

\p{Hypothesis (2)} Let $\calI(\vec{x}) \subseteq \Mod(S_g)$ denote the partial Torelli group defined by Putman \cite{Putmanpartial}, which is the subgroup of $\Mod(S_g)$ acting trivially on the homology class $\vec{x}$.  Now, $\cI(\vec{x})$ fits into a short exact sequence
\begin{displaymath}
1 \rightarrow \cI_g \rightarrow \calI(\vec{x}) \rightarrow G \rightarrow 1.
\end{displaymath}

\bn For $g \geq 3$, we know that $H_1(\cI_g;\Q)$ is finite dimensional by the work of Johnson \cite{JohnsonIII}.  Therefore by the Lyndon--Hochschild--Serre \cite{Brownbook} spectral sequence and the fact that $G$ is finitely presented \cite{Ragarithmetic}, we see that $H_2(\cI_g;\Q)_{G}$ is finite dimensional if and only if $H_2(\cI(\vec{x});\Q)$ is finite dimensional.  We now consider the equivariant homology spectral sequence for the action of $\cI(\vec{x})$ on $\cC_{\vec{x}}(S_g)$.  This action is cocompact on $\cC_{\vec{x}}(S_g)$.  Furthermore, the stabilizers of vertices and edges are finitely presented and finitely generated respectively, since these are just mapping class groups of surfaces which are finitely presented \cite{FarbMarg}.  Then $\cC_{\vec{x}}(S_g)$ is $(g-3)$--acyclic by a theorem of the author \cite{Minahanhomolconn}, so $H_2(\cI(\vec{x});\Q)$ is finite dimensional, and thus the second hypothesis is satisfied.
\end{proof}

\section{Finiteness of coinvariants in $\bE_{1,1}^2(\cI_g, C_{[a]}(S_g))$}\label{11finquotsection}  

  The main work of this section is to prove the following result.

\begin{lemma}\label{11finquotlemma}
Let $g \geq \finalbound$ and $a \subseteq S_g$ be a nonseparating simple closed curve.  Let $M \subseteq S_g \cut a$ be a nonseparating multicurve with $\left|\pi_0(M)\right| = 8$.  Let $G = \im(\Mod(S_g \cut M) \rightarrow \Sp(2g,\Z))$.  Let $\bE_{p,q}^r$ denote the equivariant homology spectral sequence for the action of $\cI_g$ on $\cC_{[a]}(S_g)$.  Then the vector space
\begin{displaymath}
H_0(G; \bE_{1,1}^2)
\end{displaymath}
\bn is finite dimensional.
\end{lemma}

\p{Notation} For the remainder of this section, we will fix $g \geq \finalbound$ and $a \subseteq S_g$ a nonseparating simple closed curve.  We will also fix $M \subseteq S_g \cut a$ as in the statement of Lemma~\ref{11finquotlemma}.  We will use $\bE_{p,q}^r$ to denote the equivariant homology spectral sequence given by the action of $\cI_g$ on $\cC_{[a]}(S_g)$.  Additionally, for the remainder of the section, we will let $\calV = \{[c]: c\in \pi_0(M)\}$.  

\p{The outline of the proof of Lemma~\ref{11finquotlemma}} The strategy used is similar to that used in the proof of Lemma~\ref{bmfinquotlemma}.  In particular, we will associate to each edge $e \subseteq X_g$ certain invariants that record how the elements of $\calV$ project onto the elements of $H(e)$.  We will show in Lemma~\ref{11finquotalginvlemma} that these numbers are preserved under the action of $G$.  We then prove Lemma~\ref{intermediatequotientlemma}, which describes an intermediate quotient between $\bE_{1,1}^1$ and $\bE_{1,1}^2$ We then describe in Lemmas~\ref{11rkshrinklemma} and~\ref{11finquotshrinkgenlemma} how these invariants change under addition of edges.  With these results in hand, we will prove Lemma~\ref{11finquotlemma}.  

\p{Algebraic invariants of edges}  Let $[b] \in H_1(S_g;\Z)$ be a nonzero primitive class such that $\langle [a],[b] \rangle = 1$ and $[b]$ intersects any element of $\calV$ trivially.  Let $x \subseteq X_g$ be an edge and let $\calH(x) = \{\calH_0^x, \calH_1^x\}$.  For each $1 \leq i \leq 8$, $0 \leq k \leq 1$, let $v_{i,k}^x = \proj_{[a]^{\perp} \cap [b]^{\perp}} v_{i}$.  We define the following two invariants of the edge $x$:
\begin{enumerate}
\item $\rk^{\calV}_{i,k}(x)$ is the maximal $n$ such that $v^x_{i,k} = nv$ for some $v \in \calH_k^x$ and
\item $\theta(\calV)_{i,j,k}(x) = \langle v^x_{i,k}, v^x_{j,k} \rangle$.
\end{enumerate}

\bn We have the following result about these algebraic invariants, which parallels Lemma~\ref{bmconjlemma}.

\begin{lemma}\label{11finquotalginvlemma}
Let $g \geq \finalbound$ and $a \subseteq S_g$ be a nonseparating simple closed curve.  Let $M \subseteq S_g \cut a$ be a nonseparating multicurve with $\left|\pi_0(M)\right| = 8$.  Let $b \subseteq S_g$ be a nonseparating simple closed curve such that $b$ has geometric intersection number with $a$, and such that $b$ intersects $M$ trivially.  Let $G = \im(\Mod(S_g \cut M) \rightarrow \Sp(2g,\Z))$.  Let $\calV = \{[c]: c \in \pi_0(M)\}$.  After possibly relabeling $\calH_0^y$ and $\calH_1^y$, suppose the following hold:
\begin{enumerate}
\item $g(\calH_0^x) = g(\calH_0^y)$,
\item $\rk^{\calV}_{i,k}(x) = \rk^{\calV}_{i,k}(y)$ for all $1 \leq i \leq 8$ and $0 \leq k \leq 1$ and
\item $\theta(\calV)_{i,j,k}(x) = \theta(\calV)_{i,j,k}(y)$ for all $1 \leq i,j \leq 8$ and $0 \leq k \leq 1$.
\end{enumerate}
\bn Then there is a $g \in G$ such that $gx = y$.
\end{lemma}

\begin{proof}
This follows by a similar argument to Lemma~\ref{bmconjlemma}, except with 2--cells replaced by edges. The $\rk^{\calV}$ here is the same as the $\rk^{\calV}$ in Lemma~\ref{bmconjlemma}, and similarly for $\theta(\calV)$ and $\theta(\calV)$.
\end{proof}

Before proving Lemmas~\ref{11rkshrinklemma} and~\ref{11finquotshrinkgenlemma}, we will prove Lemma~\ref{intermediatequotientlemma}, which describes an intermediate quotient between $\bE_{1,1}^1 \rightarrow \bE_{1,1}^2$.  We begin by showing that $\bE_{1,1}^2$ is a quotient of $\bE_{1,1}^1$.

\begin{lemma}\label{specseqquotlemma}
Let $g \geq \finalbound$ and $a \subseteq S_g$ be a nonseparating simple closed curve.  Let $\bE_{p,q}^r$ denote the equivariant homology spectral sequence for the action of $\cI_g$ on $\cC_{[a]}(S_g)$.  The inclusion map $\bE_{1,1}^2 \hookrightarrow \bE_{1,1}^1/d_{2,1}^1(\bE_{2,1}^1)$ is an isomorphism.
\end{lemma}

\begin{proof}
The lemma is equivalent to the statement that the differential $d_{1,1}^1$ is the zero map.  Since $g \geq \finalbound \geq 4$, a theorem of the author~\cite[Theorem A]{Minahanhomolconn} says that we have $H_1(C_{[a]}(S_g);\Z) = 0$.  Hence 
\begin{displaymath}
\bE_{p,q}^r \Rightarrow H_1(\cI_g;\Q) \text{ for } p + q = 1
\end{displaymath}
\bn by the properties of the equivariant homology spectral sequence~\cite[Section VII]{Brownbook}.  Therefore there is an exact sequence
\begin{displaymath}
\bE_{1,1}^1 \xrightarrow{d_{1,1}^1}  \bE_{0,1}^1 \rightarrow H_1(\cI_g;\Q).
\end{displaymath}
\bn Since $\bE_{0,1}^1 \cong H_1(\Stab_{\cI_g}(a);\Q)$, the exact sequence is
\begin{displaymath}
\bE_{1,1}^1\xrightarrow{d_{1,1}^1} H_1(\Stab_{\cI_g}(a);\Q) \rightarrow H_1(\cI_g;\Q).
\end{displaymath}
\bn A theorem of Putman~\cite[Theorem B]{Putmanjohnson} says that the map $H_1(\Stab_{\cI_g}(a);\Q)  \rightarrow H_1(\cI_g;\Q)$ is an injection, so $d_{1,1}^1 = 0$. 
\end{proof}

\bn As in Section~\ref{abelcyclesection}, let $\tau_g:\cI_g \rightarrow \midwedge^3H_1(S_g;\Z)/H_1(S_g;\Z)$ denote the Johnson homomorphism~\cite{Johnsonhomomorphism}.  If $\sigma \subseteq X_g$ is a cell, let $A_\sigma$ denote the vector space
\begin{displaymath}
\im\left(\bigoplus_{\calH \in \calH(\sigma)}\midwedge^3 \calH \rightarrow \midwedge^3H_1(S_g;\Z)/H_1(S_g;\Z)\right) \otimes \Q.
\end{displaymath}
\bn We have the following result about $\bE_{1,1}^2$.

\begin{lemma}\label{intermediatequotientlemma}
Let $g \geq \finalbound$ and $a \subseteq S_g$ be a nonseparating simple closed curve.  Let $\bE_{p,q}^r$ denote the equivariant homology spectral sequence for the action of $\cI_g$ on $\cC_{[a]}(S_g)$.  The surjection $\varphi:\bE_{1,1}^1 \rightarrow \bE_{1,1}^2$ induced by quotienting out the subspace $d_{2,1}^1(\bE_{2,1}^1)$ factors through the natural projection
\begin{displaymath}
\rho:\bE_{1,1}^1 \rightarrow \bigoplus_{x \in X_g^{(1)}} A_x.
\end{displaymath}
\end{lemma}

\begin{proof}
Let $\varphi:\bE_{1,1}^1 \rightarrow \bE_{1,1}^1/d_{2,1}^1(\bE_{2,1}^1) = \bE_{1,1}^2$ denote the quotient map, which exists by Lemma~\ref{specseqquotlemma}.  For each $x \subseteq X_g$, let $\widehat{x} \subseteq C_{[a]}(S_g)$ be a lift of $x$ such that $a$ is a vertex of $\widehat{x}$.  Let 
\begin{displaymath}
A_{\widehat{x}} = \im\left(H_1(\Stab_{\cI_g}(\widehat{x});\Q) \rightarrow H_1(\cI_g;\Q)\right).
\end{displaymath}
\bn  Let $S$ and $S'$ be the connected components of $S_g \cut \widehat{x}$.  Lemma~\ref{firsthomolratimlemma} implies that $\im(H_1(\cI(S,S_g);\Q) \rightarrow H_1(\cI_g;\Q))$ is sent to $\im(\midwedge^3 H_1(S;\Q) \rightarrow \im(\tau_g) \otimes \Q)$ under the Johnson homomorphism, and similarly for $S'$.  Therefore we have $A_{\widehat{x}} = A_x$ by the definition of $A_x$.  Hence it suffices to show that $\varphi$ factors through the quotient map
\begin{displaymath}
\rho: \bE_{1,1}^1 =  \bigoplus_{x \in X_g^{(1)}} H_1(\Stab_{\cI_g}(\widehat{x});\Q) \rightarrow \bigoplus_{x \in X_g^{(1)}} A_{x}.
\end{displaymath}
\bn Showing that $\varphi$ factors through $\rho$ is equivalent to showing that $\ker(\rho) \subseteq \ker(\varphi)$, so it suffices to show that
\begin{displaymath}
\bigoplus_{x \in X_g^{(1)}} \ker\left(H_1(\Stab_{\cI_g}(\widehat{x});\Q) \rightarrow H_1(\cI_g;\Q)\right) \subseteq \ker(\varphi).
\end{displaymath}

 The idea of the proof is to rewrite any class in $\ker(\rho)$ as a linear combination of classes, where the edges in these classes have sufficiently large genus.  If $x \subseteq X_g$ is an edge and $f \in H_1(\Stab_{\cI_g}(\widehat{x});\Q)$ is a class, we use the notation $\left(x,f\right) \in \bigoplus_{x \in X_g^{(1)}} H_1(\Stab_{\cI_g}(\widehat{x});\Q)$ to denote the class in $\bE_{1,1}^2$ equal to $f$ in the index $x$ and equal to zero in every other index.  Now, we have
\begin{displaymath}
\ker(\rho) = \bigoplus_{x \in X_g^{(1)}} \ker\left(H_1(\Stab_{\cI_g}(\widehat{x});\Q) \rightarrow H_1(\cI_g;\Q)\right),
\end{displaymath}
\bn and the latter space is spanned by elements of the form $(x,f)$.  Hence it suffices to show that $(x,f) \in \ker(\varphi)$ for any $x \in X_g^{(1)}$ and $f \in \ker(H_1(\Stab_{\cI_g}(\widehat{x});\Q) \rightarrow H_1(\cI_g;\Q))$.  Let $S', S''$ be the connected components of $S_g \cut \widehat{x}$. We have a surjection $H_1(\cI(S', S_g);\Q) \oplus H_1(\cI(S'', S_g);\Q) \rightarrow H_1(\Stab_{\cI_g}(\widehat{x});\Q $ by applying the K{\"u}nneth formula to the product $\cI(S',S_g) \times \cI(S'',S_g)$, so we may assume that $f$ is supported on one connected component $S_g \cut \widehat{x}$.  Without loss of generality, we will assume that $f \in H_1(\cI(S', S_g);\Q)$.  If $g(S') \geq 3$, then Lemma~\ref{firsthomolratimlemma} says that the map $H_1(\cI(S', S_g);\Q) \rightarrow H_1(\cI_g;\Q)$ is injective.  Since $f$ is in the kernel of the map $H_1(\cI(S', S_g);\Q) \rightarrow H_1(\cI_g;\Q)$, this implies that $f = 0$, so $(x,f) = 0 \in \bE_{1,1}^1$ and thus $(x,f) \in \ker(\varphi)$.  Otherwise, $g(S') \leq 2$, which implies in particular that $g(S'') \geq 4$.  Let $\sigma \subseteq C_{[a]}(S_g)$ be a 2--cell such that: 
\begin{itemize}
\item $\widehat{x} \subseteq \sigma$
\item for the other two edges $y,z$ of $\sigma$, the connected components $S_y', S_z'$ of $S_g \cut y$ and $S_g \cut z$ respectively that contain $S'$ have $g(S_y'), g(S_z') \geq 3$.  
\end{itemize}
\bn Now, let $\overline{\sigma}$ denote the image of $\sigma$ in $X_g$ and similarly for $\overline{y}$ and $\overline{z}$.  We have chosen $\sigma$ so that $f$ has a representative $F \in \Stab_{\cI_g}(\sigma)$. Therefore, after possibly reorienting $\sigma$, $x$ and $y$, we have a relation in $\bE_{1,1}^2$ given by $d_{2,1}^1(\overline{\sigma}, [F])$, which tells us in particular that
\begin{displaymath}
0 = (\overline{y},[F]) + (\overline{z},[F]) - (x,f)
\end{displaymath}
\bn in $\bE_{1,1}^2$.  We have $F$ supported on $S_y'$ and $S_z'$, both of which have genus at least 3.  Then by Lemma~\ref{firsthomolratimlemma}, the map $H_1(\cI(S_y', S_g);\Q) \rightarrow H_1(\cI_g;\Q)$ is injective, and similarly for $S_z'$.  But now, we have assumed that $f \in \ker(H_1(\Stab_{\cI_g}(\widehat{x});\Q) \rightarrow H_1(\cI_g;\Q))$.  Therefore $[F] \in H_1(\cI(S_y', S_g);\Q)$ and $[F]\in H_1(\cI(S_z', S_g);\Q)$ are both zero, so $(\overline{y},[F]) = (\overline{z},[F]) = 0$.  Since $d_{2,1}^1(\overline{\sigma}, [F]) \in \ker(\varphi)$ by the definition of $\bE_{1,1}^2$, we have $(x,f) \in \ker(\varphi)$, as desired.
\end{proof}

\p{Notation} For the remainder of the paper, if $x \subseteq X_g$ is an edge and $f \in A_x$ is a class, we will denote the corresponding element in $\bE_{1,1}^2$ under the image of the quotient map $\bigoplus_{x \in X_g^{(1)}}A_x \rightarrow \bE_{1,1}^2$ by $(x,f)$. 

\medskip 

\subsection{The proof of Lemma~\ref{11rkonelemma}}

We now continue with the definitions.  Let $x \subseteq X_g$ be an edge.  Let $y \subseteq X_g$ be another edge such that $x,y$ are two edges of a 2--cell $\sigma$, and let $z$ denote the third edge of $\sigma$.  We say that such a $y$ is \textit{rank $\calV$--shrinking relative to $x$} if, after possibly relabeling $\calH(x)$, $\calH(y)$ and $\calH(z)$, the following hold:
\begin{enumerate}
    \item $\rk^{\calV}_{i,k}(x) \geq \rk^{\calV}_{i,k}(y)$ and $\rk^{\calV}_{i,k}(z)$ for all $1 \leq i \leq 8$, $0 \leq k \leq 1$, and
    \item for at least one choice of $1 \leq i \leq 8$ and $ 0 \leq k \leq 1$ with 
    \begin{displaymath}
\max\{\rk^{\calV}_{i,0}(x), \rk^{\calV}_{i,1}(x) \geq\max\{\rk^{\calV}_{j,k}(x),: 1 \leq j \leq 8, 0 \leq k \leq 1\}
\end{displaymath}
\bn we have $\max\{\rk^{\calV}_{i,0}(x), \rk^{\calV}_{i,1}(x)\} > \rk^{\calV}_{i,k}(y), \rk^{\calV}_{i,k}(z).$
\end{enumerate}
\bn If $\calH_0^x \in \calH(x)$, we say that $y$ is \textit{rank $\calV$--shrinking relative to $\calH_0^x$} if, in addition, there is $\calH_i^y \in \calH(y)$ with $\calH_i^y \subseteq \calH_0^x$.  Denote the set of rank $\calV$--shrinking edges relative to $\calH_0^x$ by $\rkshrink_{\calV}(\calH_0^x)$.  We now prove Lemma~\ref{11rkshrinklemma}, which will allow us to rewrite classes in $\bE_{1,1}^2$ as linear combinations of classes with lower $\rk^{\calV}_{i,k}$.  
\begin{lemma}\label{11rkshrinklemma}
Let $g \geq \finalbound$ and $a \subseteq S_g$ be a nonseparating simple closed curve.  Let $M \subseteq S_g \cut a$ be a nonseparating multicurve with $\left|\pi_0(M)\right| = 8$.  Let $b \subseteq S_g$ be a nonseparating simple closed curve such that $b$ has geometric intersection number with $a$, and such that $b$ intersects $M$ trivially.  Let $\calV = \{[c]: c \in \pi_0(M)\}$.  Let $x \subseteq X_g$ be an edge.  Suppose that $g(\calH_0^x) \geq 12$, and that $\rk^{\calV}_{i,1}(x) > 1$ for at least one $1 \leq i \leq 8$. 
 Then the natural map
\begin{displaymath}
    \varphi: \bigoplus_{y \in \rkshrink_{\calV}(\calH_0^x)} A_y \cap A_x \rightarrow A_x
\end{displaymath}
\bn is surjective.
\end{lemma}

\begin{proof}
Let $r_1 \wedge r_2 \wedge r_3 \in A_x$.  We will show that $r_1 \wedge r_2 \wedge r_3 \in A_y$ for some $y \in \rkshrink_{\calV}(\calH_0^x).$  If $r_1, r_2, r_3 \in \calH_1^x$ then $r_1 \wedge r_2 \wedge r_3 \in \midwedge^3 \calH_1^x \subseteq \midwedge^3 \calH_1^y \subseteq A_y$ for any $y \in \rkshrink_{\calV}(\calH_0^x)$,  so it remains to prove the result in the case that $r_1 \wedge r_2 \wedge r_3 \in \calH_0^x.$  Suppose without loss of generality that $\rk^{\calV}_{1,1}(x) > 1$ and
\begin{displaymath}
\rk^{\calV}_{1,1}(x) =\max\{\max\{\rk^{\calV}_{i,0}(x), \rk^{\calV}_{i,1}(x)\} \geq\max\{\max\{\rk^{\calV}_{j,0}(x), \rk^{\calV}_{j,1}(x)\}: 1 \leq j \leq 8\}.
\end{displaymath}
\bn  For each $1 \leq i \leq 8$, let $w_{i,k}$ denote the nonzero primitive homology class in $H_1(S_g;\Z)$ such that
\begin{displaymath}
    \proj_{\calH_k^x \cap [b]^{\perp}}(v_i) = \rk^{\calV}_{i,k}(x) w_{i,k}.
\end{displaymath}
\bn We claim that there is a nonzero primitive class $u \in \calH_0^x$ such that:
\begin{itemize}
    \item $\langle u, w_{i,0} \rangle = 0 $ for all $1 \leq i \leq 8$,
    \item $\langle u, r_s \rangle = 0$ for all $1 \leq s \leq 3$, and
    \item $u$ is not in the span of $\{w_{1,0}, \ldots, w_{8,1}, r_1,\ldots, r_3\}$.
\end{itemize}
\bn Such a $u$ exists since we have assumed that $g(\calH_0^x) \geq 12$ and there are only eleven elements in the set $\{w_{1,0},\ldots, w_{8,0}, r_1,r_2,r_3\}$, so the subspace $w_{1,0}^{\perp} \cap \ldots w_{8,0}^{\perp} \cap r_1^{\perp} \cap r_2^{\perp} \cap r_3^{\perp}$ has genus at least one.  Let $h = w_{1,0} - u.$  Let $y \subseteq X_g$ be an edge such that such that:
\begin{multicols}{2}
\begin{itemize}
    \item $\calH_0^y \subseteq \calH_0^x$
    \item $u \in \calH_0^y$,
    \item $h \in \calH_1^y \cap \calH_0^x$,
    \item $w_{i,0} \in \calH_0^x \cap \calH_1^y$ for all $2 \leq i \leq 8$, and
    \item $r_s \in \calH_1^y \cap \calH_0^x$ for all $1 \leq s \leq 3.$
\end{itemize}
\end{multicols}
\bn Let $\sigma$ be a 2--cell containing $x$ and $y$, and let $z$ be the third edge of $\sigma$.  Index $\calH(z)$ so that $\calH_0^z \subseteq \calH_0^x$.   By construction, the vectors $\proj_{[b]^{\perp} \cap \calH_k^y}(v_i)$ and $\proj_{[b]^{\perp} \cap \calH_k^z}(v_i)$ are given as follows:
\begin{multicols}{2}
\begin{enumerate}
    \item  $\proj_{[b]^{\perp} \cap \calH_0^y}(v_1) = \rk^{\calV}_{1,0}(x)u$,
    \item $\proj_{[b]^{\perp} \cap \calH_0^y}(v_i) = 0$ for $2 \leq i \leq 8$,
    \item $\proj_{[b]^{\perp} \cap \calH_1^y}(v_1) = \rk^{\calV}_{1,1}(x)w_{1,1} + \rk^{\calV}_{1,0}h$,
    \item $\proj_{[b]^{\perp} \cap \calH_1^y}(v_i) = \rk^{\calV}_{i,1}(x)w_{i,1} + \rk^{\calV}_{i,0}(x)w_{i,0}$ for $2 \leq i \leq 8$,
    \item  $\proj_{[b]^{\perp} \cap \calH_0^z}(v_1) = \rk^{\calV}_{1,0}(x)h,$
    \item $\proj_{[b]^{\perp} \cap \calH_0^z}(v_i) = \rk^{\calV}_{i,0}(x)w_{i,0}$ for $2 \leq i \leq 8$,
    \item $\proj_{[b]^{\perp} \cap \calH_1^z}(v_1) = \rk^{\calV}_{1,1}(x)w_{1,1} + \rk^{\calV}_{1,0}u$, and
    \item $\proj_{[b]^{\perp} \cap \calH_1^z}(v_i) = \rk^{\calV}_{i,1}w_{i,1}$ for $2 \leq i \leq 8$.
\end{enumerate}
\end{multicols}
\bn By assumption, the homology classes $w_{i,k}$, $h$ and $u$ are all primitive.  Hence the numbers $\rk^{\calV}_{i,k}(y)$ and $\rk^{\calV}_{i,k}(z)$ are given as follows, where each relation in the following list follows from the corresponding relation in the previous list:
\begin{multicols}{2}
\begin{enumerate}[label=(\alph{enumi})]
    \item $\rk^{\calV}_{1,0}(y) = \rk^{\calV}_{1,0}(x)$,
    \item $\rk^{\calV}_{i,0}(y) = 0$ for $2 \leq i \leq 8$,
    \item $\rk^{\calV}_{1,1}(y) = \gcd(\rk^{\calV}_{1,1}(x), \rk^{\calV}_{1,0}(x))$,
    \item $\rk^{\calV}_{i,1}(y) = \gcd(\rk^{\calV}_{i,1}(x), \rk^{\calV}_{i,0}(x))$ for $2 \leq i \leq 8$,
    \item $\rk^{\calV}_{1,0}(z) = \rk^{\calV}_{1,0}(x)$,
    \item $\rk^{\calV}_{i,0}(z) = \rk^{\calV}_{i,0}(x)$ for $2 \leq i \leq 8$,
    \item $\rk^{\calV}_{1,1}(z) = \gcd(\rk^{\calV}_{1,1}(x), \rk^{\calV}_{1,0}(x))$,
    \item $\rk^{\calV}_{i,1}(z) = \rk^{\calV}_{i,1}(x)$ for $2 \leq i \leq 8.$
\end{enumerate}
\end{multicols}
\bn We have assumed that $\rk^{\calV}_{1,1}(x) > 1$.  Since $v_1$ is primitive, we have $\gcd(\rk^{\calV}_{1,0}(x), \rk^{\calV}_{1,1}(x)) = 1$, so relation (c) implies that $\rk^{\calV}_{1,1}(y) < \rk^{\calV}_{1,1}(x)$.  Similarly, relation (g) implies that $\rk^{\calV}_{1,1}(z) < \rk^{\calV}_{1,1}(x).$  Then relations (a), (b) and (d) imply that $\rk^{\calV}_{i,k}(y) \leq \rk^{\calV}_{i,k}(x)$ for all $1 \leq i \leq 8$ and $0 \leq k \leq 1$, and relations (e), (f) and (h) imply that $\rk^{\calV}_{i,k}(z) \leq \rk^{\calV}_{i,k}(y)$ for all $1 \leq i \leq 8$ and $0 \leq k \leq 1$, and thus $y \in \rkshrink_{\calV}(\calH_0^x)$.  Then we have chosen $y$ so that $r_1, r_2, r_3 \in \calH_0^y$, so $r_1 \wedge r_2 \wedge r_3 \in A_y$, and thus the proof is complete.
\end{proof}

We use Lemma~\ref{11rkshrinklemma} to prove the following, which is the first step of the proof of Lemma~\ref{11finquotlemma}.

\begin{lemma}\label{11rkonelemma}
Let $g \geq \finalbound$ and $a \subseteq S_g$ be a nonseparating simple closed curve.  Let $M \subseteq S_g \cut a$ be a nonseparating multicurve with $\left|\pi_0(M)\right| = 8$.  Let $b \subseteq S_g$ be a nonseparating simple closed curve such that $b$ has geometric intersection number with $a$, and such that $b$ intersects $M$ trivially.  Let $\calV = \{[c]: c \in \pi_0(M)\}$.  Let $\bE_{p,q}^r$ denote the equivariant homology spectral sequence for the action of $\cI_g$ on $\cC_{[a]}(S_g)$.  Let $(x,f) \in \bE_{1,1}^2$ be a class.  There is a relation in $\bE_{1,1}^2$ given by
\begin{displaymath}
(x,f) = \sum_{\ell = 1}^m \lambda_{\ell}(y_{\ell}, f_{\ell})
\end{displaymath}
\bn such that $\lambda_{\ell} \in \Q$, $\rk^{\calV}_{i,k}(y_{\ell}) \leq 1$ for all $1 \leq i \leq 8, 0 \leq k \leq 1$, and $1 \leq \ell \leq m$. 
\end{lemma}
\begin{proof}
Let $\maxrk(x) = \max_{1 \leq i \leq 8, 0 \leq k \leq 1} \rk^{\calV}_{i,k}(x)$.  Let $\nummaxrk(x)$ be the number of pairs $(i,k)$ with $1 \leq i \leq 8$ and $0 \leq k \leq 1$ such that $\rk^{\calV}_{i,k}(x) = \maxrk(x)$.  The proof proceeds by double induction on $\maxrk(x)$ and $\nummaxrk(x)$.

\p{Base case: $\maxrk(x) = 1$} In this, the resulting linear relation is $(x,f) = (x,f)$.  

\p{Inductive step: the lemma holds for all $y \subseteq X_g$ that satisfy either $\maxrk(y) < \maxrk(x)$ or both $\maxrk(y) \leq \maxrk(x)$ and $\nummaxrk(y) < \nummaxrk(x)$}  We want to apply Lemma~\ref{11rkshrinklemma}.  We begin with the the following claim.  

\p{Claim} There is a linear relation
\begin{displaymath}
(x,f) = \sum_{p = 1}^q (z_p, f_p)
\end{displaymath}
\bn such that for every $1 \leq p \leq q$:
\begin{itemize}
\item either $\maxrk(z_p) = \maxrk(x)$ and $\nummaxrk(z_p) \leq \nummaxrk(x)$, or $\maxrk(z_p) < \maxrk(x)$, and
\item if $\maxrk(z_p) \neq 1$, there is a pair $(i,k)$ such that $\rk^{\calV}_{i,k}(z_p) = \maxrk(z_p)$ and $g(\calH_k^{z_{p}}) \geq 13$.
\end{itemize}

\p{Proof of claim} By reindexing, we may assume without loss of generality that $\rk^{\calV}_{1,0}(x) = \maxrk(x)$.  If $g(\calH_0^x) \geq 13$ then we are done, so suppose that $g(\calH_0^x) \leq 12$.  Since $f \in A_x$, we may rewrite $f$ as a $\Q$--linear combination of pure tensors $r_1 \wedge r_2 \wedge r_3 \in A_x$ such that each $r_i$ is primitive.  Therefore we may assume without loss of generality that $f = r_1 \wedge r_2 \wedge r_3$ with $r_1,r_2, r_3 \in \calH_k^x$ for some $0 \leq k \leq 1$.  Now, since $g(\calH_0^x) \leq 12$, we have $g(\calH_1^x) \geq \finalbound - 12 \geq 13$.  Hence there is a primitive subgroup $\calH \subseteq \calH_1^x$ such that the following hold:
\begin{itemize}
\item $2 *g(\calH) +1 = \rk^{\calV}(\calH)$,
\item $g(\calH_1^x) - g(\calH) \leq 11$,
\item $[a] \in \calH$,
\item $r_1,r_2,r_3, v_{1,1}, v_{2,1},\ldots, v_{8,1} \in \calH^{\perp}$.
\end{itemize}
\bn Let $z_1 \subseteq X_g$ be the unique edge with $\calH \in \calH(z_1)$.  By construction, $z_1$ and $x$ share a 2--cell $\sigma \subseteq X_g$ that satisfies $\calH(\sigma) = \{\calH_0^x, \calH, \calH_1^x \cap \calH^{\perp}\}$.  Let $z_2$ be the third edge of $\sigma$.  We have $\calH(z_1) = \{\calH, \calH^{\perp}\}$ and $\calH(z_2) = \{\calH^{\perp} \cap \calH_1^x, \calH_0^x + \calH\}$.  Since $r_1,r_2,r_3$ are all in either $\calH_0^x$ or $\calH_1^x$  and are in $\calH^{\perp}$ by hypothesis, the fact that $\calH(\sigma) = \{\calH_0^x, \calH, \calH^{\perp} \cap \calH_1^x\}$ implies that $f \in A_{\sigma}$.  Therefore in $\bE_{1,1}^2$, the image of $d_{2,1}^1(\sigma, f)$ (after possibly reorienting $x,z_1$ and $z_2$) yields a relation
\begin{displaymath}
(x,f) = (z_1,f) + (z_2,f).
\end{displaymath}
\bn Assume that $\calH(z_1)$ is indexed so that $\calH = \calH_0^{z_1}$.  By our choice of $\calH$, the projections of each $v_i$ to the elements of $\calH(z_1)$ are given as follows:
\begin{itemize}
\item $v_{i,1}^{z_1} = v_i$, and
\item $v_{i,0}^{z_1} = 0$.
\end{itemize}
\bn Since each $v_i$ is primitive by assumption, we have $\rk^{\calV}_{i,k}(z_1) = 1$ or $0$ for every $1 \leq i \leq 8$ and $0 \leq k \leq 1$. Hence $(z_1,f)$ satisfies the desired properties of $z_i$ in the claim, since:
\begin{itemize}
\item $\maxrk(z_1) = 1 < \maxrk(x)$, and
\item $z_1$ does not satisfy the hypothesis in the second condition.  
\end{itemize}
\bn For the class $(z_2,f)$, after indexing $\calH(z_2)$ so that $\calH_0^x \subseteq \calH_0^{z_2}$, the projections $v_{i,k}^{z_2}$ satisfy $v_{i,k}^{z_2} = v_{i,k}^x$.  Hence we have $\rk^{\calV}_{i,k}(z_2) = \rk^{\calV}_{i,k}(x)$ for $1 \leq i \leq 8$ and $0 \leq k \leq 1$.  Now, by construction we have $g(\calH_0^{z_2}) = g(\calH_0^x) + g(\calH)$.  Since we have chosen $\calH$ so that $g(\calH_1^x) - g(\calH) \leq 11$, we have $g(\calH) \geq g(\calH_1^x) - 11$.  Then we have
\begin{displaymath}
g(\calH_0^{z_2}) = g(\calH_0^x) + g(\calH_1^x) \geq g(\calH_0^x) + g(\calH_1^x) - 11 \geq \finalbound - 1 - 11 \geq 13,
\end{displaymath}
\bn since we have assumed that $g = g(\calH_0^x) + g(\calH_1^x) + 1 \geq \finalbound$.  Therefore $z_2$ satisfies the properties in the claim, since:
\begin{itemize}
\item $\maxrk(z_2) = \maxrk(x)$ and $\nummaxrk(z_2) = \nummaxrk(x)$, and
\item any pair $(i,0)$ that satisfies $\rk^{\calV}_{i,0}(x) = \maxrk(x)$ also satisfies $\rk^{\calV}_{i,0}(z_2) = \maxrk(z_2)$, and $g(\calH_0^{z_2}) \geq 13$,
\end{itemize}
\bn so the claim holds.

We now continue with the inductive step of the proof.  By the claim, we may rewrite the class $(x,f)$ as a sum so that each summand $(z,f')$ satisfies $\maxrk(z) \leq \maxrk(x)$, and $\nummaxrk(z) \leq \nummaxrk(x)$ if $\maxrk(z) = \maxrk(x)$, and, if $(z,f)$ does not already satisfy $\maxrk(z) \leq 1$, we have $g(\calH_0^z) \geq 13$ with $\rk^{\calV}_{i,0}(z) = \maxrk(z)$ for some $i$.  Hence we may assume without loss of generality that $g(\calH_0^x) \geq 13$ and $\rk^{\calV}_{i,0}(x) = \maxrk(x)$ for some $1 \leq i \leq 8$.  Then by Lemma~\ref{11rkshrinklemma}, there is a collection of edges $y_{\ell} \in \rkshrink_{\calV}(\calH_0^x)$ and a choice of $f_{\ell} \in A_{y_{\ell}} \cap A_x$ such that
\begin{displaymath}
f = \sum_{\ell = 1}^m f_{\ell}.
\end{displaymath}
\bn For each $y_{\ell}$, there is a 2--cell $\sigma_{\ell} \subseteq X_g$ with $x_{\ell}, y_{\ell} \in X_g$ by the definition of $\rkshrink_{\calV}(\calH_0^x)$.  For each of these $\sigma_{\ell}$, let $z_{\ell} \subseteq \sigma_{\ell}$ denote the third edge besides $x$ and $y_{\ell}$.  Then $f_{\ell} \in A_{y_{\ell}} \cap A_x$ and $A_x \cap A_{y_{\ell}} = A_{\sigma_{\ell}}$ by the definition of $A_{\tau}$ for any $\tau \subseteq X_g$, so we have $f_{\ell} \in A_{\tau}$.  Therefore we have the following relation in $\bE_{1,1}^2$:
\begin{displaymath}
d_{2,1}^1\left (\sum_{\ell = 1}^m (\sigma_\ell, f_\ell)\right) = \sum_{\ell = 1}^m (x,f_{\ell}) + (y_{\ell}, f_{\ell}) - (z_{\ell}, f_{\ell}).
\end{displaymath}
\bn But now, by rearranging terms, we have a relation
\begin{displaymath}
\sum_{\ell = 1}^m (x, f_{\ell}) = \sum_{\ell = 1}^m (y_{\ell}, f_{\ell}) - (z_{\ell}, f_{\ell}).
\end{displaymath}
\bn Then since $\sum_{\ell = 1}^m f_{\ell} = f$, we have a relation
\begin{displaymath}
(x,f) = \sum_{\ell = 1}^m (z_{\ell}, f_{\ell}) - (y_{\ell}, f_{\ell}).
\end{displaymath}
\bn Since $y_{\ell} \in \rkshrink_{\calV}(\calH_0^x)$, for any $1 \leq \ell \leq m$, we have either $\maxrk(y_{\ell}) < \maxrk(x)$, or $\maxrk(y_{\ell}) = \maxrk(x)$ and $\nummaxrk(y_{\ell}) < \nummaxrk(x)$, and similarly for $z_{\ell}$.  Therefore by the inductive hypothesis, $y_{\ell}$ and $z_{\ell}$ are linear combinations of elements as in the statement of the lemma, so the lemma holds for $(x,f)$ as well.
\end{proof}

\subsection{The proof of Lemma~\ref{11algonelemma}} We now proceed to the second step of the proof of Lemma~\ref{11finquotlemma}.  Let $x \subseteq X_g$ be an edge such that $\rk^{\calV}_{i,k}(x) \leq 1$ for all $1 \leq i \leq 8$ and $0 \leq k \leq 1.$  Let $y \subseteq X_g$ be another edge such that $y$ and $x$ are two edges of a 2--cell $\sigma$, and let $z$ be the third edge of $\sigma.$  We say that $y$ is \textit{algebraically $\calV$--shrinking relative to $x$} if, after possibly reindexing $\calH(x)$, $\calH(y)$, and $\calH(z)$, the following hold:
\begin{enumerate}
\item $\rk^{\calV}_{i,k}(y), \rk^{\calV}_{i,k}(z) \leq 1$ for $1 \leq i \leq 8$ and $0 \leq k \leq 1$,
\item $\max\{\left|\theta(\calV)_{i,j,1}(x) \right|,1\} \geq \left| \theta(\calV)_{i,j,1}(y)\right|, \left| \theta(\calV)_{i,j,1}(z)\right| $ for all $1 \leq i,j \leq 8$, and
\item at least one pair $1 \leq i< j \leq 8$ and $0 \leq k \leq 1$ with
\begin{displaymath}
\left|\theta(\calV)_{i,j,k}(x)\right| = \max\{\left|\theta(\calV)_{i',j',k}(x) \right|: 1 \leq i',j' \leq 8, 0 \leq k \leq 1\}
\end{displaymath}
\bn has $\left|\theta(\calV)_{i,j,k}(x) \right|>  \left| \theta(\calV)_{i,j,k}(y)\right|, \left| \theta(\calV)_{i,j,k}(z)\right|$.
\end{enumerate}
\bn If $x$ is an edge, let $\algshrink_{\calV}(x) \subseteq X_g^{(1)}$ denote the set of algebraically $\calV$--shrinking relative to $x$ edges in $X_g$.  We will prove the following result about $\algshrink_{\calV}(x)$, which completes the second step of the proof of Lemma~\ref{11finquotlemma}.

\begin{lemma}\label{11finquotshrinkgenlemma}
Let $g \geq \finalbound$ and $a \subseteq S_g$ be a nonseparating simple closed curve.  Let $M \subseteq S_g \cut a$ be a nonseparating multicurve with $\left|\pi_0(M)\right| = 8$.  Let $b \subseteq S_g$ be a nonseparating simple closed curve such that $b$ has geometric intersection number with $a$, and such that $b$ intersects $M$ trivially.  Let $\calV = \{[c]: c \in \pi_0(M)\}$.  Let $\bE_{p,q}^r$ denote the equivariant homology spectral sequence for the action of $\cI_g$ on $\cC_{[a]}(S_g)$.  Let $x \subseteq X_g$ be an edge such that $\rk^{\calV}_{i,k}(x) \leq 1$ for all $1 \leq i \leq 8$ and $0 \leq k \leq 1$.  Suppose that there is a pair $1 \leq i < j \leq 8$ and $0 \leq k \leq 1$ such that $\left|\theta(\calV)_{i,j,k}(x)\right| > 1$.  Then the natural map
\begin{displaymath}
\varphi: \bigoplus_{y \in \algshrink_{\calV}(x)} A_y \cap A_x \rightarrow A_x
\end{displaymath}
\bn is surjective.
\end{lemma}

\begin{proof}
Choose three primitive classes $r_1,r_2,r_3 \in \calH_0^x \cup \calH_1^x$ such that $r_1 \wedge r_2 \wedge r_3 \in A_x$.  We will show that $r_1 \wedge r_2 \wedge r_3 \in \im(\varphi)$.  Since $\left|\theta(\calV)_{i,j,0}(x)\right| = \left|\theta(\calV)_{i,j,1}(x)\right|$ for all $1 \leq i,j \leq 8$ because we have assumed that $\langle v_i, v_j \rangle  =0 $, we may assume without loss of generality that $g(\calH_0^x) \geq \frac{\finalbound}2$, so in particular $g(\calH_0^x) \geq 13$.  If $r_1, r_2, r_3 \in \calH_1^x$ then we are done since $\algshrink_{\calV}(x)$ is nonempty, so assume that $r_1,r_2,r_3 \in \calH_0^x.$  We will show that there is an edge $y \in \algshrink_{\calV}(x)$ such that $r_1 \wedge r_2 \wedge r_3 \in A_y$.  For each $v_i \in \calV$, let $w_{i,k}$ be the projection $w_{i,k} = \proj_{[b]^{\perp} \cap \calH_k^x} v_i$.  By hypothesis each $w_{i,k}$ is either primitive or zero.  Assume without loss of generality that $\left|\theta(\calV)_{1,2,0}\right|$ is maximal over all $\left|\theta(\calV)_{i,j,0}\right|$.  By possibly replacing $v_1$ with $-v_1$, we may also assume without loss of generality that $\theta(\calV)_{1,2,0}(x) \geq 0$.  Since $g(\calH_0^x) \geq 13$ and the set $\{w_{1,0}, \ldots, w_{8,0}, r_1, r_2, r_3\}$ has eleven elements, we may choose two primitive classes $u_1,u_2 \in \calH_0^x$ such that the following hold:
\begin{multicols}{2}
\begin{enumerate}
\item $\langle u_t, w_{i,0} \rangle = 0$ for all $1 \leq t \leq 2, 1 \leq i \leq 8$,
\item $\langle u_t, r_s \rangle = 0$ for all $1 \leq t \leq 2, 1 \leq s \leq 3$,
\item $\langle u_1,u_2\rangle = 1$,
\end{enumerate}
\end{multicols}
\bn Consider the following classes $h_i$ for $1 \leq i \leq 8$:
\begin{multicols}{2}
\begin{enumerate}
\item $h_1 = w_{1,0} - u_1$,
\item $h_2 = w_{2,0} - u_2$ and
\item $h_i = w_{i,0}$ for $3 \leq i \leq 8$.
\end{enumerate}
\end{multicols}
\bn Now, let $y$ be an edge that shares a 2--cell with $x$ such that the following properties hold:
\begin{multicols}{2}
\begin{enumerate}
\item $\calH_0^y \subseteq \calH_0^x$,
\item $h_i \in \calH_0^x \cap \calH_0^y$ for $1 \leq i \leq 8$,
\item $r_s \in \calH_0^x \cap \calH_1^y$ for $1 \leq s \leq 3$, and
\item $u_t \in \calH_0^x \cap \calH_1^y$ for $1 \leq t \leq 2$.
\end{enumerate}
\end{multicols}
\bn We will show that:
\begin{enumerate}
\item $y \in \algshrink_{\calV}(x)$ and
\item $r_1 \wedge r_2 \wedge r_3 \in A_y$.
\end{enumerate}
\bn This completes the proof, since then $r_1 \wedge r_2 \wedge r_3 \in \im(\varphi)$.  The latter property follows by construction, so it suffices to prove the former.  Let $z$ denote the third edge of a 2--cell $\sigma$ with $x \subseteq \sigma$, $y \subseteq \sigma$.   Label $\calH(z) = \{\calH_0^z, \calH_1^z\}$ such that $\calH_0^z \subseteq \calH_0^x$.  Therefore we have $\calH_0^z = \calH_0^x \cap \calH_1^y$ and $\calH_1^z = \calH_0^y + \calH_1^x$.  For notational convenience, let $u_i = 0$ for $3 \leq i \leq 8$.  By construction, the vectors $\proj_{[b]^{\perp} \cap \calH_k^w}(v_i)$ for each choice of $1 \leq i \leq 8$, $0 \leq k \leq 1$ and $w = x,y,z$ are given as follows:
\begin{multicols}{2}
\begin{enumerate}
    \item $\proj_{[b]^{\perp} \cap \calH_0^y}(v_i) = h_i$ for $1 \leq i \leq 8$,
    \item $\proj_{[b]^{\perp} \cap \calH_1^y}(v_i) = u_i+ w_{i,1}$ for $1 \leq i \leq 8$,
    \item $\proj_{[b]^{\perp} \cap \calH_0^z}(v_i) = u_i$ for $1\leq i \leq 8$, and
    \item $\proj_{[b]^{\perp} \cap \calH_1^z}(v_i) = h_i + w_{i,1}$ for $1 \leq i \leq 8.$
\end{enumerate}
\end{multicols}
\bn Now, given these projections, we can compute ranks and intersection numbers as follows.
\begin{enumerate}
    \item We have $\rk^{\calV}_{i,k}(y), \rk^{\calV}_{i,k}(z) \leq 1$ for $1 \leq i \leq 8$ and $0 \leq k \leq 1$, since by assumption and construction each projection is primitive.
    \item We have $\left|\theta(\calV)_{i,j,0}(y) \right|= \left|\langle h_i, h_j \rangle \right| = \left|\theta(\calV)_{i,j,0}(x) \right| - \mathbbm{1}_{(i,j) = (1,2)}.$
    \item We have $\left|\theta(\calV)_{i,j,0}(z) \right| = \left|\langle u_i, u_j \rangle \right|=  \mathbbm{1}_{(i,j) = (1,2)}.$
\end{enumerate}
\bn  Furthermore, since $\langle v_i, v_j \rangle = 0$ by assumption, we have $\left|\theta(\calV)_{i,j,0}(y)\right| = \left|\theta(\calV)_{i,j,1}(y)\right|$, and similarly for $z$.  Therefore we have $\left|\theta(\calV)_{i,j,k}(x)\right| \geq \left|\theta(\calV)_{i,j,k}(y)\right|, \left|\theta(\calV)_{i,j,k}(z)\right|$ for all $1 \leq i,j \leq 8$ and $0 \leq k \leq 1$.  Additionally, we have assumed that $\left|\theta(\calV)_{i,j,0}(x)\right| > 1$, so we have $\left|\theta(\calV)_{1,2,0}(x)\right| > \left|\theta(\calV)_{1,2,0}(y)\right|, \left|\theta(\calV)_{1,2,0}(z)\right|$. 
Hence we have $y \in \algshrink_{\calV}(x)$, so the proof is complete.
\end{proof}

\bn We now complete the second step of the proof of Lemma~\ref{11finquotlemma}.

\begin{lemma}\label{11algonelemma}
Let $g \geq \finalbound$ and $a \subseteq S_g$ be a nonseparating simple closed curve.  Let $M \subseteq S_g \cut a$ be a nonseparating multicurve with $\left|\pi_0(M)\right| = 8$.  Let $b \subseteq S_g$ be a nonseparating simple closed curve such that $b$ has geometric intersection number with $a$, and such that $b$ intersects $M$ trivially.  Let $\calV = \{[c]: c \in \pi_0(M)\}$.  Let $\bE_{p,q}^r$ denote the equivariant homology spectral sequence for the action of $\cI_g$ on $\cC_{[a]}(S_g)$.  Let $(x,f) \in \bE_{1,1}^2$ be a class such that $\rk^{\calV}_{i,k}(x) \leq 1$ for all $1 \leq i \leq 8$ and $0 \leq k \leq 1$.  Then there is a relation in $\bE_{1,1}^2$ given by
\begin{displaymath}
(x,f) = \sum_{\ell = 1}^m \lambda_{\ell}(y_{\ell}, f_{\ell})
\end{displaymath}
\bn such that $\lambda_{\ell} \in \Q$, $\rk^{\calV}_{i,k}(y_{\ell}) \leq 1$, and $\theta(\calV)_{i,j,k}(y_{\ell}) \leq 1$ for all $1 \leq \ell \leq m$, $1 \leq i,j \leq 8$, and $0 \leq k \leq 1$.
\end{lemma}

\begin{proof}
Let $\maxalg(x) = \max_{1 \leq i,j \leq 8}\left|\theta(\calV)_{i,j,0}(x)\right|$.  Since $\theta(\calV)_{i,j,0}(x) = - \theta(\calV)_{i,j,1}(x)$, we only need take the maximum for $k = 0$.  Let $\nummaxalg(x)$ denote the number of pairs $1 \leq i < j \leq 8$ such that $\left|\theta(\calV)_{i,j,0}\right| = \maxalg(x)$.  The proof proceeds by double induction on $\maxalg(x)$ and $\nummaxalg(x)$.  

\p{Base case: $\maxalg(x) \leq 1$}  In this case, the relation in the lemma is the trivial relation $(x,f) = (x,f)$, so the lemma holds.

\p{Inductive step: the lemma holds for all $y \subseteq X_g$ with either $\maxalg(y) < \maxalg(x)$ or with both $\maxalg(y) \leq \maxalg(x)$ and $\nummaxalg(y) < \nummaxalg(x)$} By Lemma~\ref{11finquotshrinkgenlemma}, there is a linear combination
\begin{displaymath}
f = \sum_{\ell = 1}^m f_{\ell}
\end{displaymath}
\bn such that each $f_{\ell} \in A_{y_{\ell}}$, where $y_{\ell} \in \algshrink_{\calV}(x)$.  By definition, each $y_{\ell}$ shares a 2--cell $\sigma_{\ell}$ with $x$, and $f_{\ell} \in A_{\sigma_{\ell}}$.  Let $z_{\ell}$ denote the third edge of $\sigma_{\ell}$.  Then there is a relation in $\bE_{1,1}^2$ given by
\begin{displaymath}
\sum_{\ell = 1}^m d_{2,1}^{1}(\sigma_{\ell}, f_\ell) = \sum_{\ell=1}^m (x, f_{\ell}) + (y_{\ell}, f_{\ell}) - (z_{\ell}, f_{\ell}) = 0.
\end{displaymath}
\bn By rearranging terms and applying the fact that $f = \sum_{\ell = 1}^m f_{\ell}$, we have
\begin{displaymath}
(x,f) = \sum_{\ell=1}^m (z_{\ell}, f_{\ell})-(y_{\ell}, f_{\ell}).
\end{displaymath}
\bn But then $y_{\ell} \in \algshrink_{\calV}(x)$ for all $1 \leq \ell \leq m$, so for all $1 \leq \ell \leq m$, either $\maxalg(y_{\ell}) < \maxalg(x)$ or $\maxalg(y_{\ell}) = \maxalg(x)$ and $\nummaxalg(y_{\ell}) < \nummaxalg(x_{\ell})$, and similarly for $z_{\ell}$.  Therefore the classes $(y_{\ell}, f_{\ell})$ and $(z_{\ell}, f_{\ell})$ are linear combinations of classes as in the statement of the lemma by the inductive hypothesis, so the lemma holds for $(x,f)$ as well. 
\end{proof}

\bn We are now ready to conclude Section~\ref{11finquotsection}. 

\begin{proof}[Proof of Lemma~\ref{11finquotlemma}]
Let $W \subseteq \bE_{1,1}^2$ denote the subspace of $\bE_{1,1}^2$ spanned by elements $(x,f)$ where $x$ satisfies $\rk^{\calV}_{i,k}(x), \theta(\calV)_{i,j,k}(x) \leq 1$ for all $1 \leq i < j \leq 8$, $0 \leq k \leq 1$.  We will show that $W = \bE_{1,1}^2$.  By Lemma~\ref{intermediatequotientlemma}, it suffices to show that $(x,f) \in W$ for any $x \subseteq X_g$ and $f \in A_x$.
This follows from Lemmas~\ref{11rkonelemma} and~\ref{11algonelemma}.  Now, as a consequence of  Lemma~\ref{11finquotalginvlemma}, there is a finite set of edges $y_1,\ldots, y_n \subseteq X_g$ given by all possible combinations of genera $g(\calH(y_\ell))$ and choices of $\rk^{\calV}_{i,k}$ and $\left|\theta(\calV)_{i,j,k}\right|$ less than or equal to one, such that any $x \subseteq X_g$ with $\rk^{\calV}_{i,k}(x), \theta(\calV)_{i,j,k}(x) \leq 1$ for all $1 \leq i < j \leq 8$, $0 \leq k \leq 1$ is in the same $G$--orbit as some $y_\ell$, so the natural map
\begin{displaymath}
\bigoplus_{1 \leq \ell \leq n} \Ind_{\Stab_G(y_\ell)}^GA_{y_\ell} \rightarrow W
\end{displaymath}
is surjective.  Therefore the map
\begin{displaymath}
H_0\left(G; \bigoplus_{1 \leq \ell \leq n} \Ind_{\Stab_G(y_\ell)}^GA_{y_\ell}\right) \rightarrow H_0(G;W)
\end{displaymath}
\bn is surjective since $H_0(G, -)$ is left exact.  Then Shapiro's lemma says that
\begin{displaymath}
H_0\left(G; \bigoplus_{1 \leq \ell \leq n} \Ind_{\Stab_G(y_\ell)}^GA_{y_i}\right) \cong \bigoplus_{1 \leq \ell \leq n} H_0(\Stab_G(y_\ell); A_{y_\ell}).
\end{displaymath}
\bn But then $A_{y_{\ell}}$ is contained in $H_1(\Stab_{\cI_g}(a);\Q)$ for all $y_{\ell}$, so $A_{y_{\ell}}$ is finite dimensional for all $y_{\ell}$.  Therefore $H_0(\Stab_G(y_\ell);A_{y_\ell})$ is finite dimensional for any $1 \leq \ell \leq n$, so the proof is complete.
\end{proof}

\section{The Proof of Theorem~\ref{fincokthm}}\label{fincoksection}

In this section, we will complete the proof of Proposition~\ref{11finprop}, which, along with Proposition~\ref{homolcurvequotprop}, completes the proof of Theorem~\ref{fincokthm}.  For the remainder of this section, unless otherwise specified, fix a $g \geq \finalbound$ and $a \subseteq S_g$ a nonseparating curve.  We will also let $\bE_{p,q}^r$ denote the equivariant homology spectral sequence for the action of $\cI_g$ on $C_{[a]}(S_g)$. 

\p{The outline of Section~\ref{fincoksection}} We will devote the bulk of the section to proving Lemma~\ref{11stabgenlemma}, which is done in Section~\ref{11stabgensection}.  The statement of Lemma~\ref{11stabgenlemma} requires some notation, so we defer the statement for a moment.  We then use Lemma~\ref{11finquotlemma} and~\ref{11stabgenlemma} to prove Proposition~\ref{11finprop} in Section~\ref{fincoksubsection}, and then use Proposition~\ref{11finprop} and Proposition~\ref{homolcurvequotprop} to prove Theorem~\ref{fincokthm}.

\subsection{The proof of Lemma~\ref{11stabgenlemma}}\label{11stabgensection}

Let $\calV= \{v_1,\ldots, v_k\} \subseteq [a]^{\perp}$ be a set of primitive elements.  Let $\tau_g$ denote the Johnson homomorphism.  Let $x \subseteq X_g$ be an edge, and let $A_x^{\calV}$ denote the subspace of $\im(\tau_g) = \wedge^3H_1(S_g;\Q)/H_1(S_g;\Q)$ given by
\begin{displaymath}
\left(\im\left(\midwedge^3 \calH_0^x \oplus \midwedge^3 \calH_1^x \rightarrow \im(\tau_g)\right) \cap \im \left( \midwedge^3\calV^{\perp} \rightarrow \im(\tau_g)\right)\right) \otimes \Q.
\end{displaymath} 
\bn Let $X_g^{\calV} \subseteq X_g$ denote the subcomplex consisting of cells $\sigma$ such that $\calH(\sigma)$ is \textit{compatible with $\calV$}, i.e., every $v \in \calV$ satisfies $v \in \calH$ for some $\calH \in \calH(\sigma)$.  Let $\bE_{1,1}^{2, \calV}$ denote the image of the composition
\begin{displaymath}
\bigoplus_{x \in \left(X_g^{\calV}\right)^{(1)}} A_x^{\calV} \rightarrow \bigoplus_{x \in X_g^{(1)}} A_x \rightarrow \bE_{1,1}^{2}
\end{displaymath} 
\bn where $A_x$ is an in Section~\ref{11finquotsection}.  If $\calV = \{v\}$ is a singleton, we will denote $\bE_{1,1}^{2,\calV}$ and $A_x^{\calV}$ by $\bE_{1,1}^{2,v}$ and $A_x^v$ respectively.  We are now ready to state Lemma~\ref{11stabgenlemma}.

\begin{lemma}\label{11stabgenlemma}
Let $\calV = \{v_1,\ldots, v_9\} \subseteq [a]^{\perp}$ be a set of primitive elements such that there is a nonseparating multicurve $M \subseteq S_g \cut a$ with $\calV = \{[c]: c \in \pi_0(M)\}$.  Then the natural map
\begin{displaymath}
\xi: \bigoplus_{v_i \in \calV} \bE_{1,1}^{2,v_i} \rightarrow \bE_{1,1}^2
\end{displaymath}
\bn is surjective.  
\end{lemma}

\p{The outline of the proof of Lemma~\ref{11stabgenlemma}} We will prove Lemmas~\ref{unimodwedgelemma},~\ref{unimodwedgelemmapt2} and~\ref{genus2homolcycleh1genlemma}.  The first two lemmas are statements about generating sets for vector spaces equipped with alternating forms, while the last is a statement about a generating set for the first rational homology of a certain subcomplex of $X_g$.  In order to prove Lemma~\ref{genus2homolcycleh1genlemma}, we will prove Lemma~\ref{h1xgqdescription} and Lemma~\ref{h1torellipullbacklemma}, which are auxiliary results about rational abelianizations of subgroups and quotients of the Torelli group.  We use Lemma~\ref{unimodwedgelemma}, Lemma~\ref{unimodwedgelemmapt2},  and Lemma~\ref{genus2homolcycleh1genlemma} to prove Lemma~\ref{11stabgenlemma}. 

\begin{lemma}\label{unimodwedgelemma}
Let $V$ be a finite dimensional $\Q$--vector space equipped with an alternating form $\langle \cdot, \cdot \rangle$.  Let $v_1,v_2,v_3$ be elements in $V$ such that the image of the set $\{v_1,v_2,v_3\}$ under the adjoint map $V \rightarrow \ho_{\Q}(V,\Q)$ is linearly independent.  Then the natural map
\begin{displaymath}
\psi: \bigoplus_{i \in \{1,2,3\}} \midwedge^2 v_i^{\perp} \rightarrow \midwedge^2V
\end{displaymath}
\bn is surjective.
\end{lemma}

\begin{proof}
This follows by a dimension count.  Let $\calB = \{a_1,\ldots, a_n\}$ be a basis for $V$ such that $\langle v_i,a_j\rangle = \delta_{ij}$.  Let $\calB_i = \{a_1,\ldots, \widehat{a_i}, \ldots, a_n\}$, which is a basis for $v_i^{\perp}$ for any $i \in \{1,2,3\}$.  The vector space $\midwedge^2 V$ has a basis consisting of pairs of elements in $\calB$.  Each pair of these elements is contained in at least one $\calB_i$.  Therefore $\im(\psi)$ contains a basis for $\midwedge^2 V$, so $\psi$ is surjective.
\end{proof}
\bn We now extend Lemma~\ref{unimodwedgelemma} as follows.
\begin{lemma}\label{unimodwedgelemmapt2}
Let $V$ be a finite dimensional $\Q$--vector space equipped with an alternating form $\langle \cdot, \cdot \rangle$.  Let $\calV = \{v_1,\ldots, v_n\}$ be a set of elements in $V$ with $\left|\calV\right| \geq 3$ such that the image of $\calV$ under the adjoint map $V \rightarrow \ho_{\Z}(V,\Z)$ is linearly independent.  Let $m \leq \left|\calV\right| -2$ be a natural number.  Then the natural map
\begin{displaymath}
\psi_m: \bigoplus_{\calV' \subseteq \calV: \left| \calV'\right| = m} \midwedge^2\left(\calV'\right)^{\perp}\rightarrow \midwedge^2 V
\end{displaymath}
\bn is surjective.
\end{lemma}

\begin{proof}
This follows by induction on $m$.

\p{Base case: $m = 1$} This follows from Lemma~\ref{unimodwedgelemma}.

\p{Inductive step: the lemma holds for $m' < m$} Let $\psi_m$ be as in the lemma.  By the inductive hypothesis, the map $\psi_{m-1}$ is surjective.  Hence it suffices to show that $\midwedge^2\left(\calW^{\perp}\right)\subseteq \im(\psi_m)$ for any $\calW \subseteq \calV$ with $\left| \calW \right| = m - 1$.  Let $\calW \subseteq \calV$ be a subset with $\left|\calW \right| = m-1$.  Since we have chosen $m$ with $m \leq \left|\calV\right| - 2$, we have $\left|\calV \setminus \calW\right| \geq 3$.  Hence Lemma~\ref{unimodwedgelemma} applied to $L = \calW^{\perp}$ and the set $\calV \setminus \calW$ says that the map
\begin{displaymath}
\bigoplus_{v \in \calV \setminus \calW} \midwedge^2(v^{\perp} \cap \calW^{\perp}) \rightarrow \midwedge^2 \calW^{\perp}
\end{displaymath}
\bn is surjective.  Since $\midwedge^2(v^{\perp} \cap \calW^{\perp}) \subseteq \im(\varphi_m)$ for any $v \in \calV \setminus \calW$ because $\left|\{v\} \cup \calW\right| = m$, we have $\midwedge^2 \calW^{\perp} \subseteq \im(\psi_m)$.  Hence $\im(\psi_m) = \im(\psi_{m-1})$ and the inductive hypothesis says that $\im(\psi_{m-1})$ is surjective, so $\im(\psi_m)$ is surjective as well.  
\end{proof}

\bn We will now give an explicit description of the vector space $H_1(X_g;\Q)$.  We first give an alternate description of the Johnson homomorphism.

\p{The Johnson homomorphism, alternate version} Let $\pi_1^{(k)}(S_g)$ denote the $k$th term of the lower central series of the fundamental group of $\pi_1(S_g)$ (we suppress the basepoint in the notation, since the choice of basepoint does not affect the construction).  The \textit{Johnson homomorphism} is a map
\begin{displaymath}
\tau_g: \cI_g \rightarrow \ho_{\Z}\left(\pi_1(S_g)/\pi_1^{(1)}(S_g), \pi_1^{(1)}(S_g)/\pi_1^{(2)}(S_g)\right).
\end{displaymath}
\bn If $\gamma \in \pi_1(S_g)$ is a loop and $f \in \cI_g$ is a mapping class, then $\tau_g(f)(\gamma) = \gamma^{-1}f(\gamma)$, where the element $\gamma^{-1}f(\gamma)$ is only defined up to conjugation by $\pi_1(S_g)$.  Johnson showed that this is well defined as a map $\pi_1(S_g)/\pi_1^{(1)}(S_g) \rightarrow \pi_1^{(1)}(S_g)/\pi_1^{(2)}(S_g)$.  Now, if $\omega' \in \midwedge^2 H_1(S_g;\Z)$ is given by $a_1 \wedge b_1 + \ldots + a_g \wedge b_g$ for $\{a_i,b_i\}_{1 \leq i \leq g}$ a symplectic basis for $H_1(S_g;\Z)$, then the Johnson homomorphism can be rewritten as a map
\begin{displaymath}
\tau_g: \cI_g \rightarrow \ho_{\Z}\left(H_1(S_g;\Z), \midwedge^2 H_1(S_g;\Z)/\Z\omega'\right).
\end{displaymath}
\bn This description of the Johnson homomorphism allows us to prove Lemma~\ref{h1xgqdescription}.  If $\calW \subseteq H_1(S_g;\Z)$ is a set of elements, recall that $X_g^{\calW} \subseteq X_g$ denotes the subcomplex of $X_g$ generated by elements $\sigma$ such that $\calH(\sigma)$ is compatible with $\calW$, i.e., for each $w\in \calW$ there there is an $\calH \in \calH(\sigma)$ such that $w \in \calH$.  If $\alpha_1,\beta_1,\ldots, \alpha_g, \beta_g$ are a symplectic basis for $H_1(S_g;\Z)$ with $[a] = \alpha_1$, let $\omega'_a = \alpha_2 \wedge \beta_2 + \ldots + \alpha_g \wedge \beta_g$.  If $V \subseteq H_1(S_g;\Q)$ is a subspace, let $X_g^V$ denote the subcomplex of $X_g$ consisting of cells $x$ such that $V \subseteq \calH \otimes \Q$ for some $\calH \in \calH(\sigma)$.

\begin{lemma}\label{h1xgqdescription}
Let $g \geq 4$ and let $a \subseteq S_g$ be a nonseparating simple closed curve.  Then there is an isomorphism $\overline{\tau_g}: H_1(X_g;\Q) \cong \midwedge^2[a]^{\perp} \otimes \Q/\Q\omega'_a$.  Furthermore, this isomorphism is functorial in the following sense.  Let $\calW \subseteq [a]^{\perp}$ be a set of elements such that $g(\calW^{\perp} \cap [a]^{\perp}) \geq 1$.  Then $\im(H_1(X_g^{\calW};\Q) \rightarrow H_1(X_g;\Q) \rightarrow \midwedge^2 [a]^{\perp} \otimes \Q/\Q\omega_a') \supseteq \im(\midwedge^2(\calW^{\perp} \cap [a]^{\perp}) \otimes \Q\rightarrow \midwedge^2 [a]^{\perp} \otimes \Q/\Q\omega_a') $.
\end{lemma}

\begin{proof}
Since $g \geq 4$, the complex $\cC_{[a]}(S_g)$ is $1$--acyclic by a theorem of the author~\cite[Theorem A]{Minahanhomolconn}.  Therefore the last three terms of the five term exact sequence associated to the equivariant homology spectral sequence for the action of $\cI_g$ on $\cC_{[a]}(S_g)$ form a right exact sequence
\begin{displaymath}
H_1(\Stab_{\cI_g}(a);\Q) \rightarrow H_1(\cI_g;\Q) \rightarrow H_1(X_g;\Q) \rightarrow 0.
\end{displaymath}
\bn We have the following claim.

\p{Claim} We have $\im(H_1(\Stab_{\cI_g}(a);\Q) \rightarrow H_1(\cI_g;\Q)) = \ker(\eval_{[a]} \circ \tau_g)$, where 
\begin{displaymath}
\eval_{[a]}:\ho_{\Z}(H_1(S_g;\Z), \midwedge^2 H_1(S_g;\Z)/\Z\omega') \rightarrow \midwedge^2 H_1(S_g;\Z)/\Z\omega'
\end{displaymath}
\bn is the linear map given by evaluation on $[a]$.

\p{Proof of claim} We prove each containment in turn.  If $f \in \Stab_{\cI_g}(a)$, we have $\tau_g(f)([a]) = 0$, so $f \in \ker(\eval_{[a]})$ and the $\subseteq$ containment holds.  For the $\supseteq$ containment, let $f \in \cI_g$ be a mapping class such that $\eval_{[a]}(\tau_g(f)) = 0$.  Let $\gamma$ be a loop in $\pi_1(S_g)$ such that $\gamma$ is homotopic to $a$ as an unbased loop.  Since the choice of representative of $[a]$ is arbitrary, we must have $\gamma^{-1}f(\gamma) \in \pi_1^{(2)}(S_g)$ since $\tau_g(f)(a) = 0$ by hypothesis.  A theorem of Church~\cite[Theorem 1.1]{Churchorbit} tells us that there is some $h \in \calK_g$, where $\calK_g$ is the Johnson kernel~\cite[Section 6.6]{FarbMarg}, such that $h a = f(a)$, so $h^{-1}f(a) = a$, and therefore $h^{-1}f \in \Stab_{\cI_g}(a)$.  Johnson~\cite[Lemma 4A]{Johnsonhomomorphism} showed that $h \in \ker(\tau_g)$, so $[h^{-1}f] = [f] \in H_1(\cI_g;\Q)$, since $\im(\tau_g) \otimes \Q \cong H_1(\cI_g;\Q)$~\cite{JohnsonIII}.  Therefore $[f] \in \im(H_1(\Stab_{\cI_g}(a);\Q) \rightarrow H_1(\cI_g;\Q))$, so the claim holds.

\medskip

Given the claim, we have an exact sequence
\begin{displaymath}
H_1(\Stab_{\cI_g}(a);\Q) \rightarrow H_1(\cI_g;\Q) \xrightarrow{\eval_{[a]}} \midwedge^2H_1(S_g;\Q)/\Q\omega'.
\end{displaymath}
\bn Therefore, it suffices to show that $\im\left(\eval_{[a]}\right) = \midwedge^2[a]^{\perp}/\Q \omega'_a$.  This is a consequence of Johnson's computation of the image of the Johnson homomorphism~\cite[Theorem 1]{Johnsonhomomorphism}.

We now use Lemma~\ref{wedgegenlemma} to prove the second part of the lemma.  In particular, it suffices to show that any element $\gamma \wedge \delta \in \midwedge^2 \calW^{\perp} \cap \midwedge^2 [a]^{\perp}$ with $\gamma, \delta$ primitive and $\langle \gamma, \delta \rangle = 1$ lies in the image of the composition
\begin{displaymath}
H_1(X_g^{\calW};\Q) \rightarrow H_1(X_g;\Q) \rightarrow \midwedge^2[a]^{\perp} \otimes \Q/Q\omega'_a.
\end{displaymath}
\bn Let $x \subseteq X_g$ be an edge, $\widehat{x}$ a lift of $x$ to $\cC_{[a]}(S_g)$ with one vertex equal to $a$, and $f \in \cI_g$ a mapping class taking $a$ to the other endpoint of $\widehat{x}$.  By the construction of the equivariant homology spectral sequence, the class $[x] \in H_1(X_g;\Q)$ is the image under the map $H_1(\cI_g;\Q) \rightarrow H_1(X_g;\Q)$ of the class $[f] \in H_1(\cI_g;\Q)$.  Then by alternate definition of $\tau_g$ given before the lemma, we have $\eval_{[a]}(f) = \omega_{\calH}$, where $\omega_\calH \in \midwedge^2 H_1(S_g;\Q)$ is the characteristic element for some $\calH \in \calH(x)$.  Then for any $\gamma \wedge \delta$ with $\gamma, \delta$ primitive and $\langle \gamma, \delta \rangle = 1$, there is an $x \subseteq X_g$ and $\calH \in \calH(x)$ with $\omega_\calH = \gamma \wedge \delta$.  The set of such $\gamma \wedge \delta$ spans $\midwedge^2 [a]^{\perp} \otimes \Q$ by Lemma \ref{wedgegenlemma}, so the image of the composition 
\begin{displaymath}
H_1(X_g^{\calW};\Q) \rightarrow H_1(X_g;\Q) \rightarrow \midwedge^2[a]^{\perp} \otimes \Q/Q\omega'_a.
\end{displaymath}
\bn contains $\im(\midwedge^2(\calW^{\perp} \cap [a]^{\perp}) \otimes \Q\rightarrow \midwedge^2 [a]^{\perp} \otimes \Q/\Q\omega_a')$, so the lemma is complete.
\end{proof}

\bn We will need another auxiliary lemma about the rational abelianizations of the Torelli groups of surfaces.  

\begin{lemma}\label{h1torellipullbacklemma}
Let $g \geq 4$, and let $\Tau_0, \Tau_1$ be two surfaces each equipped with an embedding $\iota^i:\Tau_i \hookrightarrow S_g$.  Assume that the following hold:
\begin{multicols}{2}
\begin{itemize}
\item each embedding $\iota^i$ is clean,
\item $S_g \cut \iota^i(\Tau_i)$ is connected,
\item $g(\Tau_i) \geq 3$ for $i = 0,1$,
\item the pullback $S$ of the maps $\iota^0$ and $\iota^1$ is a connected, smooth manifold, 
\item $S_g \cut \left(\iota^0(\Tau_0) \cap \iota^1(\Tau_1)\right)$ is connected, and
\item $g(S) \geq 3$.
\end{itemize}
\end{multicols}
\bn The following commuting square

\centerline{\xymatrix{
H_1(\cI(S, S_g);\Q) \ar[r]\ar[d]& H_1(\cI(\Tau_1,S_g);\Q) \ar[d]^{\iota_*^1}\\
H_1(\cI(\Tau_2,S_g);\Q) \ar[r]^{\iota_*^2}& H_1(\cI_g;\Q),
}}

\bn is a pullback square.
\end{lemma}

\begin{proof}
For each $\Tau_i$, let $\kappa^i$ denote the map $S \rightarrow \Tau_i$, and let $\iota$ denote the composition $\iota^i \circ \kappa^i$.  Note that since $S$ is a pullback, we have $\iota^1 \circ \kappa^1 = \iota^0 \circ \kappa^0$, so the definition of $\iota$ does not depend on $i$.  Now, since we have assumed that each embedding $\iota^i$ is clean, we have $\iota^i_*$ injective by a theorem of Putman~\cite[Theorem B]{Putmanjohnson}.  The map $\iota$ is also clean by hypothesis, and therefore the pushforward $\iota_*$ is injective by the same theorem of Putman, so it suffices to show that $\im(\iota_*) = \im(\iota_*^0) \cap \im(\iota_*^1)$.  Lemma~\ref{firsthomolratimlemma} says that that $\im(\cI(\iota^i)_*) = \im(\midwedge^3 H_1(\Tau_i;\Q) \rightarrow \midwedge^3 H_1(S_g;\Q))$ and $\im(\cI(\iota)_*) = \im(\midwedge^3 H_1(S;\Q) \rightarrow \midwedge^3 H_1(S_g;\Q))$.  Since the functor $\midwedge^3$ from $\Q$--vector spaces to $\Q$--vector spaces sends pullbacks of monomorphisms to pullbacks of monomorphisms, it is enough to show that 
\begin{displaymath}
\im(H_1(\Tau_0;\Q) \rightarrow H_1(S_g;\Q)) \cap \im(H_1(\Tau_1;\Q) \rightarrow H_1(S_g;\Q)) = \im(H_1(S;\Q) \rightarrow H_1(S_g;\Q)). 
\end{displaymath} 
\bn  This follows from our hypotheses that $S_g \cut \iota(S)$ and $S_g \cut \iota^i(\Tau_i)$ are connected, so embedding $S$ and $\Tau_i$ into $S_g$ does not introduce any new relations in $H_1(S;\Q)$ or $H_1(\Tau_i;\Q)$.
\end{proof}

\bn We now prove the following.

\begin{lemma}\label{genus2homolcycleh1genlemma}
Let $L \subseteq [a]^{\perp}$ be a free abelian subgroup with $g(L) = 2$ and $\dim(L \otimes \Q) = 4$.  Let $V = L \otimes \Q$.  Let $\calV$ be as in Lemma~\ref{11stabgenlemma}.  Then the natural map
\begin{displaymath}
\psi: \bigoplus_{\calV' \subseteq \calV: |\calV'| = 4} H_1(X_g^V \cap X_g^{\calV'};\Q) \rightarrow H_1(X_g^V;\Q)
\end{displaymath}
\bn is surjective.
\end{lemma}

\begin{proof}
We begin with the following claim.

\p{Claim} The pushforward $H_1(X_g^V;\Q) \rightarrow H_1(X_g;\Q)$ is an injection, and the image of this map is sent to $\midwedge^2 V^{\perp} \cap \midwedge^2 [a]^{\perp}$ under the isomorphism $H_1(X_g;\Q) \cong \midwedge^2 [a]^{\perp}/\Q \omega_{\alpha}$ from Lemma~\ref{h1xgqdescription}.

\p{Proof of claim} Let $S_2^1 \subseteq S_g$ be a compact subsurface such that $\im\left(H_1(S_2^1;\Z) \hookrightarrow H_1(S_g;\Z)\right) = V$.  Let $\cI_{g-2}^1$ denote the subgroup of $\cI_g$ generated by elements that fix $\partial S_2^1$ and restrict to the identity on $S_2^1$.  Let $\cC_{[a]}(S_g, S_2^1)$ denote the subcomplex of $\cC_{[a]}(S_g)$ generated by curves $c$ such that $c$ is disjoint from $S_2^1$.  By a result of Kent, Leininger and Schleimer~\cite[Theorem 7.2]{KLS}, the fibers of the natural map $\cC_{[a]}(S_g, S_2^1) \rightarrow \cC_{[a]}(S_{g-2})$ are all trees, so $\cC_{[a]}(S_g, S_2^1)$ is homotopy equivalent to $\cC_{[a]}(S_{g-2})$.  Then a theorem of the author~\cite[Theorem A]{Minahanhomolconn} says that $\cC_{[a]}(S_{g-2})$ is at least 1--acyclic, and thus $\cC_{[a]}(S_g, S_2^1)$ is at least 1--acyclic as well.  Therefore the equivariant homology spectral sequence $\bE_{p,q}^r(\cI^{1}_{g-2},\cC_{[a]}(S_g, S_2^1);\Q)$ converges to $H_1(\cI^{1}_{g-2};\Q)$.  Hence there is a right exact sequence
\begin{displaymath}
H_1(\Stab_{\cI_{g-2}^1}(a);\Q) \rightarrow H_1(\cI_{g-2}^1;\Q) \rightarrow H_1(X_g^V;\Q) \rightarrow 0.
\end{displaymath}
\bn Since $g \geq \finalbound$, we have $g\left(S_{g-2}^1 \cut a\right) \geq 3$.  Then the inclusion $S_{g-2}^1 \cut a \hookrightarrow S_{g-2}^1$ is a clean embedding (as in Section~\ref{homolcurvequotsection}), so a theorem of Putman~\cite[Theorem B]{Putmanjohnson} says that the pushforward map $H_1(\Stab_{\cI_{g-2}^1}(a);\Q) \rightarrow H_1(\cI_{g-2}^1;\Q)$ is an injection.  Hence the above right exact sequence is in fact exact:
\begin{displaymath}
0 \rightarrow H_1(\Stab_{\cI_{g-2}^1}(a);\Q) \rightarrow H_1(\cI_{g-2}^1;\Q) \rightarrow H_1(X_g^V;\Q) \rightarrow 0.
\end{displaymath}
\bn Furthemore, there is a morphism of short exact sequences

\centerline{\xymatrix{
0\ar[r]& H_1(\Stab_{\cI_{g-2}^1}(a);\Q)\ar[d]^{\rho_a}\ar[r]& H_1(\cI_{g-2}^1;\Q) \ar[d]^{\rho}\ar[r]& H_1(X_g^V;\Q) \ar[d]^{\overline{\rho}}\ar[r]& 0\\
0\ar[r]& H_1(\Stab_{\cI_g}(a);\Q)\ar[r]& H_1(\cI_g;\Q) \ar[r]& H_1(X_g;\Q) \ar[r]& 0,
}}

\bn where $\rho_a$ and $\rho$ are the natural pushforward maps.  Lemma~\ref{h1torellipullbacklemma} says that left square is a pullback.  Furthermore, each map in this square is an injection by the aforementioned theorem of Putman, so the map $\cok(\rho_a) \rightarrow \cok(\rho)$ is injective.  Furthermore, since $\rho$ is injective, we have $\ker(\rho) = 0$.  Then the snake lemma says that there is an exact sequence
\begin{displaymath}
\ker(\rho) \rightarrow \ker(\overline{\rho}) \rightarrow \cok(\rho_a) \rightarrow \cok(\rho).
\end{displaymath}
\bn The map $\cok(\rho_a) \rightarrow \cok(\rho)$ is injective and $\ker(\rho) = 0$, so $\ker(\overline{\rho}) = 0$, and in particular we that $\overline{\rho}:H_1(X_g^V;\Q) \rightarrow H_1(X_g;\Q)$ is an injection.  The second part of the claim follows from the fact that for any edge $x \subseteq X_g^V$, the characteristic element $\omega_\calH$ for any $\calH \in \calH(x)$ lies in $\midwedge^2 V^{\perp} \cap \midwedge^2 [a]^{\perp} \otimes \Q$, because $V$ is a vector space.  This implies that $\im(H_1(X_g^V;\Q) \rightarrow H_1(X_g;\Q)$ is contained in $\im(\midwedge^2 V^{\perp} \cap \midwedge^2[a]^{\perp} \otimes \Q$.  Then the second part of Lemma~\ref{h1xgqdescription} gives the containment in the other direction, so we have equality.  

Given the claim, we now continue with the proof of the lemma.  The claim says that $H_1(X_g^V;\Q) \cong \midwedge^2 [V]^{\perp} \cap [a]^{\perp} \otimes \Q$.  Since $\dim(V) = 4$ and the elements of $\calV$ have pairwise trivial algebraic intersection, the set $\proj_{V^{\perp} }(\calV)$ contains at least six linearly independent elements.  Let $\cB \subseteq \calV$ be a subset of maximal size with $\cB' = \proj_{V^{\perp} }(\cB)$ linearly independent.  Since we have assumed that each element $v \in \calV$ has trivial algebraic intersection with $[a]$, we have $\cB \subseteq [a]^{\perp} \cap V^{\perp}$.  Then since $\left|\cB'\right| = 6$, Lemma~\ref{unimodwedgelemmapt2} applied to the set $\cB'$ and the vector space $V^{\perp} \cap [a]^{\perp}$ with $m = 4$ says that the natural map
\begin{displaymath}
\psi: \bigoplus_{\calV' \subseteq \cB': \left|\calV'\right| = 4} \midwedge^2(V^{\perp} \cap (\calV')^{\perp} \cap [a]^{\perp}) \otimes \Q \rightarrow \midwedge^2(V^{\perp} \cap [a]^{\perp}) \otimes \Q
\end{displaymath}
\bn is a surjection.  Then, for any $v \in \calB$, we have $v^{\perp} \cap V^{\perp} = w^{\perp} \cap V^{\perp}$, where $w = \proj_{V}(v)$, so the natural map
\begin{displaymath}
\psi: \bigoplus_{\calV' \subseteq \cB: \left|\calV'\right| = 4} \midwedge^2(V^{\perp} \cap (\calV')^{\perp}) \otimes \Q \rightarrow \midwedge^2(V^{\perp} \cap [a]^{\perp}) \otimes \Q
\end{displaymath}
\bn is a surjection.  We have $H_1(X_g^V) \cong \midwedge^2(V^{\perp} \cap [a]^{\perp}) \otimes \Q$ by the claim.  Furthermore, by Lemma~\ref{h1xgqdescription}, we have $\im(H_1(X_g^V \cap X_g^v);\Q) \rightarrow H_1(X_g;\Q)) \supseteq \midwedge^2(V^{\perp} \cap v^{\perp})$ for any $v \in \calB$.  Then by substituting in each of these homology groups in the previous equality, we see that the natural map
\begin{displaymath}
\psi: \bigoplus_{\calV' \subseteq \cB: \left|\calV'\right| = 4} H_1(X_g^{V} \cap X_g^{\calV'};\Q) \rightarrow H_1(X_g^V;\Q)
\end{displaymath}
\bn is surjective, so the lemma is complete.
\end{proof}

\bn We are now almost ready to conclude Section~\ref{11stabgensection}.  We will first prove the following auxiliary result, and then we will prove Lemma~\ref{11stabgenlemma}.

\begin{lemma}\label{11stabedgefixlemma}
Let $(x,f) \in \bigoplus_{x \in X_g^{(1)}}A_x$ be a class.  There is a linear combination in $\bE_{1,1}^2$ given by
\begin{displaymath}
(x,f)= \sum_{i = 1}^m (x_i, f_i)
\end{displaymath}
\bn such that, for each $1 \leq i \leq m$, each $x_i$ has a representative $\widehat{x}_i \subseteq \cC_{[a]}(S_g)$ and $f_i$ has a representative $F_i \in \Stab_{\cI_g}(\widehat{x}_i) $.  These representatives have the property that there is a subsurface $\Tau_i \subseteq S_g$ with $\Tau_i \cong S_2^1$, $F_i$ supported on $\Tau_i$, and $\widehat{x}_i$ disjoint from $\Tau_i$.
\end{lemma}

\begin{proof}
Let $\widehat{x} \subseteq \cC_{[a]}(S_g)$ be a representative for $x$.  Let $S', S''$ be the connected components of $S_g \cut \widehat{x}$.  We have a surjection $\cI(S', S_g) \times \cI(S'', S_g) \rightarrow \Stab_{\cI_g}(\widehat{x})$, and hence by the K\"unneth formula we have a surjection
\begin{displaymath}
H_1(\cI(S', S_g);\Q) \oplus H_1(\cI(S'', S_g);\Q) \rightarrow H_1(\Stab_{\cI_g}(\widehat{x});\Q).
\end{displaymath}
\bn Hence we may assume that, without loss of generality, $f$ is represented by a mapping class $F \in \cI(S', S_g)$.  Furthermore, we may assume that $g(S') \geq 4$.  Indeed, if $g(S') \leq 3$, then $g \geq \finalbound \geq 7$ implies that there is a 2--cell $\widehat{\sigma} \subseteq \cC_{[a]}(S_g)$ such that $\widehat{x} \subseteq \widehat{\sigma}$, and such that for $\widehat{y}, \widehat{z} \subseteq \widehat{\sigma}$ the other two edges, the connected component of $S_g \cut \widehat{y}$ that contains $S'$ has genus at least 3, and similarly for $\widehat{z}$.  Then since $g(S') \geq 4 \geq 3$, Lemma~\ref{putmancorrlemma} says that $\cI(S', S_g)$ is generated by bounding pair maps.  Then bounding pair maps supported on separating curves vanish in $H_1(\cI_g;\Q)$ since Dehn twists along separating curves vanish under $\tau_g$.  Then Lemma~\ref{firsthomolratimlemma} says that if $\iota:S' \rightarrow S_g$ is the inclusion map, the pushforward $\cI(\iota)_*$ in $H_1$ is injective, so bounding pair maps supported on separating curves are trivial in $H_1(\cI(S',S_g);\Q)$.  Hence the class $f$ is a linear combination of classes represented by bounding pair maps supported on nonseparating curves contained in $S'$, so we may assume that $f$ has a representative $T_{c,c'} \in \cI(S', S_g)$ for $c \cup c'$ a bounding pair with $[c] \neq 0$.  We now have two cases.

\p{Case 1: $[c] \neq [a]$} We first show that we can assume that no connected component of $S' \cut (c \cup c')$ has genus zero.  Suppose otherwise, so one connected component of $S' \cut (c \cup c')$ is a subsurface $P$ with $P \cong S_0^4$.  Since we have assumed that $g(S') \geq 4 \geq 3$, there is another curve $c''$ disjoint from $c \cup c'$ and not equal to $c$ or $c''$ with $[c''] = [c]$ and such that both connected components of $S' \cut (c \cup c'')$ have positive genus.  Now, we have $T_{c,c'} = T_{c,c''} T_{c'', c'}$, and both $c \cup c''$ and $c'' \cup c'$ satisfied the desired condition on the genera of connected components.  Now, assuming that both connected components of $S' \cut (c \cup c')$ have positive genus, we can rewrite $T_{c,c'}$ as a product of bounding pair maps $T_{c_0, c_1} T_{c_1,c_2} \ldots T_{c_{n-1}, c_n}$ such that:
\begin{itemize}
\item $c_0 = c$, $c_n = c'$, and
\item at least one connected component of $S' \cut (c_i \cup c_{i+1})$ has genus one.
\end{itemize}
Hence we may assume without loss of generality that $c \cup c'$ is supported on a surface of genus one with two boundary components.  Then there is a surface $\Tau$ that contains $c \cup c'$ and is disjoint from $\widehat{x}$ such that $\Tau \cong S_2^1$, as desired.

\p{Case 2: $[c] = [a]$} We will reduce to Case 1, by showing that $[T_{c,c'}]\in H_1(\cI(S', S_g);\Q)$ is a sum $[T_{d,d'}] + [T_{e,e'}]$ with $d \cup d'$ and $e \cup e'$ bounding pairs such that $[d], [e] \neq [a]$.  As in Case 1, we may assume that the connected component of $S' \cut (c \cup c')$ that does not contain $\widehat{x}$ has genus one.  Then by Lemma~\ref{firsthomolratimlemma}, we have $H_1(\cI(S', S_g);\Q) \cong \midwedge^3 H_1(S';\Q)$.  By Lemma~\ref{bpjohnsoncomp}, if the connected component $S''$ of $S' \cut (c \cup c')$ that does not contain $\widehat{x}$ has $a,b \in H_1(S'';\Z)$ a pair of primitive elements with $\langle a, b \rangle = 1$, then the image of $T_{c,c'}$ in $\midwedge^3 H_1(S';\Q)$ is $[a] \wedge [b] \wedge [c]$.  Now, choose nonzero primitive $[d], [e] \in H_1(S';\Z)$ such that $[d] + [e] = [c]$ and $[d], [e] \in [a]^{\perp} \cap [b]^{\perp}$.  Then we take $T_{d,d'}$ and $T_{e,e'}$ bounding pair maps in $\cI(S', S_g)$ such that $[T_{d,d'}] = [a] \wedge [b] \wedge [d]$ and $[T_{e,e'}] = [a] \wedge [b] \wedge [e]$.
\end{proof}

\bn We are now ready to conclude Section~\ref{11stabgensection}.

\begin{proof}[Proof of Lemma~\ref{11stabgenlemma}]
Let $(x,f) \in \bigoplus_{x \in X_g^{(1)}} A_x$ be a class, where $x \subseteq X_g$ is an edge and $f \in A_x$.  Let $\widehat{x}$ be a lift of $x$ to $\cC_{[a]}(S_g)$ such that $a$ is a vertex of $\widehat{x}$.  We will show that the image of $(x,f)$ in $\bE_{1,1}^2$ is contained in $\im(\xi)$.  By Lemma~\ref{11stabedgefixlemma}, it suffices to prove the result in the case that $f$ is represented by a bounding pair map $F \in \Stab_{\cI_g}(\widehat{x})$ and that there is an inclusion $\iota: S_2^1 \hookrightarrow S_g$ such that $\widehat{x}$ is disjoint from $\im(\iota)$ and $F$ is supported on $\im(\iota)$.  Let $V = \im(\iota_*)$.  By Lemma~\ref{genus2homolcycleh1genlemma}, there is a linear combination in $H_1(X_g^V;\Q)$ given by 
\begin{displaymath}
[x] = \sum_{i = 0}^n \lambda_i [x_i]
\end{displaymath}
\bn such that each $x_i$ is contained in $X_g^{\calV'} \cap X_g^{V}$ for some $\calV' \subseteq \calV$ with $\left| \calV' \right| = 4$.  Since each $x_i \subseteq X_g^V$, we have $f \in A_{x_i}$ for each $x_i$, and thus there is a linear combination in $\bE_{1,1}^2$ given by
\begin{displaymath}
(x,f) = \sum_{i=0}^n (x_i,f).
\end{displaymath}
\bn Hence it is enough to prove the result in the case that $x \in X_g^{\calV'}$ for some $\calV' \subseteq \calV$ with $\left|\calV' \right| = 4$.  Let $\calH(x) = \{\calH_0^x, \calH_1^x \}$ and, without loss of generality, assume $V \subseteq \calH_0^x$.  Now, if $\calV' \not \subseteq \calH_0^x$ for $i = \{0,1\}$, we have $(x,f) \in \bE_{1,1}^{2, v}$ for any $v \in \calV'$.  Otherwise, the vector space $\calH_0^x \otimes \Q$ and the set $\calV'$ satisfy the hypotheses of Lemma~\ref{unimodwedgethreelemma}.  Hence there is a linear combination
\begin{displaymath}
f = \sum_{j=1}^m \lambda_j f_j
\end{displaymath}
\bn where each $f_j \in \midwedge^3 (\calH_0^x \cap v_i^{\perp}) \otimes \Q$ for some $v_i \in \calV'$.  Hence in $\bE_{1,1}^2$, we have
\begin{displaymath}
(x,f) = \sum_{j=1}^m \lambda_j (x, f_j)
\end{displaymath}
\bn with each $(x,f_j) \in \bE_{1,1}^{2,v}$ for some $v \in \calV' \subseteq \calV$.  Therefore we have $(x,f) \in \im(\xi)$ as desired.
\end{proof}

\subsection{The proof of Theorem~\ref{fincokthm}}\label{fincoksubsection}

\bn We begin by showing that $\bE_{1,1}^2$ is finite dimensional.

\begin{proof}[Proof of Proposition~\ref{11finprop}]
Let $G = \im(\Mod(S_g \cut a) \rightarrow \Sp(2g,\Z))$.  We will show that the hypotheses of Proposition \ref{gengrpprop} are satisfied for the $G$--representation $\bE_{1,1}^2$ with $d = 9$.  The first hypothesis is exactly Lemma \ref{11stabgenlemma} and the second is exactly Lemma \ref{11finquotlemma}, so $\bE_{1,1}^2$ is finite dimensional by Proposition \ref{gengrpprop}.
\end{proof}
\bn We are now ready to complete Section~\ref{fincoksection}.

\begin{proof}[Proof of Theorem~\ref{fincokthm}]
Let $g \geq \finalbound$ and $\bE_{*,*}^*$ denote the equivariant homology spectral sequence for the action of $\cI_g$ on $C_{[a]}(S_g)$.  By a result of the author~\cite[Theorem A]{Minahanhomolconn}, $\bE_{p,q}^r$ converges to $H_2(\cI_g;\Q)$ for $p + q = 2$.  Hence it suffices to show that $\bE_{2,0}^2$ and $\bE_{1,1}^2$ are finite dimensional.  The vector space $\bE_{2,0}^2$ is finite dimensional by Proposition~\ref{homolcurvequotprop} and $\bE_{1,1}^2$ is finite dimensional by Proposition~\ref{11finprop}.
\end{proof}

\section{The proof of Theorem~\ref{mainthm}} \label{mainpfsection}

We now prove Theorem~\ref{mainthm}.  We will begin by proving Lemma~\ref{KassPutaltlemma}, which is an alternate proof of a corollary of a result of Kassabov and Putman result~\cite[Theorem A]{KassabovPutman}.  We will then use Theorem~\ref{fincokthm} along with Proposition~\ref{gengrpprop} to prove Theorem~\ref{mainthm}.  

\begin{lemma}\label{KassPutaltlemma}
Let $g \geq 3$.  The vector space $H_0(\Sp(2g,\Z); H_2(\cI_g;\Q))$ is finite dimensional.
\end{lemma}

\begin{proof}
Let $\bE_{p,q}^r$ be the Leray--Serre spectral sequence associated to the short exact sequence
\begin{displaymath}
1 \rightarrow \cI_g \rightarrow \Mod(S_g) \rightarrow \Sp(2g,\Z) \rightarrow 1.
\end{displaymath}
\bn  The vector space $H_2(\Mod(S_g);\Q)$ is finite dimensional~\cite[Section 5.4]{FarbMarg}.  Then since $H_1(\cI_g;\Q)$ is finite dimensional~\cite{JohnsonI}, the modules $\bE_{p,q}^2 = H_p(\Sp(2g,\Z); H_q(\cI_g;\Q))$ are all finite dimensional for $q \leq 1$~\cite[Corollary 3]{Ragarithmetic}.  Hence both the image and the kernel of the pushforward map
\begin{displaymath}
H_0(\Sp(2g,\Z);H_2(\cI_g;\Q)) \rightarrow H_2(\Mod(S_g);\Q)
\end{displaymath}
\bn are finite dimensional, so $H_0(\Sp(2g,\Z); H_2(\cI_g;\Q))$ is finite dimensional.
\end{proof}

\bn We now prove the main result of the paper.

\begin{proof}[Proof of Theorem~\ref{mainthm}] Let $G = \Sp(2g,\Z)$ and $V = H_2(\cI_g;\Q)$.  We will show that the $G$--representation $V$ satisfies the hypotheses of Proposition \ref{gengrpprop} for $d = 1$, and hence is finite dimensional.  The first hypothesis is the content of Theorem \ref{fincokthm}, while the second is the content of Lemma \ref{KassPutaltlemma}.
\end{proof} 

\bibliographystyle{plain}
\bibliography{mainbib}

\end{document}